\documentclass[12pt]{article}
\input epsf.tex

\usepackage{graphicx}
\usepackage{amsmath,amsthm,amsfonts,amscd,amssymb,comment,eucal,latexsym,mathrsfs}
\usepackage{stmaryrd}
\usepackage[all]{xy}

\usepackage{epsfig}

\usepackage[all]{xy}
\xyoption{poly}
\usepackage{fancyhdr}
\usepackage{wrapfig}
\usepackage{epsfig}

\theoremstyle{plain}
\newtheorem{thm}{Theorem}[section]
\newtheorem{prop}[thm]{Proposition}
\newtheorem{lem}[thm]{Lemma}
\newtheorem{cor}[thm]{Corollary}
\newtheorem{conj}{Conjecture}

\theoremstyle{definition}
\newtheorem{defn}{Definition}
\theoremstyle{remark}
\newtheorem{remark}{Remark}
\newtheorem{example}{Example}

\newtheorem{exercise}{Exercise}
\newtheorem{question}{Question}
\newtheorem{problem}{Problem}

\newtheorem{notation}{Notation}

\advance \topmargin by -\headheight
\advance \topmargin by -\headsep
\evensidemargin \oddsidemargin
  \def\C{{\mathbb{C}}}  \def\E{{\mathbb{E}}} \def\F{{\mathbb{F}}}     \def\K{{\mathbb{K}}}   \def\N{{\mathbb{N}}}  \def\P{{\mathbb{P}}}  \def\R{{\mathbb{R}}}  \def\T{{\mathbb{T}}}      \def\Z{{\mathbb{Z}}}

                         \def\bZ{{\bar{Z}}}

   \def\bd{{\bar{d}}}            \def\bp{{\bar{p}}}          

                       \def\bfX{{\bf{X}}} \def\bfY{{\bf{Y}}} \def\bfZ{{\bf{Z}}}

                       \def\bfx{{\bf{x}}} \def\bfy{{\bf{y}}} 

\def\cA{{\mathcal{A}}} \def\cB{{\mathcal{B}}} \def\cC{{\mathcal{C}}}   \def\cF{{\mathcal{F}}}  \def\cH{{\mathcal{H}}} \def\cI{{\mathcal{I}}}   \def\cL{{\mathcal{L}}}  \def\cN{{\mathcal{N}}} \def\cO{{\mathcal{O}}} \def\cP{{\mathcal{P}}} \def\cQ{{\mathcal{Q}}} \def\cR{{\mathcal{R}}} \def\cS{{\mathcal{S}}}  \def\cU{{\mathcal{U}}} \def\cV{{\mathcal{V}}}  \def\cX{{\mathcal{X}}} \def\cY{{\mathcal{Y}}} \def\cZ{{\mathcal{Z}}}

                     \def\sV{{\mathscr{V}}}    

                   \def\tT{{\tilde{T}}}    \def\tX{{\tilde{X}}}  

     \def\tf{{\tilde{f}}}                    

 \def\tmu{{\widetilde{\mu}}}    \def\tphi{{\widetilde{\phi}}}         

\newcommand{\G}{\Gamma}
\newcommand{\Ga}{\Gamma}
\newcommand{\La}{\Lambda}

\newcommand{\Si}{\Sigma}
\newcommand{\eps}{\epsilon}

\renewcommand\a{\alpha}
\renewcommand\b{\beta}
\renewcommand\d{\delta}
\newcommand\g{\gamma}
\renewcommand\k{\kappa}
\renewcommand\l{\lambda}
\newcommand\s{\sigma}

\newcommand\Aut{\operatorname{Aut}}

\newcommand\Cay{\operatorname{Cay}}

\newcommand\Fix{{\operatorname{Fix}}}
\newcommand\Hom{\operatorname{Hom}}

\newcommand\Law{\operatorname{Law}}

\newcommand\Map{{\operatorname{Map}}}

\newcommand\Part{\operatorname{Part}}

\newcommand\Prob{\operatorname{Prob}}

\newcommand\sym{\operatorname{Sym}}
\newcommand\Stab{\operatorname{Stab}}

\newcommand\Span{\operatorname{span}}

\def\cc{{\curvearrowright}}
\newcommand{\resto}{\upharpoonright}

\begin{document}
\title{Examples in the entropy theory of countable group actions}
\author{Lewis Bowen\footnote{supported in part by NSF grant DMS-0968762, NSF CAREER Award DMS-0954606 and BSF grant 2008274} \\ University of Texas at Austin}
\maketitle

\begin{abstract}

Kolmogorov-Sinai entropy is an invariant of measure-preserving actions of the group of integers that is central to classification theory. There are two recently developed invariants, sofic entropy and Rokhlin entropy, that generalize classical entropy to actions of countable groups. These new theories have counterintuitive properties such as factor maps that increase entropy. This survey article focusses on examples, many of which have not appeared before, that highlight the differences and similarities with classical theory. 
\end{abstract}


\noindent
\tableofcontents

\section{Introduction}

Subsections \ref{sec:classical}-\ref{sec:OW-intro} of this introduction set notation, give a brief review of classical entropy theory and motivate entropy theory for actions of general countable groups. Subsections \ref{sec:f-intro}- \ref{sec:intro-rokhlin} provide an intuitive approach to the $f$-invariant, sofic groups, sofic and Rokhlin entropy. The last subsection \ref{sec:what} summarizes the contents of this article.

\subsection{Classical entropy theory}\label{sec:classical}

To set notation, let $(X,\mu),(Y,\nu)$ denote standard probability spaces. An {\bf automorphism} of $(X,\mu)$ is a measurable map $T:X \to X$ with measurable inverse that preserves the measure $\mu$. Two such maps $T:X \to X$, $S:Y \to Y$ are {\bf measurably conjugate} or {\bf isomorphic} if there exists a measure-space isomorphism $\Phi:X \to Y$ such that $\Phi \circ T = S \circ \Phi$ almost everywhere. The main motivating problem of this article is to classify automorphisms (and more generally, group actions) up to measure-conjugacy.

A special type of automorphism, called a Bernoulli shift, plays a central role. To define it, let $K$ denote a standard Borel space and $K^\Z$ the infinite direct power. An element $x\in K^\Z$ is a sequence $x=\{x_n\}_{n\in \Z}$ with values $x_n \in K$.  Let $\sigma:K^\Z \to K^\Z$ denote the {\bf shift map} defined by $\sigma(x)_n = x_{n-1}$. If $\kappa$ is a probability measure on $K$ then the shift map preserves the product measure $\kappa^\Z$. The triple $(\sigma,K^\Z,\kappa^\Z)$ is called the {\bf Bernoulli shift over the integers with base space $(K,\kappa)$}. In the early days of ergodic theory, von Neumann asked for a classification of these Bernoulli shifts. At the time, it was known that all Bernoulli shifts are spectrally isomorphic (that is, the induced operators on $L^2(K^\Z,\kappa^\Z)$ are unitarily isomorphic). However not a single pair of Bernoulli shifts was known to be non-isomorphic and there were no nontrivial tools for proving isomorphism.

Motivated by this problem, Kolmogorov introduced dynamical entropy theory  \cite{kolmogorov-1958, kolmogorov-1959}. Here is an intuitive explanation: suppose that a system is under observation. At each unit of time, a measurement is made and recorded. The measuring device can only take on a finite number of distinct values. The entropy of the system is the amount of new information gained per unit time, on average and in the long run.

In this interpretation, $(X,\mu)$ represents all possible states of the system and $T:X \to X$ represents time evolution. The measuring device is represented by a measurable partition $\cP$ of $X$. The {\bf Shannon entropy} of $\cP$ is defined by
$$H_\mu(\cP) := \sum_{P\in \cP} \mu(P) I_\mu(P) = - \sum_{P\in \cP} \mu(P) \log(\mu(P)).$$
To motivate the above, suppose $x\in X$ is random. The amount of information gained by learning which part $P$ of $\cP$ contains $x$ is defined by $I_\mu(P):=-\log \mu(P)$. So $H_\mu(\cP)$ is the average amount of information gained from learning which part of $\cP$ contains $x$. (The definition of $I_\mu(P)$ is chosen so that if $P,Q \subset X$ are independent events then $I_\mu(P\cap Q)=I_\mu(P)+I_\mu(Q)$). 

The coarsest common refinement of two partitions $\cP,\cQ$ is denoted $\cP \vee \cQ$. The {\bf entropy rate of $T$ with respect to $\cP$} is
$$h_\mu(T,\cP) := \lim_{n\to\infty} \frac{1}{n} H_\mu\left( \bigvee_{i=0}^n T^{-i}\cP \right).$$
This is average quantity of information gained per unit time (represented by $T$) when observing the itinerary of a $\mu$-random point $x$ through the partition $\cP$. 

The {\bf entropy rate of $T$} is defined by
$$h_\mu(T):=\sup_\cP h_\mu(T,\cP).$$
A partition $\cP$ is {\bf generating} if the smallest sigma-algebra containing $T^{-n}\cP$ for all $n\in \Z$ is the sigma-algebra of all measurable sets, up to sets of measure zero. Kolmogorov proved the crucially important result that if $\cP$ is any generating partition with finite Shannon entropy then the entropy rate of $T$ is $h_\mu(T,\cP)$. Therefore, to compute entropy of $T$ one can choose any convenient generating partition. In the special case of the Bernoulli shift with base $(K,\kappa)$, if $K$ is countable then the {\bf time $0$ partition} $\cP=\{P_k:~k\in K\}$ defined by $P_k=\{ x\in K^\Z:~x_0=k\}$ is generating. The entropy rate of the Bernoulli shift coincides with the Shannon entropy of $\cP$. The latter is also called the {\bf Shannon entropy of $(K,\kappa)$}:
$$H(K,\kappa):= - \sum_{k\in K} \kappa(\{k\}) \log \kappa (\{k\})$$
if $\kappa$ is purely atomic and $H(K,\kappa):=+\infty$ otherwise. This shows that Bernoulli shifts with different base space entropies are not measurably conjugate.

In 1970, Ornstein proved the converse: two Bernoulli shifts with the same entropy are isomorphic. Moreover, he developed a deep set of tools for determining whether a given automorphism is Bernoulli. With these tools, he and co-authors proved that many automorphisms are isomorphic to Bernoulli shifts including mixing Markov chains, hyperbolic toral automorphisms, the time 1 map of geodesic flow on a hyperbolic surface and more. 

In 1964, Sinai proved that every ergodic automorphism is a zero-entropy extension of a Bernoulli shift (which may be trivial). This explains why Bernoulli shifts are so important to the classification of automorphisms in general.

\subsection{General groups and naive entropy}\label{sec:general}

Now let $\G$ be a countable group. An {\bf action} of $\G$ on a set $X$ is a collection $T=(T^g)_{g\in \G}$ of transformations $T^g:X \to X$ satisfying $T^{gh}=T^gT^h, T^{g^{-1}}=(T^g)^{-1}$. For convenience we may write $gx$ for $T^gx$ when there is only one action of $\G$ on $X$ under consideration. An action $T$ on a probability space $(X,\mu)$ is 
{\bf probability-measure-preserving} ({\bf pmp}) if each $T^g$ preserves $\mu$. We also denote an action by $\G \cc^T (X,\mu)$ or simply $\G \cc (X,\mu)$ when $T$ is implicit. A pmp action $T$ of $\G$ on $(X,\mu)$  {\bf factors onto} an action $S$ of $\G$ on $(Y,\nu)$ if there is a measurable map $\Phi:X \to Y$ that pushes $\mu$ forward to $\nu$ and intertwines the action (so $\Phi  T^g =S^g\Phi$ up to measure zero). Such a map is called a {\bf factor map}. If, in addition, $\Phi$ is invertible with measurable inverse then the two actions are said to be  {\bf measurably conjugate} or {\bf isomorphic}. The main motivating problem is to classify actions up to measure-conjugacy and determine which actions factor onto which.

The Bernoulli shifts described above generalize to this context. To be precise, let $K$ be a standard Borel space and $K^\G$ be the set of all functions $x:\G \to K$. When convenient we may represent $x\in K^\G$  as a collection $x=(x_g)_{g\in \G}$ of elements $x_g\in K$. The {\bf shift action} of $\G$ is denoted $S=(S^g)_{g\in \G}$ where $S^g:K^\G \to K^\G$ is the transformation
$$(S^gx)_h = x_{g^{-1}h}.$$
If $\k$ is a probability measure on $K$ then $S$ preserves the product measure $\kappa^\Ga$. The action $\G \cc (K,\k)^\G$ is called the {\bf Bernoulli shift over $\Ga$ with base space $(K,\kappa)$}. 

There is a large difference between the entropy theory of amenable group actions and that of non-amenable groups. So it is worthwhile to review the definitions. A countable group $\G$ is {\bf amenable} if there exists a sequence $\{F_n\}$ of non-empty finite subsets of $\G$ satisfying 
$$\lim_{n\to\infty} \frac{ |KF_n  \cap F_n| }{|F_n|} = 1$$
for every non-empty finite $K \subset \Ga$. Such a sequence is called a {\bf F\o lner sequence}. There are many other equivalent definitions of amenability  \cite{MR3616077, bekka2008kazhdan}. For example, abelian, nilpotent and solvable groups are amenable while non-abelian free groups, $SL(n,\Z)$ ($n \ge 2$), mapping class groups (with a few exceptions) and fundamental groups of closed hyperbolic $n$-manifolds $(n\ge 2)$ are not.

Now suppose $\G$ is amenable with F\o lner sequence $\{F_n\}$ and $T$ is a pmp action of $\G$ on $(X,\mu)$.  The standard definition of entropy is:
\begin{eqnarray*}
h_\mu(T, \cP) &=& \lim_{n\to\infty} |F_n|^{-1}H_\mu\left( \bigvee_{f\in F_n} T^{f^{-1}}\cP\right)\\
h_\mu(T) &=& \sup_\cP h_\mu(T, \cP)
\end{eqnarray*}
where $\cP$ is an arbitrary countable measurable partition of $X$ with finite Shannon entropy.

By a sub-additivity argument it can be shown that the limit defining $h_\mu(T, \cP)$ exists, does not depend on the choice of F\o lner sequence and moreover, if $\cP$ is a {\bf generating} partition (this means that the smallest $\Ga$-invariant sigma-algebra containing $\cP$ is the sigma-algebra of all measurable sets up to measure zero) with finite Shannon entropy then $h_\mu(T, \cP) = h_\mu(T)$. In particular, the entropy of the Bernoulli shift action $\G \cc (K,\k)^\G$ equals the Shannon entropy of the base space $H(K,\k)$. Entropy for amenable groups was first considered in \cite{kieffer-1975a}. See also \cite{ollagnier-book, OW80}. 

Moreover the above entropy coincides with the so-called {\bf naive entropy} defined by:
$$h^{\rm{naive}}_\mu(T,\cP):=\inf_{F \Subset \G} |F|^{-1} H_\mu\left( \bigvee_{f\in F} T^{f^{-1}}\cP\right)$$
$$h^{\rm{naive}}_\mu(T):=\sup_\cP h^{\rm{naive}}_\mu(T,\cP)$$
where $F \Subset \G$ means $F$ is a non-empty {\em finite} subset of $\G$. This definition makes sense for arbitrary countable groups $\G$. However if $\G$ is non-amenable then Theorem \ref{thm:naive-non-amenable} below shows that $h^{\rm{naive}}_\mu(T) \in \{0,\infty\}$ so naive entropy cannot distinguish Bernoulli shifts in this case.

\subsection{The Ornstein-Weiss example}\label{sec:OW-intro}

The next example convinced many researchers that entropy theory could not be extended to non-amenable groups. To explain, first suppose $\Ga$ is amenable. Consider the {\bf full $n$-shift} over $\Ga$; this is the Bernoulli shift with base space $(\Z/n, u_n)$ where $\Z/n$ is the cyclic group of order $n$ and $u_n$ is the uniform probability measure.  The entropy of the full $n$-shift is $\log(n)$. Because entropy is non-increasing under factor maps, the full 2-shift cannot factor onto the full 4-shift. However, Ornstein and Weiss showed in \cite{OW87} that when ${\F_2}=\langle a,b\rangle$ is the rank 2 free group, the full 2-shift does indeed factor onto the full 4-shift. Their example is as follows: define $\phi:(\Z/2)^{\F_2} \to (\Z/2\times \Z/2)^{\F_2}$ by 
$$\phi(x)_g = (x_g + x_{ga}, x_g + x_{gb}).$$
The spaces $(\Z/2)^{\F_2}$ and $(\Z/2\times \Z/2)^{\F_2}$ are compact abelian groups under pointwise addition. With this interpretation, $\phi$ is a group homomorphism. It is a good exercise to show that it is surjective. Surjectivity implies that $\phi$ takes Haar measure to Haar measure. Therefore, it is indeed a factor map. In fact the kernel consists of the constants, so it is a 2-1 factor map. In the setting of $\Z$-actions, entropy is preserved under finite-to-1 maps. So this example led some researchers to speculate that the 2-shift and the 4-shift over $\F_2$ could be isomorphic. The $f$-invariant (and later, sofic entropy) was developed to prove that they are not. 

In work in progress \cite{bowen-ri}, the author has shown that if $\G$ is any non-amenable group then all Bernoulli shifts over $\G$ factor onto each other. This is based on a generalization of the Gaboriau-Lyons Theorem \cite{gaboriau-lyons} and is reviewed in \S \ref{sec:bernoulli-bernoulli}. So there does not exist a monotone invariant that distinguishes Bernoulli shifts over a non-amenable group.


\subsection{The $f$-invariant}\label{sec:f-intro}


The $f$-invariant is a measure-conjugacy invariant  for actions of free groups that distinguishes Bernoulli shifts. To explain, it will be convenient to use probabilistic notation as follows. Suppose $X$ is a random variable that takes values in a Borel space $K$. Let $\Prob(K)$ denote the set of all Borel probability measures on $K$. The {\bf law} of $X$ is the probability measure $\Law(X) \in \Prob(K)$ satisfying
$$\Law(X)(E) = P(X \in E)$$
for $E \subset K$ where $P(\cdot)$ denotes probability. If $Y$ is also a random variable taking values in a Borel space $L$, then $\Law(X,Y) \in \Prob(K\times L)$ is the probability measure satisfying
$$\Law(X,Y)(E\times F) = P(X \in E, Y \in F).$$
Also $\Law(X|Y)$ is the random variable taking values in $\Prob(K)$ defined by
$$ \Law(X|Y)(E)  = P(X \in E|Y).$$

In general, a {\bf stationary $\Ga$-process}  is a $\Ga$-indexed family of random variables $\bfX=(X_g)_{g\in \Ga}$ such that each $X_g$ takes values in some Borel space $K$ and the law of $\bfX$ is invariant under left-multiplication of the indices. So the law of $\bfX$ is a $\G$-invariant Borel probability measure on $K^\G$ and stationarity means that the law of $(X_g)_{g\in \G}$ is the same as the law of $(X_{hg})_{g\in \G}$ for any $h \in \G$. 

The $f$-invariant is motivated by way of Markov chains over free groups. Before getting into that, it makes sense to recall Markov chains over the integers. A stationary $\Z$-process $\bfX=(X_i)_{i \in \Z}$ is {\bf Markov} if
$$\Law(X_0 | X_{-1}, X_{-2}, \ldots) = \Law(X_0 | X_{-1}).$$
The law of a stationary Markov process is uniquely determined by the {\bf time-0 distribution} $\Law(X_0)$ and the transition probabilities $P(X_0=k_0|X_{-1}=k_{-1})$ (for $k_{-1},k_0 \in K$). Moreover, the Markov process uniquely maximizes entropy over all stationary processes that have the same time-0 distribution and transition probabilities.

If $\bfX=(X_i)_{i\in \Z}$ is an arbitrary stationary process with values in a finite or countable set $K$ then its {\bf $n$-th step Markov approximation} is the Markov process $\bfY^{(n)}=(Y^{(n)}_i)_{i\in \Z}$ taking values in the Cartesian power $K^n$ satisfying
$$\Law(Y^{(n)}_0) = \Law(X_0,\ldots, X_{n-1})$$
$$\Law\left(Y^{(n)}_0\Big|Y^{(n)}_{-1} = (x_{-1},x_0, \ldots, x_{n-2})\right) = \Law\big(X_0,\ldots, X_{n-1}| (X_{-1},\ldots, X_{n-2}) = (x_{-1},x_0, \ldots, x_{n-2})\big)$$
for any $x_{-1},\ldots, x_{n-2} \in K$. The entropy rate of $\bfX$ satisfies 
$$h(\bfX) = \lim_{n\to\infty} h(\bfY^{(n)}) = \lim_{n\to\infty} H\left(Y^{(n)}_0|Y^{(n)}_{-1}\right).$$
(Recall that the Shannon entropy of a random variable $X$ conditioned on another random variable $Y$ that takes on only countably many values is
$$H(X|Y):= - \sum_{x,y} P(X=x, Y=y) \log( P(X=x|Y=y) ).$$
Since 
$$H(Y^{(n)}_0|Y^{(n)}_{-1}) = H(X_0,\ldots, X_{n-1}|X_{-1},\ldots, X_{n-2}) = H(X_{n-1}|X_{n-2},\ldots, X_{-1})= H(X_0|X_{-1},\ldots, X_{-n})$$
we arrive at the familiar formula
$$h(\bfX) = H(X_0|X_{-1},X_{-2},\ldots).$$

This generalizes to free groups as follows: let ${\F_r}=\langle s_1,\ldots, s_r\rangle$ be a free group of rank $r$ and $\bfX=(X_g)_{g\in {\F_r}}$ a stationary process over ${\F_r}$. The process $\bfX$ is {\bf Markov} if for every $1\le i \le r$
$$\Law(X_e | (X_g)_{g\in \rm{Pre(s_i)}} ) = \Law( X_e | X_{s_i} )$$
where $\rm{Pre(s_i)} \subset {\F_r}$ is the set of all elements with prefix $s_i$. In other words, $g \in \rm{Pre(s_i)}$ if and only if $|s_i^{-1}g| < |g|$ where $|\cdot |: \F_r \to \N$ denote the word length (so $|\g|$ is the smallest natural number $n$ such that  $\g$ can be written as a product of $n$ elements of $\{s_1,\dots, s_r, s_1^{-1},\ldots, s_r^{-1}\}$).

Define 
$$F( \bfX) :=- (r-1)H(X_e) + \sum_{i=1}^r H(X_e| X_{s_i}).$$
The intuition for this formula is as follows: the term $H(X_e|X_{s_i})$ measures the entropy in the $s_i$-direction. The sum  $\sum_{i=1}^r H(X_e| X_{s_i})$  ``counts'' the entropy at the identity $r$ times. To compensate for this, substract $(r-1)H(X_e)$ to obtain the formula above. 

Now suppose $\bfX=(X_g)_{g\in {\F_r}}$ is an arbitrary stationary ${\F_r}$-process taking values in a finite or countable set $K$. The {\bf $n$-th Markov approximation} to $\bfX$ is the Markov process $\bfY^{(n)}=(Y^{(n)}_g)_{g\in {\F_r}}$ taking values in $K^{B(n)}$ (where $B(n) \subset {\F_r}$ is the ball of radius $n$) satisfying
$$\Law(Y^{(n)}_e) = \Law( (X_g)_{g\in B(n)} ), ~\Law( Y^{(n)}_e | Y^{(n)}_{s_i} = x) = \Law( (X_g)_{g\in B(n)} | (X_{s_ig})_{g\in B(n)} = x)$$
for any function $x:B(n) \to K$. Define
$$f(\bfX):=\lim_{n\to\infty} F( \bfY^{(n)}).$$
In \S \ref{sec:f} a proof is sketched that this does indeed define a measure-conjugacy invariant, called the $f$-invariant. Moreover, the $f$-invariant of the Bernoulli shift ${\F_r} \cc (K^\Ga,\kappa^\Ga)$ is $H(K,\kappa)$. This proves the 2-shift over $\F_r$ is not isomorphic to the 4-shift. 

The $f$-invariant is a particularly nice invariant: it can be computed exactly for Markov chains (\S \ref{sec:markov-chains}), it satisfies an ergodic decomposition formula (\S \ref{sec:ergodic decomposition}), a subgroup formula (\S \ref{sec:subgroups}), an Abramov-Rokhlin type formula (\S \ref{sec:abramov}), a Yuzvinskii-type addition formula (\S \ref{sec:yuz}) and is additive under direct products. However, it can increase under factor maps and it can take on negative values (\S \ref{sec:f-trivial}). It can be generalized to some other groups (\S \ref{sec:other}) and is related to sofic entropy (\S \ref{sec:interpretation}).

\subsection{Sofic groups via Benjamini-Schramm convergence}\label{sec:sofic-bs}

The sofic concept provides a new perspective on the $f$-invariant and extends entropy theory beyond amenable groups. The most intuitive definition of sofic groups is based on Benjamini-Schramm convergence \cite{benjamini-schramm-2001}. Only what is needed for sofic group theory will be explained here; for the more general theory see \cite{lovasz-local-global}.

Suppose $\Ga$ has a finite symmetric generating set $S \subset \Ga$. The {\bf Cayley graph} of $(\Ga,S)$ is a directed graph with edge labels in $S$, denoted by $\Cay(\Gamma,S)$. Its vertex set is $\Ga$ and for every $s\in S$ and $g\in \Ga$ there is an $s$-labeled directed edge from $g$ to $gs$.  These are all of the edges. Because the edges of the Cayley graph are directed and labeled, the group can be recovered from the Cayley graph. This would not be true otherwise since there are groups with Cayley graphs that are isomorphic as unlabeled graphs.

Now consider a {\em finite} $S$-edge-labeled directed graph $G=(V,E)$. For $r>0$, let $\sV_r(G)$ be the set of all vertices $v\in V$ such that there exists a graph isomorphism from the ball of radius $r$ centered at $v$ to the ball of radius $r$ centered at the identity element in the Cayley graph $\Cay(\Gamma,S)$. This  isomorphism is required to map $v$ to the identity, preserves edge directions and preserves labels. The graph $G$ is called an  {\bf $(r,\epsilon)$-sofic approximation} to $(\Ga,S)$ if $|\sV_r(G)|\ge (1-\epsilon)|V|$. So with probability $\ge 1-\epsilon$, a uniformly random vertex's radius $r$-neighborhood looks the same as the radius $r$-neighborhood of the identity in the Cayley graph $\Cay(\Ga,S)$. 

By definition, a sequence $\{ G_i \}$ of finite $S$-edge-labeled directed graphs {\bf Benjamini-Schramm converges} to the Cayley graph $\Cay(\Ga,S)$ if 
$$\lim_{i\to\infty} \frac{|\sV_r(G_i)|}{|V_i|}=1$$ 
for every $r>0$. Such a sequence is called a {\bf sofic approximation} to $\Ga$. The group $\Ga$ is {\bf sofic} if there exists a sofic approximation to $\Ga$.  

\begin{exercise}
Show that soficity does not depend on the choice of generating set $S$.
\end{exercise}

\begin{exercise}
Show that $\Z^d$ and finitely generated free groups are sofic.
\end{exercise}

For example, if $C_n$ is the directed $n$-cycle then $C_n$ Benjamini-Schramm converges to the standard Cayley graph of $\Z$ as $n\to\infty$. For another example, suppose $\Gamma$ is a finitely generated residually finite group. Residual finiteness means there exists a decreasing sequence $N_i\vartriangleleft \Gamma$ of finite-index normal subgroups with $\cap_i N_i =\{e\}$. Let $S$ be a finite generating set for $\Gamma$ and consider the associated Cayley graphs $G_n$ with vertex set $\Gamma/N_n$ and edges $\{ (gN_n, gsN_n):~ s\in S\}$. This sequence Benjamini-Schramm converges to the Cayley graph of $\Gamma$ with respect to $S$. Therefore all residually finite groups are sofic. 

The above definition is the most intuitive.  However, it has the unfortunate drawback that it applies only to finitely generated groups. The section \S \ref{sec:sofic-sym} presents a more general definition based on maps from $\Ga$ into the symmetric group. 

It is a major open problem whether all countable groups are sofic. For further reading, there are several surveys on sofic groups \cite{pestov-sofic-survey, pestov-kwiatkowska, capraro-lupini}.

\subsection{Sofic entropy: a probabilistic approach}

A more thorough account of sofic entropy is presented in \S \ref{sec:pseudo-top}-\S \ref{sec:measure-sofic}. Here is an intuitive approach under simplifying conditions. Let $\bfX=(X_g)_{g\in \Ga}$ be a $\Ga$-stationary process taking values in a finite set $K$. Assume $\Ga$ has a finite generating set $S$.  Fix a sequence $\Sigma=\{G_i\}_{i=1}^\infty$ of finite $S$-labeled directed graphs $G_i=(V_i,E_i)$ that Benjamini-Schramm converge to the Cayley graph $\Cay(\G,S)$. The sofic entropy of $\bfX$ with respect to $\Sigma$, denoted $h_\Sigma(\bfX)$, is the exponential rate of growth of the number of microstates for $\bfX$ on $G_i$. Intuitively, a microstate is a function $\phi:V_i \to K$ that approximates $\bfX$ in a local statistical sense. To be precise, fix a radius $r>0$ and let vertex $v \in \cV_r(G_i)$ be a uniformly random vertex. Consider the restriction of $\phi$ to the ball of radius $r$ centered at $v$. Because $\Si$ is a sofic approximation, this ball is isomorphic to the ball $B_r(\G,S)$ of radius $r$ centered at the identity in $\Cay(\G,S)$ with high probability. So the law of the restriction $\phi \resto B_r(v)$ determines a (sub-)probability measure on the set $K^{B_r(\G,S)}$ of all functions from $B_r(\G,S)$ to $K$.  If this law is $\delta$-close in total variation distance to the law of $(X_g)_{g\in B_r(\G,S)}$ then  $\phi$ is said to be an {\bf $(r,\delta)$-microstate} of $\bfX$. 

The {\bf sofic entropy} of $\bfX$ is
$$h_\Sigma(\bfX) = \inf_{\delta>0} \inf_{r>0} \limsup_{i\to\infty} |V_i|^{-1} \log (\# (r,\delta)-\textrm{microstates on } G_i).$$
This is also called the {\bf $\Si$-entropy}. In \cite{bowen-jams-2010} it was shown that this entropy is invariant under measure-conjugacy and the entropy of an i.i.d. process is the Shannon entropy $H(X_{1_\G})$. 



A few words about the definition of $h_\Sigma(\bfX)$: if $\G$ is amenable then the sofic entropy agrees with classical entropy. However, if $\Ga$ is non-amenable then it is possible that there are no $(r,\delta)$-microstates for any graph $G_i$ in the sofic approximation. In this case, $h_\Sigma(\bfX)=-\infty$. Examples of such behaviour are presented in \S \ref{sec:trivial}. Also the $\limsup$ can be replaced with $\liminf$ or by an ultralimit. These changes give apriori different invariants (when $\G$ is non-amenable). In \S \ref{sec:trivial} and \S \ref{sec:dependence}  examples are presented of sofic approximations $\Sigma,\Sigma'$ to a group $\Ga$ and an explicit action such that the $\Si$-entropy is $-\infty$ but the $\Si'$ entropy is non-negative. The following is a major open problem:

\begin{problem}
Suppose $\Sigma,\Sigma'$ are two sofic approximations to $\G$ and the $\Si$-entropy and $\Si'$-entropy of $\bfX$ are both positive. Are they necessarily equal?
\end{problem} 

Other expositions of sofic entropy include \cite{MR3411529, MR3666024, MR3616077}.

The concept of sofic approximation can be generalized by replacing the finite graphs with random finite graphs. That is, the $i$-th approximating graph $G_i$ is allowed to be random; but the number of vertices $|V_i|$ is required to be determined. Now define
$$h_\Sigma(\bfX) = \inf_{\delta>0} \inf_{r>0} \limsup_{i\to\infty} |V_i|^{-1} \log \E[(\# (r,\delta)-\textrm{microstates on } G_i)].$$
In the special case in which $\Ga=\langle s_1,\ldots, s_r\rangle$ is a free group, this leads to a new interpretation of the $f$-invariant:
$$f(\bfX) = h_\Sigma(\bfX)$$
where $\Sigma = \{G_i\}_{i=1}^\infty$ is the ``permutation model'' of the random $2r$-regular graphs.  To be precise, let $\sigma:\Ga \to \sym(n)$ be a homomorphism chosen at random uniformly amongst all $(n!)^r$ homomorphisms where $\sym(n)$ is the symmetric group on $[n]:=\{1,\ldots, n\}$. Then $G_n$ is the graph with vertex $V_n=[n]$ and $s_i$-labeled edges $(p, \sigma(s_i)p)$ for $1\le i \le r$ and $p \in [n]$. This is Theorem \ref{thm:f} below. 

\subsection{Rokhlin entropy}\label{sec:intro-rokhlin}
Suppose $T=(T^g)_{g\in \G}$ is a pmp action of $\Ga$ on $(X,\mu)$ and $\cP$ is a generating partition for the action. It follows immediately that the sofic entropy of  $T$ is bounded by the Shannon entropy $H_\mu(\cP)$. This leads to the following idea: let $h^{\rm{Rok}}(T)$ denote the infimum of $H_\mu(\cP)$ over all generating partitions $\cP$. If $T$ is ergodic then this is called the {\bf Rokhlin entropy} of the action (the non-ergodic case is slightly different; see \S \ref{sec:Rokhlin} for details). Some basic facts:
\begin{itemize}
\item Rokhlin entropy is a measure-conjugacy invariant. Moreover, it is well-defined for every action of every countable group (even non-sofic ones, if they exist!).
\item Rokhlin entropy is an upper bound for sofic entropy.
\item Rokhlin entropy agrees with classical entropy for amenable groups \cite{seward-tucker-drob}.
\end{itemize}
Moreover, in recent groundbreaking work, the following results have been obtained:
\begin{itemize}
\item \cite{seward-tucker-drob} For every $\epsilon>0$, every essentially free ergodic action with positive Rokhlin entropy admits a factor that is essentially free and has Rokhlin entropy $<\epsilon$,
\item \cite{seward-kreiger-1} Every ergodic action with Rokhlin entropy $<\log(n)$ admits a generating partition with $n$ parts,
\item \cite{seward-sinai} Every essentially free ergodic action with positive Rokhlin entropy admits a Bernoulli factor.
\end{itemize}
See \S \ref{sec:Rokhlin} and \S \ref{sec:ornstein theory} for more details.

\begin{problem}
Suppose $T$ is ergodic and essentially free. If the $\Si$-entropy of $T$ is not $-\infty$ then is it equal to the Rokhlin entropy?
\end{problem}



\subsection{What's in this article?}\label{sec:what}

\S \ref{sec:preliminaries}  reviews the fundamental aspects of sofic groups, sofic entropy, the $f$-invariant, Rokhlin entropy and naive entropy. 

\S \ref{sec:special classes} covers a list of examples in which entropy has been computed. Perhaps the most interesting cases are the Bernoulli shifts, Markov chains over free groups and principal algebraic actions (in which the entropy is related to the Fuglede-Kadison determinant). There are also degenerate cases in which the entropy is non-positive.  This includes (under mild conditions) trivial actions, distal, smooth and non-free actions. One surprising case is that of lattices $\G,\La$ in a totally disconnected locally compact group $G$ such that $\G \cc G/\La$ has positive entropy (Example \ref{ex:tree} in \S \ref{sec:markov-chains}). 

About \S \ref{sec:perturbations} and \S \ref{sec:combinations}: much of the usefulness of classical entropy theory derives from a list of formulas and inequalities expressing how entropy changes as the system is perturbed or combined with other systems. This includes: inducing to a subgroup, co-inducing a subgroup action, continuity or semi-continuity in the measure, the partition or the action or passing to an orbit-equivalent action. It also includes direct products, ergodic decomposition, relative entropy and inverse limits. In the sofic case, we usually have an inequality where in the classical case an equality holds. Moreover, there are counterexamples. For example, in \S \ref{sec:direct products} we present an example showing that sofic entropy need not be additive under direct products.

  \S \ref{sec:OW}   presents a finite-to-1 factor map from a zero entropy action to a Bernoulli shift, \S \ref{sec:bernoulli-bernoulli}   sketches a proof that if $\G$ is non-amenable then all Bernoulli shifts factor onto each other, \S \ref{sec:zero-entropy} sketches a proof that if $\G$ is non-amenable then every free ergodic action has a zero-entropy extension, \S \ref{sec:finite-to-1} explores how entropy varies under finite-to-1 factor maps.

\S \ref{sec:ornstein theory} covers generalizations of Ornstein theory for non-amenable groups including the Isomorphism Theorem, Krieger's generator Theorem and Sinai's Factor theorem. It also contains counterexamples such as Popa's example of a non-Bernoulli factor of a Bernoulli shift, and a non-Bernoulli $\bd$-limit of Bernoulli shifts. 

\S \ref{sec:variational} sketches a proof of the variational principle for sofic entropy. This naturally leads to the question of whether measures of maximal entropy exist and whether or not they are unique. The existence problem is similar to that of the classical case: namely, existence occurs under weak forms of expansitivity that imply upper semi-continuity of entropy with respect to weak* topology on the space of measures. Regarding uniqueness: an example is presented in \S \ref{sec:uniqueness}  of a mixing Markov chain over the free group with multiple measures of maximum $f$-invariant. 

\S \ref{sec:GIBBS} defines sofic pressure and equilibrium states (measures) for actions of sofic groups and relates them to Gibbs measures on random regular graphs.

\S \ref{sec:relative} is a short section on relative entropy. This includes an Abramov-Rokhlin formula for actions of free groups.

\S \ref{sec:outer} defines and explores outer sofic and Rokhlin entropy. For example, outer sofic entropy of a factor map is the exponential rate of growth of the number of microstates for the target action that lift to microstates for the source action. When $\G$ is amenable, this is just the entropy of the target. However, when $\G$ is non-amenable it can be different; for example the outer sofic entropy of the Ornstein-Weiss map is $\log(2)$, not $\log(4)$. Using outer entropy, we define outer Pinsker algebra and completely positive outer entropy. For example, Bernoulli shifts and a large class of algebraic actions are known to have completely positive outer entropy. This notion is also related to uniform model mixing which is a generalization of uniform mixing to the sofic context.





{\bf Acknowledgements}. 
 I am most grateful for discussions with Tim Austin, Peter Burton, Ben Hayes, David Kerr, Hanfeng Li,  Sorin Popa, Brandon Seward, Jean-Paul Thouvenot, Robin Tucker-Drob and Benjy Weiss. Many of these researchers have contributed examples which appear in this article, some for the first time.

\section{Preliminaries}\label{sec:preliminaries}       

\subsection{Notation and conventions}

Throughout this article, all measure spaces are standard and all maps measurable unless otherwise specified. We often ignore measure zero phenomena without explicit mention. Also, $\Ga$ denotes a countable group, $(X,\mu),(Y,\nu)$ probability spaces and $\Ga \cc^T (X,\mu)$, $\Ga \cc^S (Y,\nu)$ are probability-measure-preserving (pmp) actions of $\Ga$ (unless otherwise specified). This means $T=(T^g)_{g\in \G}$ is a collection of measure-space automorphisms of $(X,\mu)$ such that $T^{gh}=T^gT^h$ and $T^{g^{-1}}=(T^g)^{-1}$ almost everywhere.

A {\bf factor map} between these actions is a measurable map $\Phi:X \to Y$ such that $\Phi_*\mu=\nu$ and $\Phi(gx)=g\Phi(x)$ for a.e. $x$ and every $g\in \Ga$. There is a natural correspondence between factors of the action $\Ga \cc (X,\mu)$ and $\Ga$-invariant sigma-algebras of $X$ (up to measure zero sets). Namely, if $\Phi:X \to Y$ is a factor map and $\cB_Y$ is the Borel sigma-algebra on $Y$ then $\Phi^{-1}(\cB_Y)$ is a $\Ga$-invariant sigma-algebra of $X$. We will call this the {\bf sigma-algebra associated with $\Phi$}. Conversely, if we are given a $\Ga$-invariant sigma-algebra $\cF \subset \cB_X$ (where $\cB_X$ is the Borel sigma-algebra on $X$) then by the Mackey realization Theorem there is a Borel space $Y$ and a factor map $\Phi:X \to Y$ such that $\cF=\Phi^{-1}(\cB_Y)$. 

If $\cP$ is a measurable partition of $X$ then the {\bf factor associated to $\cP$} is the factor associated to the smallest $\G$-invariant sigma-algebra containing $\cP$.

The notation $X \Subset Y$ means that $X$ is a finite subset of $Y$.






\subsection{Sofic groups}       

The next definition might be less intuitive than the definition of soficity in the introduction; however it is the most useful.

\subsubsection{Soficity via maps into symmetric groups}\label{sec:sofic-sym}       

Let $\Ga$ be a countable group. Throughout, $V$ denotes a finite set and  $\sym(V)$ the group of bijections from $V$ to itself. 

\begin{defn}\label{defn:sofic}
Given $F \subset G$ and $\delta>0$ we say that a map $\sigma:\Ga \to \sym(V)$ is
\begin{itemize}
\item {\bf $(F,\delta)$-multiplicative} if 
$$1-\delta < |V|^{-1} \# \{p \in V:~ \sigma(g)\sigma(h)p=\sigma(gh)p\} \quad \forall g,h\in F$$
\item {\bf $(F,\delta)$-trace-preserving} if 
$$\delta > |V|^{-1} \# \{p \in V:~ \sigma(g)p=p\} \quad \forall g\in F \setminus \{e\}.$$
\end{itemize}
A {\bf sofic approximation} of a group $\Ga$ is a sequence $\Sigma=\{\sigma_i\}_{i \in \N}$ of set maps $\sigma_i:\Ga \to \sym(V_i)$ such that for every finite $F \subset \Ga$ and $\delta>0$ there exists $I$ such that $i>I$ implies $\sigma_i$ is $(F,\delta)$-multiplicative, $(F,\delta)$-trace-preserving and $\lim_{i\to\infty} |V_i| = +\infty$.
A group $\Ga$ is {\bf sofic} if it admits a sofic approximation. 
\end{defn}

\begin{exercise}
Show that the definition above is equivalent to the definition in \S \ref{sec:sofic-bs} in case $\Ga$ is finitely generated.
\end{exercise}

\begin{exercise}\label{exer:rf}
Suppose that $\G$ is residually finite. So there exist finite-index normal subgroups $\G \ge N_1 \ge N_2 \ge \cdots$ such that $\cap_i N_i$ is trivial. Show that the canonical homomorphisms $\G \to \sym(G/N_i)$ form a sofic approximation.
\end{exercise}

\begin{exercise}
Suppose that $\G$ is amenable and $\{F_i\}$ is a F\o lner sequence. Show that for every $i$ there is a map $\sigma_i:\G \to \sym(F_i)$ such that $\sigma_i(g)f = gf$ whenever $gf \in F_i$. Show that these maps form a sofic approximation.
\end{exercise}

\subsubsection{Soficity via ultraproducts}\label{sec:ultra}

Suppose $\Sigma$ is a sofic approximation to $\Ga$ as above. Let $\cU$ be a non-principal ultrafilter on $\N$. Let $\prod_{i \to \cU} \sym(V_i)$ denote the ultraproduct of symmetric groups. To be precise, $\prod_{i \to \cU} \sym(V_i)$ is the direct product $\prod_i \sym(V_i)$ modulo the equivalence relation $(x_i) \sim (y_i)$ iff $\{i\in \N:~x_i=y_i\} \in \cU$. This is a group. Let $N \le \prod_{i \to \cU} \sym(V_i)$ be the set of all $\cU$-equivalence classes of sequences $(\pi_1,\pi_2,\ldots)$ (with $\pi_i \in \sym(V_i)$ ) such that

$$\lim_{i\to\cU} |V_i|^{-1} |\Fix(\pi_i)| = 1$$
where $\Fix(\pi_i)$ is the set of fixed points of $\pi_i$ in $V_i$. Then $N$ is a normal subgroup of $\prod_{i \to \cU} \sym(V_i)$ and the map
$$g\in \Ga \mapsto (\sigma_1(g),\sigma_2(g),\ldots)$$
determines an injective homomorphism from $\Ga$ into the quotient group $\prod_{i \to \cU} \sym(V_i)/N$. 
\begin{exercise}
Prove that $\Ga$ is sofic if and only if for some (any) increasing sequence $\{V_i\}_{i=1}^\infty$ of finite sets, $\Ga$ admits an injective homomorphism into $\prod_{i \to \cU} \sym(V_i)/N$. 
\end{exercise}
From this description it can be shown that the group von Neumann algebra of $\G$ satisfies Connes' embedding conjecture (see  \cite{elek-szabo-2005}). This point of view is elaborated on in \cite{pestov-sofic-survey, pestov-kwiatkowska, capraro-lupini}.

\subsubsection{Which groups are sofic?}


\begin{thm}
The class of sofic groups is closed under
\begin{enumerate}
\item subgroups, direct products, direct limits, inverse limits (so a residually sofic group is sofic), free products,
\item extensions by amenable groups,
\item free products with amalgamation over an amenable subgroup,
\item certain graph products and wreath products.
\end{enumerate}
\end{thm}

For detailed proofs of (1-2) see \cite{elek-szabo-2006}. There are 3 different proofs of (3) in \cite{dykema-2014, paunescu-2011, elek-szabo-2011}. Soficity of graph products is studied in \cite{CHR14} and wreath products in \cite{hayes-sale1, MR3795485}. 

\begin{proof}[Proof sketch of (1)]
Suppose $\Sigma=\{\sigma_i\}$ is a sofic approximation to a group $\Ga$ as in Definition \ref{sec:sofic-sym}. Restricting to a subgroup $\La\le \Ga$ yields a sofic approximation to $\La$. This shows soficity is closed under subgroups. If $\Sigma'=\{\sigma'_i\}$ is a sofic approximation to a group $\Ga'$ then the direct product $\sigma_i \times \sigma'_i: \Ga \times \Ga' \to \sym(V_i) \times \sym(V'_i) \le \sym(V_i \times V'_i)$ gives a sofic approximation to $\Ga \times \Ga'$. So soficity is closed under direct products. Diagonalization arguments show that soficity is preserved under direct limits and inverse limits.

To see that soficity is preserved under free products, consider $\Ga,\Ga'$ as above. Suppose $V_i=V'_i$ and let $\pi_i \in \sym(V_i)$ be a uniformly random permutation. Define $\sigma^{\pi_i}_i:\Ga \to \sym(V_i)$ by conjugation: $\sigma_i^{\pi_i}(g)=\pi_i\sigma_i(g)\pi_i^{-1}$. Now we define $\sigma^{\pi_i}_i*\sigma'_i:\Ga*\Ga' \to \sym(V_i)$ by
$$\sigma^{\pi_i}_i*\sigma'_i(g_1g'_1\cdots g_ng'_n) = \sigma^{\pi_i}_i(g_1)\sigma'_i(g'_1)\cdots \sigma^{\pi_i}_n(g_n)\sigma'_n(g'_n)$$
if $g_1,\ldots, g_n \in \Ga \setminus \{e\}$ and $g'_1,\ldots, g'_n \in \Ga' \setminus \{e\}$ for example. It can be shown that, with probability 1, $\{\sigma^{\pi_i}_i*\sigma'_i\}$ is a sofic approximation to $\Ga*\Ga'$. 


\end{proof}

\begin{remark}
Mal'cev proved that all finitely generated linear groups are residually finite \cite{malcev-1940} and therefore they are sofic. Because soficity is preserved under direct limits it follows that all countable linear groups are sofic.
\end{remark}


It is open whether all countable groups are sofic. However the soficity of the following groups is unknown: free Burnside groups (this was pointed out by Benjy Weiss \cite{weiss-2000}), Tarski monsters,  $SL(3,\Z) *_{F_1=F_2} SL(3,\Z)$ where $F_1,F_2\le SL(3,\Z)$ are isomorphic non-abelian free groups, and the Burger-Mozes groups from \cite{MR1446574, MR1839489}. On the other hand, A. Thom constructed a non-residually finite property (T) sofic group \cite{thom-examples} and Y. de Cornulier constructed a sofic group that is not a limit of amenable groups in the space of marked groups \cite{cornulier-sofic-2011}. Elek and Szabo show that there exists a non-amenable simple sofic group \cite{elek-szabo-2005}. There are several recent surveys on sofic groups \cite{pestov-sofic-survey, pestov-kwiatkowska, capraro-lupini}.

\subsubsection{The space of sofic approximations}\label{sec:space}

\begin{problem}
For a given interesting group $\Ga$, describe the set of all sofic approximations to $\Ga$. 
\end{problem}

Here we will make the above problem more precise and explain some partial results and specific questions. 

To begin we observe that it is possible to perturb a sofic approximation in an inessential way. To be precise, let $\Sigma=\{\sigma_i\}, \Sigma'=\{\sigma'_i\}$ be two sofic approximations to $\Ga$ and suppose that 
$$\sigma_i:\Ga \to \sym(V_i), \quad \sigma'_i:\Ga \to \sym(V'_i).$$
In the special case that $V_i = V'_i$ for all $i$ we can define the {\bf edit-distance} between $\Sigma$ and $\Sigma'$ with respect to a finite set $F \subset \Ga$ by: 
$$d^F(\Sigma,\Sigma') = \limsup_{i\to\infty} |V_i|^{-1} \#\{ v\in V_i:~ \exists f \in F, ~\sigma_i(f)v \ne \sigma_i'(f)v  \}.$$
Strictly speaking this is a pseudo-distance since it is entirely possible that two different sofic approximations satisfy $d^F(\Sigma,\Sigma')=0$ for all $F \subset \Ga$.  If $d^F(\Sigma,\Sigma')=0$ for every finite $F \subset \Ga$ then an exercise shows that the sofic entropy with respect to $\Sigma$ equals the sofic entropy with respect to $\Sigma'$. So we call two sofic approximations that have this property {\bf equivalent}.

If $\Ga$ is amenable then in \cite{elek-szabo-2011} it is shown that every sofic approximation to $\Ga$ is equivalent to one obtained from finite unions of F\o lner sets in a natural way. This completely describes all sofic approximations to $\Ga$. 

We will say that $\Sigma$ is {\bf by homomorphisms} if each $\sigma_i:\Ga \to \sym(V_i)$ is a homomorphism. For example, if $\Sigma=\{\sigma_i\}_{i=1}^\infty$ is any sofic approximation to a free group $\Ga=\langle S \rangle$ and $\Sigma'=\{\sigma'_i\}_{i=1}^\infty$ is the sofic approximation defined by: $\sigma'_i:\Ga \to \sym(V_i)$ is the unique homomorphism satisfying
$$\sigma'_i(s)=\sigma_i(s)~\forall s\in S$$
then $\Sigma'$ is by homomorphisms and it is equivalent to $\Sigma$.

\begin{problem}
If $\Ga$ is an interesting group, such as the fundamental group of a surface, $\F_2 \times \Z, \F_2 \times \F_2, SL(2,\Z) \rtimes \Z^2$ or $SL(3,\Z)$, is every sofic approximation to $\Ga$ equivalent to one by homomorphisms?
\end{problem}







\subsection{Topological sofic entropy}\label{sec:pseudo-top}       

Given a countable group $\Ga$, a sofic approximation $\Sigma$ to $\Ga$, a compact metrizable space $X$,  and an action $T=(T^g)_{g\in \G}$ on $X$ by homeomorphisms, we will define the topological sofic entropy $h_\Sigma(T)$. In a nutshell, the entropy is the exponential rate of growth of the number of approximate partial orbits that can be distinguished up to some small scale. 

First, we recall some basic concepts. A {\bf pseudometric} on a space $X$ is a function $\rho:X \times X \to [0,\infty)$ satisfying all of the properties of a metric with one exception: it is possible that $\rho(x,y)=0$ even if $x \ne y$. If $\rho$ is a pseudo-metric then a subset $S \subset X$ is {\bf $(\rho,\epsilon)$-separated} if $\rho(s_1,s_2)\ge \epsilon$ for all $s_1,s_2 \in S$ with $s_1\ne s_2$. Let $N_\epsilon(S,\rho)$ denote the maximum cardinality of a $(\rho,\epsilon)$-separated subset of $S$. We also let $\rho_2$ and $\rho_\infty$ denote the pseudometrics on $X^d$ (for any integer $d\ge 1$) defined by
$$\rho_\infty(x,y) = \max_i \rho(x_i,y_i),\quad \rho_2(x,y) =\left(\frac{1}{d}\sum_i \rho(x_i,y_i)^2\right)^{1/2}$$
where $x=(x_1,\ldots, x_d), y=(y_1,\ldots, y_d) \in X^d$.

Given an action $T=(T^g)_{g\in \G}$ on $X$, a map $\sigma:\Ga \to \sym(d)$, a finite subset $F \subset \Ga$ and $\delta>0$, let $\Map(T,\rho,F,\delta,\sigma)$ denote the set of all $x \in X^d$ such that
$$\rho_2( T^fx, x \circ \sigma(f) ) < \delta \quad \forall f\in F$$
where $(T^fx)_i = T^fx_i$ and $(x\circ\sigma(f))_i  = x_{\sigma(f)i}$ for all $i$.

In the literature, an element $x \in \Map(T,\rho,F,\delta,\sigma)$ has been referred to as a {\bf microstate}, a {\bf good model} or a {\bf good map}. These terms will be used informally and will not be defined rigorously. The {\bf entropy of $T$ with respect to $\rho$} is
$$h_{\Sigma}(T,\rho) = \sup_{\epsilon>0} \inf_{F \Subset \Ga} \inf_{\delta >0} \limsup_{i\to\infty} |V_i|^{-1}\log \left(N_\epsilon( \Map(T,\rho,F,\delta,\sigma_i), \rho_\infty)\right)$$
where $F\Subset \Ga$ means that $F$ is a finite subset of $\Ga$. We will also write 
$$h_{\Sigma}(\G \cc X,\rho) = h_{\Sigma}(T,\rho)$$
if $T$ is implicit.

\begin{exercise}\label{exer:periodic}
Suppose $\G$ is residually finite and has finite-index normal subgroups $\Ga \ge N_1\ge N_2 \ge \cdots$ with $\cap_i N_i = \{e\}$. Let $\s_i:\G \to \sym(\G/N_i)$ be the canonical homomorphisms. By exercise \ref{exer:rf}, $\Si=\{\s_i\}$ is a sofic approximation. Now suppose that $z\in X$ is stabilized by $(T^g)_{g\in N_i}$ (so $z$ is {\bf $(T,N_i)$-periodic}). Show that if $x\in X^{\G/N_i}$ is defined by $x_{gN_i} = T^gz$ then $x \in \Map(T,\rho,F,\delta,\sigma_i)$ for every $F,\delta$. It follows that the $\Si$-entropy of $T$ is at least the exponential rate of growth of the $(T,N_i)$-periodic points.
\end{exercise}

A pseudometric $\rho$ on $X$ is {\bf generating} for the action if for every $x,y \in X$ with $x\ne y$ there exists $g\in \Ga$ with $\rho(gx,gy)>0$. 
\begin{thm}\label{thm:top-entropy}
Let $\Ga \cc X$ be an action by homeomorphisms on a compact metrizable space. If $\rho_1,\rho_2$ are continuous generating pseudometrics on $X$ then 
$$h_{\Sigma}(T,\rho_1) = h_{\Sigma}(T,\rho_2).$$  
\end{thm}

This theorem (and the definition of topological sofic entropy) is due to Kerr-Li \cite{kerr-li-variational}. See also \cite[Proposition 2.4]{kerr-li-sofic-amenable} and \cite{MR3616077} for a simplified exposition.

\begin{proof}[Proof sketch]
The first step is showing that we can replace a generating pseudometric $\rho$ with a metric $\rho'$. To be precise: let $\phi \in \ell^1(\Ga)$ be a strictly positive function. Define
$$\rho'(x,y) := \sum_{g\in \Ga} \rho(T^gx,T^gy) \phi(g).$$
Then $\rho'$ is a continuous metric on $X$ and
$$h_{\Sigma}(T,\rho) = h_{\Sigma}(T,\rho').$$
Informally, this is because any microstate for $\Ga \cc X$ with respect to $\rho$ is a microstate with respect to $\rho'$ and vice versa, although the parameters $F$ and $\delta$ may change. 

We can now assume that $\rho_1$ and $\rho_2$ are metrics. The statement can now be derived from the observation that for any $\epsilon>0$  there is a $\delta>0$ such that $\rho_1(x,y)<\delta \Rightarrow \rho_2(x,y)<\epsilon$ and vice versa. 
\end{proof}

\begin{defn}
The {\bf $\Sigma$-entropy of $T$} is $h_{\Sigma}(T): = h_{\Sigma}(T, \rho)$ where $\rho$ is any continuous generating pseudometric.
\end{defn}

\begin{remark}
The $\rho_\infty$ appearing in the formula for $h_\Sigma(T,\rho)$ can be replaced with $\rho_2$ without affecting the definition of $h_\Sigma(T,\rho)$. Also the $\limsup$ can be replaced by a $\liminf$ or an ultralimit; however these replacements can lead to different invariants because sofic entropy depends on the choice of sofic approximation in general (see \S \ref{sec:dependence}).
\end{remark}

\begin{exercise}[Symbolic dynamics]
Suppose $A$ is a finite set. An element $\bfx \in A^\G$ is written as either a collection $\bfx = (x_g)_{g\in \G}$ or a function $\bfx:\G \to A$. Let $T=(T^g)_{g\in \G}$ be the {\bf shift action} on $A^\G$ defined by $T^g\bfx(f)=\bfx(g^{-1}f)$. 
\begin{enumerate}
\item Let $\rho$ be the pseudo-metric on $A^\G$ given by $\rho(\bfx,\bfy)=1$ if $\bfx_e\ne \bfy_e$ and $\rho(\bfx,\bfy)=0$ otherwise. Show that $\rho$ is generating for the shift-action.
\item Suppose $X \subset A^\G$ is closed and shift-invariant. Given $\bfx \in A^{V_i}$, its {\bf pullback name} is 
$$\Pi_v^{\sigma_i}(\bfx) \in A^\G, \quad  \Pi_v^{\sigma_i}(\bfx)(g) = \bfx(\sigma_i(g)^{-1}v).$$
Also let
$$P^{\sigma_i}_\bfx = |V_i|^{-1} \sum_{v\in V_i} \delta_{\Pi_v^{\sigma_i}(\bfx)} \in \Prob(A^\G)$$
be its {\bf empirical distribution}. Show that the entropy of the restriction of $T$ to $X$ simplifies to
$$h_{\Sigma}(T \resto X) = \inf_\cO \inf_{\delta>0} \limsup_{i\to\infty} |V_i|^{-1}\log \#\{\bfx\in A^{V_i}:~ P^{\sigma_i}_\bfx(\cO) > 1-\delta  \}.$$
where the first infimum is over all open neighborhoods $\cO$ of $X$ in $A^\G$. 
\item Show that $h_\Si(T)=\log |A|$. 
\item Show that if $X \subset A^\G$ is not $A^\G$ then $h_\Si(T \resto X) < \log |A|$.
\item Suppose $\G$ is residually finite and has finite-index normal subgroups $\Ga \ge N_1\ge N_2 \ge \cdots$ with $\cap_i N_i = \{e\}$. Let $\s_i:\G \to \sym(\G/N_i)$ be the canonical homomorphisms. Let $\cO$ be an open neighborhood of $X$ in $A^\G$ and $\delta>0$. Show that $h_\Si(T \resto X)$ is at most the exponential rate of growth of the number of $(T,N_i)$-periodic points $z \in A^\G$ such that 
$$\#\{gN_i \in \G/N_i:~ gz \in \cO\} \ge (1-\delta) |\G/N_i|.$$
Compare with the lower bound in exercise \ref{exer:periodic}.
\end{enumerate}

\end{exercise}

\begin{thm}\cite{kerr-li-sofic-amenable}
If $\Ga$ is amenable then topological sofic entropy agrees with classical topological entropy. 
\end{thm}

\begin{remark}
The main tool involved in the proof of this theorem is a Rokhlin Lemma for sofic approximations of countable amenable groups. This lemma allows us to express any sofic approximation to an amenable group $\G$ as essentially equivalent to a F\o lner sequence. 
\end{remark}

Topological sofic entropy can also be defined in terms of open covers \cite{zhang-local-variational} (in a manner similar to the original definition of topological entropy \cite{adler-konheim-mcandrew}) or in terms of sequences of continuous functions \cite{kerr-li-variational}. 

\subsubsection{An application: Gottschalk's conjecture and Kaplansky's conjecture}\label{sec:Gottschalk}

Any self-map of a finite set satisfies the following property: if it is injective then it must be surjective. This fundamental property is called {\bf surjunctivity}. It has been generalized to algebraic varieties and regular maps \cite{MR0229613} and proalgebraic varieties satisfying a soficity condition \cite{gromov-1999}. 

\begin{conj}[Gottschalk's Surjunctivity Conjecture]
Suppose $A$ is a finite set (called an alphabet), $\Ga$ a countable group and $\phi:A^\Ga \to A^\Ga$ a continuous $\Ga$-equivariant map. If $\phi$ is injective then it must be surjective. 
\end{conj}

It was this conjecture that inspired Gromov to invent sofic groups (although the name `sofic', derived from the Hebrew word for finite, was coined by Benjy Weiss \cite{weiss-2000}). Gromov proved the conjecture holds for all sofic groups. A new proof, obtained by D. Kerr and H. Li \cite{kerr-li-variational} goes as follows: assuming $\phi$ is injective, 
$$h_\Sigma(\Ga \cc \phi(A^\Ga)) = h_\Sigma(\Ga \cc A^\Ga) = \log |A|.$$
However, if $\phi$ is not surjective then $h_\Sigma(\Ga \cc \phi(A^\Ga))  < \log |A|$ (because  the image $\phi(A^\Ga)$ has trivial intersection with some finite cylinder set). This implies the Conjecture.

Now suppose that $A$ is a finite field and $\phi$ is $A$-linear. In this case, we can think of $\phi$ as an element of the group ring $A\Ga$. The theorem implies that the group $A\Ga$ is {\bf directly finite}: that is $xy = 1$ implies $yx = 1$ for all $x,y \in A\Ga$. More generally, because all fields can be embedded into an ultraproduct of finite fields, the same result holds when $A$ is an arbitrary field. This proves Kaplansky's Direct Finiteness Conjecture for sofic groups. Actually, more is true: $A\Ga$ is directly finite whenever $A$ is a matrix algebra over a division ring \cite{ElekSzabo2004}.





\subsection{Measure sofic entropy}\label{sec:measure-sofic}       
There are two equivalent definitions of measure sofic entropy: one via pseudo-metrics (similar to topological entropy) and one via partitions.

\subsubsection{The pseudometric definition}\label{sec:pseudometric}      

Suppose $X$ is a compact metrizable space, $T=(T^g)_{g\in \Ga}$ is an action on $X$ by homeomorphisms and  $\mu$ is an invariant Borel probability measure on $X$. Let $\Prob(X)$ denote the space of Borel probability measures on $X$. Recall that the weak* topology on $\Prob(X)$ is defined as follows: a sequence $\{\mu_n\}_{n\in \N}$ converges to a measure $\mu_\infty$ if and only if for every continuous function $f:X \to \C$, 
$$\int f~d\mu_n \to \int f~d\mu_\infty$$
as $n\to\infty$. By the Banach-Alaoglu Theorem, $\Prob(X)$ is compact in the weak* topology.

Given a pseudo-metric $\rho$ on $X$, a finite subset $F \subset \Ga$, $\delta>0$ and $\sigma:\Ga \to \sym(V)$, define $\Map(T,\rho,F,\delta,\sigma)$ as in \S \ref{sec:pseudo-top}. In addition, if $\cO \subset \Prob(X)$ is an open neighborhood of $\mu$ then let $\Map(T,\rho,\cO,F,\delta,\sigma)$ denote the set of all $x \in \Map(T,\rho,F,\delta,\sigma)$ such that $x_*u_V \in \cO$ where $u_V$ denotes the uniform probability measure on $V$. These are the microstates that are approximately equidistributed.

The {\bf sofic entropy of $T$ with respect to $\rho$ and $\Sigma$} is
$$h_{\Sigma,\mu}(T,\rho) = \sup_{\epsilon>0} \inf_{\cO} \inf_{F \Subset \Ga} \inf_{\delta >0} \limsup_{i\to\infty} |V_i|^{-1}\log \left(N_\epsilon( \Map(T,\rho,\cO,F,\delta,\sigma_i), \rho_\infty)\right).$$
Intutively, this measures the exponential rate of growth of the number of microstates for the action that are approximately equidistributed with respect to $\mu$.

As in the topological case, a pseudometric $\rho$ on $X$ is {\bf generating} with respect to $T$ if for every $x,y \in X$ with $x\ne y$ there exists $g\in \Ga$ with $\rho(T^gx,T^gy)>0$. 

\begin{thm}\cite{kerr-li-variational}\label{thm:measure-entropy}
For $i=1,2$ let $T_i$ be pmp actions of $\G$ by homeomorphisms on compact metrizable spaces $X_i$ and $\rho_i$ be continuous generating pseudometrics on $X_i$. If these actions are measurably conjugate then
$$h_{\Sigma,\mu_1}(T_1,\rho_1) = h_{\Sigma,\mu_2}(T_2,\rho_2).$$  
\end{thm}

\begin{proof}[Proof sketch]
As in the proof of Theorem \ref{thm:top-entropy}, we can assume, without loss of generality that $\rho_1$ and $\rho_2$ are metrics, not just pseudometrics. Let $\Phi:X_1 \to X_2$ be a measure-conjugacy. By Lusin's Theorem, for every $\eta>0$ there exists a compact set $Y_1 \subset X_1$ such that $\Phi$ restricted to $Y_1$ is uniformly continuous and $\mu_1(Y_1)>1-\eta$. 

Recall that a subset $Z \subset X_1$ is a {\bf continuity set} if $\mu_1(\partial Z)=0$ where $\partial Z = \bZ \cap \overline{X_1-Z}$. For simplicity suppose that Lusin's set $Y_1$ defined above is a continuity set and that its image $\Phi(Y_1)=:Y_2$ is also a continuity set. This does not have to be true but since the continuity sets form an algebra that is dense in the measure algebra it is approximately true.

The portmanteau Theorem states that a sequence $\{\nu_n\}$ of Borel probability measures in $\Prob(X_1)$ (say) converges to $\nu_\infty$ in the weak* topology if and only if $\lim_n \nu_n(Z) = \nu_\infty(Z)$ for every continuity set $Z \subset X_1$. It follows that any microstate for $\Ga \cc (X_1,\mu_1)$ pushes forward under $\Phi$ to a microstate for $\Ga \cc (X_2,\mu_2)$ although the parameters qualifying how good (or bad) the microstate is may change. The theorem follows from this. 
 


\end{proof}

\begin{remark}
The proof sketch above is very different from the proofs in \cite{kerr-li-variational} which are operator-algebraic. See also \cite[Proposition 3.4]{kerr-li-sofic-amenable}.
\end{remark}
\begin{remark}
 In \cite{MR3773269} Ben Hayes relaxes the condition that $X$ is compact to being merely completely metrizable and separable assuming the measure satisfies a `tightness' condition. 
\end{remark}

\begin{defn}[Measure sofic entropy]
The {\bf measure sofic entropy} of the action $T$ with respect to $\Sigma$ is
$$h_{\Sigma,\mu}(T)  = h_{\Sigma,\mu_1}(T_1,\rho_1) $$
where $\Ga \cc^{T_1} (X_1,\mu_1)$ is any compact topological model for $\Ga \cc^T (X,\mu)$ and $\rho_1$ is any generating pseudo-metric on $X_1$.
\end{defn}

\begin{thm}
If $\Ga$ is amenable then measure sofic entropy agrees with classical Kolmogorov-Sinai entropy. 
\end{thm}

\begin{remark}
There are two very different proofs of this result. The one in \cite{kerr-li-sofic-amenable} is based on the sofic Rokhlin's Lemma for amenable groups. The other in \cite{bowen-entropy-2012} is based on a sofic-version of the Rudolph-Weiss Theorem that relative entropy is preserved under orbit-equivalence with respect to the orbit-change sigma-algebra. 
\end{remark}

\begin{remark}
As in the topological case, the $\rho_\infty$ appearing in the formula for $h_{\Sigma,\mu}(T,\rho)$ can be replaced with $\rho_2$ without affecting its value. Also the $\limsup$ can be replaced by a $\liminf$ or an ultralimit; however these replacements can lead to different invariants because sofic entropy depends on the choice of sofic approximation in general (see \S \ref{sec:dependence}).
\end{remark}

\begin{thm}\label{thm:bernoulli-sofic}
If $\Ga \cc (K,\kappa)^\Ga$ is a Bernoulli shift and $\Sigma$ is an arbitrary sofic approximation to $\Ga$ then
$$h_{\Sigma,\kappa^\Ga}(\Ga \cc K^\Ga) = H(K,\kappa).$$
\end{thm}

The case in which $H(K,\kappa)<\infty$ is proven in \cite{bowen-jams-2010}. The infinite entropy case is handled in \cite{MR2813530}.

\begin{proof}[Proof sketch]
 The lower bound is obtained as follows: fix an open neighborhood $\cO$, finite $F\subset \Ga$ and $\delta>0$. Let $\phi:V_i \to K$ be a random map with law equal to the product measure $\kappa^{V_i}$. Let $\tphi:V_i \to K^\G$ be the ``pullback'' defined by
 $$\tphi(v)(g)=\phi(\sigma_i(g)^{-1}v).$$
  Then using Chebyshev's inequality it is shown that with high probability, when $i$ is large, $\tphi \in  \Map(T,\rho,\cO,F,\delta,\sigma_i)$. Here $\rho$ is the pseudometric given by $\rho(x,y)=\rho_K(x(e),y(e))$ where $\rho_K$ is an arbitrary metric on $K$ (which may be assumed to be a compact metrizable space). The law of large numbers now gives the lower bound.

In the case $H(K,\kappa)<\infty$, the upper bound is shown as follows. Let $\phi:V_i \to K^\Ga$ be any microstate. Let $\pi:K^\Ga \to K$ be projection to the identity coordinate. Then if $\phi$ is a good enough microstate the composition $\pi \circ \phi$ pushes the uniform measure $u_{V_i}$ forward to a measure on $K$ that is close to $\kappa$ in total variation distance. On the other hand, observe that $\phi$ is essentially determined by  $\pi\circ \phi$. So it suffices to observe that the number of maps $\phi':V_i \to K$ such that $\phi'_*u_{V_i}$ is close to $\kappa$ is approximately $\exp(H(K,\kappa)|V_i|)$. This is an application of elementary combinatorics and Stirling's formula.
\end{proof}

\subsubsection{The partition definition}\label{sec:partition}       

\begin{defn}
If $\Sigma_1,\Sigma_2$ are sigma-algebras on sets $X_1,X_2$ respectively then a {\bf homomorphism} between  them is a map $\phi:\Sigma_1\to \Sigma_2$ such that for all $A,B \in \Sigma_1$,
$$\phi(A\cup B) = \phi(A) \cup \phi(B), ~\phi(A \cap B) = \phi(A) \cap \phi(B), ~\phi(\emptyset)  = \emptyset, ~\phi(X_1)=X_2.$$
\end{defn}

Suppose $\Ga \cc^T (X,\mu)$ is a pmp action. For simplicity, $T^gx$ is denoted by $gx$ for $g\in \G,x\in X$.
\begin{defn}
Let $\cP$ be a finite measurable partition of $X$, $F \subset \Ga$ a finite set with $1_\G \in F$, and $\delta>0$. Let $\cP^F = \bigvee_{f\in F} f^{-1}\cP$ be the coarsest partition containing $f^{-1}\cP$ for $f\in F$. If $\cQ$ is any partition, let $\s\textrm{-alg}(\cQ)$ be the smallest sigma-algebra containing $\cQ$. Also let $2^V$ denote the sigma-algebra of all subsets of $V$ and $u_V$ be the uniform probability measure on $V$.

Given $\sigma:\Ga \to \sym(V)$, let $\Hom_\mu(\cP,F,\delta,\sigma)$  be the set of all homomorphisms $\phi: \s\textrm{-alg}(\cP^F) \to 2^V$ such that
\begin{enumerate}
\item $\sum_{P \in \cP} u_V(\sigma_s \phi(P) \vartriangle \phi(sP) ) < \delta$ for all $s\in F^{-1}$ and
\item $\sum_{A \in \cP^F} |u_V(\phi(A)) - \mu(A)|<\delta$. 
\end{enumerate}
The sofic entropy is defined as the exponential rate of growth of the number of such homomorphisms that can be extended to a more refined partition. To be precise, if $\cQ\le \cP$ is a partition coarser than $\cP$ and $\phi: \s\textrm{-alg}(\cP^F) \to 2^V$ is a homomorphism then let $\phi\resto \cQ$ be the restriction of $\phi$ to $\cQ$. Let $|\Hom_\mu(\cP,F,\delta,\sigma)|_\cQ$ be the cardinality of the set of restrictions $\{ \phi \resto \cQ:~ \phi \in \Hom_\mu(\cP,F,\delta,\sigma)\}$. Finally, for a sigma-sub-algebra $\cS \subset \cB$ define
\begin{eqnarray*}
h_{\Sigma,\mu}(T,\cS) = h_{\Sigma,\mu}(\Ga\cc X,\cS) = \sup_\cQ \inf_\cP\inf_{F \subset \Ga} \inf_{\delta>0}  \limsup_{i\to\infty} \frac{1}{|V_i|} \log |\Hom_\mu(\cP,F,\delta,\sigma_i)|_\cQ
\end{eqnarray*}
where the sup is over all finite partitions $\cQ \subset \cS$, the first inf is over finite partitions $\cP$ with $\cQ \le \cP \subset \cS$ and the second inf is over all finite subsets of $\Ga$.
\end{defn}

Recall that a sigma-sub-algebra $\cS \subset \cB_X$ is {\bf generating} for $T$ if $\cB_X$ is the smallest $T(\Ga)$-invariant sigma-algebra containing $\cS$ up to sets of measure zero. 
\begin{thm}\cite{kerr-partitions}\label{thm:measure-entropy-partition}
If $\cS \subset \cB_X$ is a generating sigma-sub-algebra then $h_{\Sigma,\mu}(T,\cS) = h_{\Sigma,\mu}(T).$
\end{thm}

\begin{remark}
In the special case in which $\cS$ is the sigma-algebra generated by a finite partition, the definition above is easily seen to be equivalent to the one given in the introduction.
\end{remark}





\subsection{The $f$-invariant}\label{sec:f}

Let $S$ be a finite or countable set and $\Ga=\langle S \rangle$ the free group generated by $S$. Let $\Ga \cc^T (X,\mu)$ be a pmp action and $\cP$ a measurable partition of $X$ with finite Shannon entropy. Define
$$F_\mu(T,\cP) := H_\mu(\cP) + \sum_{s\in S} \big(H_\mu(\cP \vee s\cP)  - 2H_\mu(\cP)\big),$$
$$f_\mu(T,\cP) := \inf_{W \Subset \Ga} F_\mu(T,\cP^W)$$
where $\cP^W= \bigvee_{w\in W} w^{-1}\cP$. For simplicity, we have written $s\cP$ instead of $T^s\cP$.

\begin{thm}\label{thm:f1}\cite{bowen-annals-2010}
If $\cP,\cQ$ are generating partitions each with finite Shannon entropy then $f_\mu(T,\cP) = f_\mu(T,\cQ)$.
\end{thm}

The $f$-invariant of the action is defined by $f_\mu(T) = f_\mu(T,\cP)$ where $\cP$ is any generating partition with finite Shannon entropy. The $f$-invariant is undefined if no such partition exists.

\begin{remark}
This was proven in \cite{bowen-annals-2010} under the assumption that $S$ is finite. The proof when $S$ is countable is essentially the same. 
\end{remark}

\begin{proof}[Proof sketch]
Given two partitions $\cP,\cQ$, their {\bf Rokhlin distance} is defined by
$$d(\cP,\cQ) : = H_\mu(\cP|\cQ) + H_\mu(\cQ|\cP).$$
Partitions that agree up to measure zero sets are identified. With this convention, the Rokhlin distance really is a distance function on the space of all partitions with finite Shannon entropy, which is denoted by $\Part(X,\mu)$. 

Two partitions $\cP,\cQ$ are {\bf combinatorially equivalent} if there exist finite subsets $F,K \subset \Ga$ such that $\cQ \le \cP^F$ and $\cP \le \cQ^K$. The first step is showing that if $\cP \in \Part(X,\mu)$ is a generating partition then its combinatorial equivalence class is dense in the subspace of all generating partitions with finite Shannon entropy. Since $F$ is continuous on $\Part(X,\mu)$, $f$ is upper semi-continuous. It now suffices to show that if $\cP,\cQ$ are combinatorially equivalent then $f_\mu(T,\cP) = f_\mu(T,\cQ)$. 

The partition $\cQ$ is a {\bf simple splitting} of $\cP$ if there is an element $s\in S \cup S^{-1}$ and a partition $\cR \le \cP$ such that $\cQ = \cP \vee s\cR$. We say $\cQ$ is a {\bf splitting} of $\cP$ if there is a sequence $\cP=\cQ_0,\cQ_1,\ldots, \cQ_n=\cQ$ such that $\cQ_{i+1}$ is a simple splitting of $\cQ_i$ for $0\le i <n$. The second step is showing that if $\cP_1,\cP_2$ are combinatorially equivalent then there exists a common splitting $\cQ$ of both of them. Moreover, splittings preserve the combinatorial equivalence class. Therefore, it suffices to show: if $\cQ$ is a simple splitting of $\cP$ then $F_\mu(T,\cQ) \le F_\mu(T,\cP)$. This fact follows from a short calculation. For simplicity, assume $\cR \le \cP$, $t\in S$ and $\cQ = \cP \vee t\cR$. Then
\begin{eqnarray*}
 F_\mu(T,\cQ) &=& H_\mu(\cQ) + \sum_{s\in S} H_\mu(\cQ \vee s\cQ)  - 2H_\mu(\cQ) \\
  &=& H_\mu(\cP) + H_\mu(\cQ|\cP) + \sum_{s\in S} H_\mu(\cP \vee s\cP)  - 2H_\mu(\cP) + H_\mu(\cQ \vee s\cQ|\cP\vee s\cP)  - 2H_\mu(\cQ|\cP) \\
  &=& F_\mu(T,\cP) + H_\mu(\cQ|\cP) + \sum_{s\in S} H_\mu(\cQ \vee s\cQ|\cP\vee s\cP)  - 2H_\mu(\cQ|\cP) \\
    &=& F_\mu(T,\cP) + H_\mu(\cQ|\cP) + \sum_{s\in S} H_\mu(\cQ |\cP\vee s\cP)  + H_\mu(s\cQ|\cP\vee s\cP \vee \cQ) - 2H_\mu(\cQ|\cP)\\
    &=& F_\mu(T,\cP) +\Big(H_\mu(\cQ |\cP\vee t\cP)  + H_\mu(t\cQ|\cP\vee t\cP \vee \cQ) - H_\mu(\cQ|\cP) \Big)\\
    &&+ \sum_{s\in S-\{t\}} H_\mu(\cQ |\cP\vee s\cP)  + H_\mu(s\cQ|\cP\vee s\cP \vee \cQ) - 2H_\mu(\cQ|\cP).
 \end{eqnarray*}
Observe that $H_\mu(\cQ |\cP\vee s\cP)  + H_\mu(s\cQ|\cP\vee s\cP \vee \cQ) - 2H_\mu(\cQ|\cP) \le 0$ for every $s\in S$ and moreover if $s=t$ then $H_\mu(\cQ|\cP\vee s\cP)=0$. So
$$F_\mu(T,\cQ)  \le F_\mu(T,\cP) + H_\mu(t\cQ|\cP\vee t\cP \vee \cQ) - H_\mu(\cQ|\cP)  \le  F_\mu(T,\cP).$$
\end{proof}

\begin{remark}
The proof shows a little more: partitions can be partially ordered by $\cP \prec \cQ$ if $\cQ$ is a splitting of $\cP$. Then  $f_\mu(T,\cP)$ is the limit of $F_\mu(T,\cQ)$ as $\cQ$ tends to infinity in this partial order. Moreover, 
$$f_\mu(T,\cP) = \lim_{n\to\infty} F_\mu(T,\cP^{W_n})$$
where $W_n$ is any increasing sequence of finite subsets of $\Ga$ such that (1) the induced subgraph of $W_n$ is connected in the Cayley graph of $(\Ga,S)$ and (2) $\cup_n W_n =\Ga$. Using this last fact, a direct computation shows that the $f$-invariant of the Bernoulli shift $\Ga \cc (K,\kappa)^\Ga$ is the Shannon entropy $H(K,\kappa)$. 
\end{remark}

Why do we care about the $f$-invariant? In contrast to sofic entropy, the $f$-invariant tends to be easy to compute. For example, for Markov processes $f=F$  (see \S \ref{sec:markov-chains}). Morever, it is additive under direct products, satisfies an ergodic decomposition formula, a subgroup formula, has a relative entropy theory and satisfies some cases of Yuzvinskii's formula. These results do not hold for sofic entropy in general.

\subsubsection{Other formulations and other groups}\label{sec:other}
If $\cP$ is a Markov partition then $F_\mu(T,\cP)=f_\mu(T,\cP)$. Using this fact, the following alternative formula for the $f$-invariant was found in \cite{bowen-entropy-2010a}. Define
$$F^*_\mu(T,\cP) : = H_\mu(\cP) + \sum_{s\in S} \big(h_\mu(T^s,\cP) - H_\mu(\cP)\big).$$
\begin{thm}\label{thm:f*}\cite{bowen-entropy-2010a}
$$f_\mu(T,\cP) = \inf_{W \Subset \Ga} F^*_\mu(T,\cP^W) = \lim_{n\to\infty} F^*_\mu(T,\cP^{W_n})$$
where the limit is with respect to any increasing sequence of finite subsets of $\Ga$ such that (1) the induced subgraph of $W_n$ is connected in the Cayley graph of $(\Ga,S)$ and (2) $\cup_n W_n =\Ga$. 
\end{thm}

The theorem above leads to the following idea: suppose for $i=1,2$, $\Ga_i$ are amenable groups, $A_i \le \Ga_i$ are subgroups and $\phi:A_1 \to A_2$ is an isomorphism. Let $\Ga=\Ga_1 *_\phi \Ga_2$ be the amalgamated free product. Let $\Ga \cc^T (X,\mu)$ be a pmp action and $\cP$ a partition of $X$ with finite Shannon entropy. Define
$$F_\mu(T,\cP) = h_{\mu}(\Ga_1 \cc X,\cP) + h_{\mu}(\Ga_2 \cc X,\cP) - h_\mu(A \cc X,\cP)$$
where $A \le \Ga$ is the subgroup corresponding to $A_1,A_2$ and, for example, $ h_{\mu}(\Ga_i \cc X,\cP) $ is the classical entropy rate of $\cP$ with respect to $\Ga_i$-action. Also let
$$f_\mu(T,\cP) = \inf_{W \Subset \Ga} F_\mu(T,\cP^W).$$
As above, it can be shown that if $\cP,\cQ$ are generating partitions with finite Shannon entropy then $f_\mu(T,\cP)  = f_\mu(T,\cQ)$ and so this determines a measure-conjugacy invariant for $\Ga$-actions.

\begin{problem} This idea has not appeared yet in the literature and is worthy of further exploration: can one extend it to other graphs of groups? Does it depend on how the group is represented as a graph of groups? Can one obtain results for such invariants similar to the results for the $f$-invariant of free group actions (for example, the sofic interpretation, the ergodic decomposition formula, the subgroup formula, and so on)? For example, the fundamental group of a closed surface of genus $g \ge 2$ can be written as a free product of free groups amalgamated over an infinite cyclic subgroup.
\end{problem}
\subsubsection{The sofic interpretation of the $f$-invariant}\label{sec:interpretation}       

In order to interpret the  $f$-invariant as a kind of sofic entropy, we introduce random sofic approximations and their sofic entropies.  Let $\{V_i\}_i$ be a sequence of finite sets and for each $i$, let $\P_i$ be a probability measure on the space of maps $\G \to \sym(V_i)$. The sequence $\P:=\{\P_i\}_{i=1}^\infty$ is a {\bf random sofic approximation} to $\Ga$ if for every finite set $F \subset \Ga$ and $\delta>0$,
\begin{itemize}
\item
$$1=\lim_{i\to\infty} \P_i\{ \sigma: \Ga \to \sym(V_i):~ \sigma \textrm{ is } (F,\delta)\textrm{-trace preserving} \}, $$
\item there exists $I$ such that $i>I$ implies $\P_i$-a.e. $\sigma$ is $(F,\delta)$-multiplicative,
\item $\lim_i |V_i| = +\infty$. 
\end{itemize}
Each definition of $\Sigma$-entropy given above can be generalized to $\P$-entropy by replacing $N_\epsilon(\cdot)$ or $|\Hom_\mu(\cdot)|$ with its expectation with respect to $\P_i$. For example, suppose $\G \cc^T X$ is a continuous action on a compact space, $\mu$ an invariant probability measure on $X$ and $\rho$ a continuous generating pseudo-metric. Then the topological $\P$-entropy and measure $\P$-entropy are: 
$$h_{\P}(T,\rho) := \sup_{\epsilon>0} \inf_{F \Subset \Ga} \inf_{\delta >0} \limsup_{i\to\infty} |V_i|^{-1}\log \E_i\left(N_\epsilon( \Map(T,\rho,F,\delta,\sigma_i), \rho_\infty)\right)$$
$$h_{\P,\mu}(T,\rho) := \sup_{\epsilon>0} \inf_\cO \inf_{F \Subset \Ga} \inf_{\delta >0} \limsup_{i\to\infty} |V_i|^{-1}\log \E_i\left(N_\epsilon( \Map(T,\rho,\cO,F,\delta,\sigma_i), \rho_\infty)\right)$$
where $\E_i$ denotes expectation with respect to $\P_i$. The obvious analogs of Theorems \ref{thm:top-entropy}, \ref{thm:measure-entropy} and \ref{thm:measure-entropy-partition}  remain true and the proofs are essentially the same.

Now let $\Ga=\langle S \rangle$ be a free group where $S$ is finite or countable. The set of homomorphisms from $\G$ to $\sym(n)$ is naturally identified with the set $\sym(n)^S$ of all maps from $S$ to $\sym(n)$. Let $\pi_n$ be the uniform probability measure on $\sym(n)$ and  $\P_n=\pi_n^S$ be the product measure on $\sym(n)^S$. 

\begin{exercise}
$\P=\{\P_n\}_{n=1}^\infty$ is a random sofic approximation to $\Ga$. 
\end{exercise}

\begin{thm}\label{thm:f}
If $\cP$ is a generating partition with finite Shannon entropy then
$$f_\mu(T) =  h_{\P,\mu}(T).$$
That is, the $f$-invariant of the action is the same as the sofic entropy with respect to the random sofic approximation $\cP$.
\end{thm}

\begin{remark}
This is proven in \cite{bowen-entropy-2010b} when $S$ is finite. The case of countably infinite $S$ is similar. 
\end{remark}

\begin{proof}[Proof sketch]
For simplicity, assume $S$ is finite, $\cP$ is finite and $\mu(P \cap T^sQ)$ is rational for every $P,Q \in \cP$ and $s\in S$. Given $\sigma: \Ga \to \sym(n)$, let $\Part(\sigma,\cP)$ be the set of all maps $\phi: [n] \to \cP$ such that 
$$n \mu(P \cap s^{-1}Q)  = \#\{ v \in [n]:~ \phi(v)=P, \phi(\sigma(s)v) =Q\}.$$
By direct combinatorial arguments, one can obtain an exact formula for $\E[ |\Part(\sigma,\cP)|]$ where $\sigma:\Ga \to \sym(n)$ is uniformly random. An application of Stirling's formula shows
$$F_\mu(T,\cP) =  \limsup_{n\to\infty} n^{-1} \log \E[ |\Part(\sigma,\cP)|].$$ 
To handle the case in which $\mu(P \cap T^sQ)$ is irrational, let $\Part(\sigma,\cP,\epsilon)$ be the set of all maps $\phi: [n] \to \cP$ such that 
$$\Big| \mu(P \cap s^{-1}Q) - n^{-1} \#\{ v \in [n]:~ \phi(v)=P, \phi(\sigma(s)v) =Q\}\Big| < \epsilon.$$
A perturbation argument and Stirling's formula implies
$$F_\mu(T,\cP) = \lim_{\epsilon \searrow 0} \limsup_{n\to\infty} n^{-1} \log \E[ |\Part(\sigma,\cP,\epsilon)|] = \lim_{\epsilon \searrow 0} \liminf_{n\to\infty} n^{-1} \log \E[ |\Part(\sigma,\cP,\epsilon)|].$$
The Theorem follows by replacing $\cP$ with $\cP^W$ and taking the infimum over finite $W \subset \Ga$. 
\end{proof}

\begin{problem}
Unlike the $f$-invariant, $h_{\P,\mu}(T)$ is well-defined even if $\Ga \cc^T (X,\mu)$ does not have a generating partition with finite Shannon entropy. The $f$-invariant satisfies many useful identities: it is additive under direct products, satisfies an ergodic decomposition formula, a subgroup formula and possesses a relative version. Can any of these results be extended to $h_{\P,\mu}(T)$? 
\end{problem}


\subsection{Rokhlin entropy}\label{sec:Rokhlin}

Let $\cB_X$ denote the sigma-algebra of measurable subsets of $X$. For any subcollection $\cF \subset \cB_X$, let $\s\textrm{-alg}(\cF) \subset \cB_X$ denote the sub-sigma-algebra generated by $\cF$ and, if  $\Ga \cc^T X$ is a measurable action then let $\s\textrm{-alg}(T,\cF)$ denote the smallest sub-sigma-algebra containing $T^gF$ for every $g \in \Ga$ and $F \in \cF$. We do not distinguish between sigma-algebras that agree up to null sets. Thus we write $\cF_1=\cF_2$ if $\cF_1$ and $\cF_2$ agree up to null sets.

\begin{defn}[Rokhlin entropy]
The {\bf Rokhlin entropy} of an ergodic pmp action $\Ga \cc^T (X,\mu)$ is defined by
$$h^{\rm{Rok}}(T) = \inf_\cP H_\mu(\cP)$$
where the infimum is over all partitions $\cP$ with $\s\textrm{-alg}(T,\cP)=\cB_X$. 
For any $\cF \subset \cB_X$  the {\bf relative Rokhlin entropy} is defined by
$$h^{\rm{Rok}}(T|\cF) = \inf_\cP H(\cP|\cF)$$
where the infimum is over all partitions $\cP$ such $\s\textrm{-alg}(T,\cP \cup \cF) = \cB_X$.  If $T$ is non-ergodic then the Rokhlin entropy is defined by 
$$h^{\rm{Rok}}(T) = h^{\rm{Rok}}(T|\cI_T)$$
where $\cI_T$ denotes the sigma-algebra of $T(\Ga)$-invariant Borel sets. 
\end{defn}

Rokhlin entropy is clearly a measure-conjugacy invariant. Moreover, in case $\Ga$ is amenable, it agrees with Kolmogorov-Sinai entropy \cite{seward-tucker-drob}  (the special case in which $A=\Z$ was handled earlier by Rokhlin \cite{rohlin-lectures}). However, it can be difficult to compute. For example, it is not known whether every countable group has an ergodic essentially free action with positive Rokhlin entropy. The only known lower bound is:

\begin{prop}\label{prop:Rokhlinsofic}
For any pmp action $\G \cc^T (X,\mu)$ and any sofic approximation $\Si$, $h^{\rm{Rok}}(T) \ge h_{\Si,\mu}(T)$.
\end{prop}

\begin{exercise}
Use the partition definition of $\Si$-entropy to prove Proposition \ref{prop:Rokhlinsofic}.
\end{exercise}

\begin{question}
Suppose $T$ is an essentially free ergodic pmp action. Does $h_{\Sigma,\mu}(T) \ne -\infty$ necessarily imply $h_{\Sigma,\mu}(T)  = h^{\rm{Rok}}(T)$? 
\end{question}

\begin{question}
Suppose $\Ga$ is a finitely generated free group, $T$ is an essentially free ergodic pmp action with a finite generating partition $\cP$. Is the difference
$$f_\mu(T) - h^{\rm{Rok}}(T)$$
an invariant of the weak equivalence class of the action? The notions of weak containment and equivalence for group actions were introduced by A. Kechris as an analogues of weak containment  and equivalence for unitary representations \cite[II.10 (C)]{Kechris-global-aspects}.  
\end{question}

\subsubsection{Applications to the classification of Bernoulli shifts}
By Theorem \ref{thm:bernoulli-sofic}, if $\Ga$ is sofic then the sofic entropy of the Bernoulli shift $\Ga \cc (K,\kappa)^\Ga$ is the Shannon entropy of the base $H(K,\kappa)$. It is clear that the Rokhlin entropy of this shift is $\le H(K,\kappa)$ since the partition $\cP=\{P_k:~k\in K\}$ defined by $P_k=\{x\in K^\Ga:~x(e)=k\}$ is generating and $H_\mu(\cP)=H(K,\kappa)$. So in this case at least, the Rokhlin entropy agrees with the sofic entropy.

What if $\Ga$ is non-sofic? Of course, we do not know whether non-sofic groups exist but even in this case there are some very interesting results. To describe them, let $h^{\rm{Rok}}_{sup}(\Ga)$ be the supremum of $h^{\rm{Rok}}(T)$ over all essentially free, ergodic actions $T$ of $\Ga$ with finite Rokhlin entropy. Of course, when $\Ga$ is sofic then $h^{\rm{Rok}}_{sup}(\Ga) = +\infty$. We do not know whether or not $h^{\rm{Rok}}_{sup}(\Ga) = +\infty$ for all countable groups. However:
\begin{thm}\label{thm:seward-bernoulli} \cite{seward-kreiger-2}
$h^{\rm{Rok}}(\Ga \cc (K,\kappa)^\Ga) = \min( H(K,\kappa), h^{\rm{Rok}}_{sup}(\Ga) )$.
\end{thm}

\begin{proof}[Remarks on the proof]
The full proof is quite intricate and the reader is encouraged to see \cite{seward-kreiger-2} for details. It is a proof by contradiction. Assuming the result is false, there exists an essentially free ergodic action $T$ with
$$h^{\rm{Rok}}(\Ga \cc (K,\kappa)^\Ga) < h^{\rm{Rok}}(T) < H(K,\kappa).$$
From this, one constructs a $\Ga$-equivariant Borel map $\Phi:X \to K^\Ga$ such that (1) $\Phi$ is an isomorphism onto its image, (2) $\Phi_*\mu$ is close to $\kappa^\Ga$ in the weak* topology. Using upper semi-continuity of Rokhlin entropy this implies a contradiction. The construction of $\Phi$ is highly non-trivial. It combines tools from Seward's generalization of Krieger's Generator Theorem \cite{seward-kreiger-1} with an Abert-Weiss factor map \cite{abert-weiss-2013}. 
\end{proof}

It follows that if $h^{\rm{Rok}}_{sup}(\Ga)=+\infty$ then Rokhlin entropy distinguishes Bernoulli shifts. Moreover, Seward proves in \cite{seward-kreiger-2} that if $h^{\rm{Rok}}_{sup}(\Ga)<\infty$ and $H$ is any infinite locally finite group then $h^{\rm{Rok}}_{sup}(\Ga\times H)=0$. This is a most interesting condition! It implies that all ergodic actions of $\Ga\times H$ have Rokhlin entropy zero, even the Bernoulli shift with base space $([0,1],\rm{Leb})$! Thus if it is true that for every countable group $\Ga$ there exists some ergodic essentially free action with positive Rokhlin entropy, then Bernoulli shifts are distinguished by Rokhlin entropy.

\subsection{Naive entropy}       

 \begin{defn}
Let $\Ga \cc^T (X,\mu)$ be a pmp action and $\cP$ a partition of $X$. The {\bf naive entropy} of $\cP$ is 
$$h^{\rm{naive}}_\mu(T,\cP) = \inf_{W \Subset \Ga} |W|^{-1} H(\cP^W)$$
where $\cP^W = \bigvee_{w\in W} (T^w)^{-1}\cP$. The {\bf naive entropy} of $T$ is 
$$h^{\rm{naive}}_\mu(T) = \sup_\cP h^{\rm{naive}}_\mu(T,\cP)$$
where the supremum is over all finite-entropy partitions $\cP$.
\end{defn}

It is an exercise to show that if $\Ga$ is amenable then naive entropy coincides with Kolmogorov-Sinai entropy. However if $\Ga$ is non-amenable the situation is very different:

\begin{thm}\label{thm:naive-non-amenable}
If $\Ga$ is non-amenable and $\G \cc^T (X,\mu)$ is a pmp action then $h^{\rm{\rm{naive}}}_\mu(T) \in \{0,+\infty\}$.
\end{thm}

\begin{proof}
Suppose there is a finite-entropy partition $\cP$ of $X$ with $h^{\rm{\rm{naive}}}_\mu(T,\cP)>0$. Let $W \subset \Ga$ be finite. Then
$$h^{\rm{\rm{naive}}}_\mu(T,\cP^W) = \inf_{F \Subset \Ga} |F|^{-1} H_\mu(\cP^{WF}) = \inf_{F \Subset \Ga}  \frac{H_\mu(\cP^{WF})}{|WF|} \frac{|WF|}{|F|}.$$
Since $\Ga$ is non-amenable, for every real number $r>0$ there is a finite $W \subset \Ga$ such that
$$\inf_{F \Subset \Ga} \frac{|WF|}{|F|}>r.$$
Hence $\sup_{W \Subset \Ga} h_\mu^{\rm{naive}}(T,\cP^W) = +\infty$ proving the theorem.
\end{proof}

In many respects naive entropy behaves better than sofic entropy. For example, it is monotone under factor maps and additive under direct products. Moreover, the naive Pinsker algebra is naturally defined as the collection of all measurable subsets $A \subset X$ such that the partition $\cP=\{A,X-A\}$ has zero naive entropy. An exercise shows that this really is a sigma-sub-algebra that contains all zero-naive-entropy factors. Moreover:

\begin{prop}\label{prop:naive-Rokhlin2}
For ergodic actions, naive entropy is an upper bound for Rokhlin entropy.
\end{prop}

The proof of this requires the following surprising result due to Brandon Seward:
\begin{thm}\label{prop:naive-Rokhlin1}\cite{seward-kreiger-2}
Let $T$ be an ergodic pmp action with a generating partition $\cP$ with finite Shannon entropy. Then 
$$h^{\rm{Rok}}(T) \le h^{\rm{naive}}_\mu(T,\cP).$$
\end{thm}

\begin{proof}[Proof of Proposition \ref{prop:naive-Rokhlin2}]
Let $T$ be an ergodic pmp action with naive entropy zero. It suffices to show that, for every $\epsilon>0$, there exists a generating partition with Shannon entropy $<\epsilon$. 

Let $\cP$ be a countable generating partition for $T$. Such a partition always exists by a general result due to Rokhlin \cite{rohlin-lectures}. It might, however, be the case that $\cP$ has infinite Shannon entropy. In any case, there exist finite partitions $\cQ_1 \le \cQ_2 \le \cdots \le  \cP$ such that $\cP  = \bigvee_i \cQ_i$. We apply Theorem \ref{prop:naive-Rokhlin1} to the factor generated by $\cQ_i$. So there exists a partition $\cR_i$ such that $H_\mu(\cR_i)< \epsilon/2^i$ and $\cR_i$ generates the same factor as $\cQ_i$. It follows that if $\cR_\infty := \bigvee_i \cR_i$, then $H_\mu(\cR_\infty)<\epsilon$ and $\cR_\infty$ is generating.
\end{proof}

\subsubsection{Applications to Gottschalk's Conjecture}

\begin{thm}\cite{seward-kreiger-2}
If $h^{\rm{Rok}}_{sup}(\Ga) =+\infty$ then Gottshalk's Conjecture (from \S \ref{sec:Gottschalk}) is true for $\Ga$.
\end{thm}

\begin{proof}[Proof sketch]
Let $A$ be a finite alphabet and let $u_A$ denote the uniform probability measure on $A$. Assuming $\Phi:A^\Ga \to A^\Ga$ is continuous, $\Ga$ equivariant and injective, Theorem \ref{thm:seward-bernoulli} implies
$$h^{\rm{Rok}}(\Ga \cc (A,u_A)^\Ga) = \log |A| = h^{\rm{Rok}}(\Ga \cc (A^\Ga,\Phi_*u_A^\Ga)).$$
Using Theorem \ref{prop:naive-Rokhlin1}  with the canonical partition, it can be shown that if $\Phi$ is not surjective then 
$$ h^{\rm{Rok}}(\Ga \cc (A^\Ga,\Phi_*u_A^\Ga)) < \log |A|.$$
This contradiction proves the theorem.
\end{proof}

\subsubsection{Topological naive entropy}

Recent work of Peter Burton \cite{MR3649602}  introduced the following topological counterpart:
\begin{defn}[Topological naive entropy]
Let $\Ga \cc^T X$ be a continuous action on a compact metrizable space. Given an open cover $\cU$ of $X$, let $N(\cU)$ be the smallest cardinality of a subcover. The {\bf naive entropy} of $\cU$ is
$$h^{\rm{naive}}(T,\cU) = \inf_{F \Subset \Ga} |F|^{-1} \log(N(\cU^F))$$
where $\cU^F$ is the open cover $\vee_{f\in F} (T^f)^{-1}\cU$.  The {\bf naive entropy} of $T$ is 
$$h^{\rm{naive}}(T) = \sup_\cU h^{\rm{naive}}(T,\cU)$$
where the supremum is over all finite open covers.
\end{defn}
Burton shows that topological naive entropy provides an upper bound for measure naive entropy and that distal systems have zero naive entropy in both measure and topological senses when $\Ga$ has an element of infinite order. He also shows that the generic action of the free group by homeomorphisms on the Cantor set has zero topological naive entropy. However the following question appears to be open:

\begin{question}
Does every countable group admit an essentially free pmp action with zero naive entropy?
\end{question}

\section{Special classes of actions}\label{sec:special classes}       

\subsection{Trivial actions}\label{sec:trivial}

It might come as a surprise that the trivial action of $\Ga$ on $(X,\mu)$ is interesting, from the point of view of sofic entropy theory. To explain, fix a sofic approximation $\Sigma=\{\sigma_n\}_{n \in \N}$ to $\Ga$ and a standard probability space $(X,\mu)$ (which may be atomic). The trivial action $\tau_X=(\tau^g_X)_{g\in \G}$ on $X$ is defined by $\tau^g_Xy=y$ for all $y\in X$ and $g\in \Ga$. We will show that $h_{\Sigma,\mu}(\tau_X)\in \{-\infty,0\}$ and that both cases occur. The upper bound $h_{\Sigma,\mu}(\tau_X)\le 0$ can be derived directly. 

\subsubsection{Expanders}\label{sec:expanders}

For simplicity, suppose $\Ga$ is finitely generated and let $S\subset \Ga$ be a finite generating set. Let $\Sigma=\{\s_n:\G \to \sym(V_n)\}$  be a sofic approximation. Let $G_n=(V_n,E_n)$ be the graph with edges $\{v, \sigma_n(s)v\}$ for $v\in V_n$, $s\in S$. The sequence of graphs $\{G_n\}_n$ is an {\bf expander sequence} if there is an $\epsilon>0$ such that for every subset $A \subset V_n$ with $|A_n| \le |V_n|/2$, 
$$ |\partial A_n| \ge \epsilon |A_n|$$
where $\partial A_n$ is the set of edges $e \in E_n$ with one endpoint in $A_n$ and one endpoint not in $A_n$. In this case, we say  $\Sigma$ is {\bf a sofic approximation by expanders}. 


\begin{remark}[Actions on ultraproduct spaces]
The set $V_n$ can be thought of as a probability space endowed with the uniform probability measure. The ultraproduct of the $V_n$'s forms a nonstandard probability measure space on which $\G$ acts by measure-preserving transformations \cite{MR2964622}. If $\Sigma$ is by expanders then the action $\Ga$ on this ultraproduct  is ergodic. On the other hand, if the action is ergodic then there is an equivalent sofic approximation $\Si'$ such that $\Si'$ is by expanders. In this case, $\Sigma$ is said to be an {\bf ergodic sofic approximation} \cite{MR3693125}. 
\end{remark}

\begin{example}\label{ex:expanders}
G. Margulis showed that if $\Ga$ has property (T) and $N_i\vartriangleleft \Ga$ is a decreasing sequence of finite-index normal subgroups of $\Ga$ then the Schreier coset graphs of $\Ga/N_i$ form an expander sequence. Therefore, the sofic approximation $\Sigma=\{\sigma_i: \Ga \to \sym(\Ga/N_i)\}_i$ given by the canonical actions of $\Ga$ on $\Ga/N_i$ by left translation is by expanders. Gabor Kun recently proved every sofic approximation of a Property (T) group is equivalent (in the sense of \S \ref{sec:space}) to one that is a disjoint union of expanders \cite{kun-2016}.
\end{example}

\begin{prop}\label{prop:trivial}
If $\Sigma$ is a sofic approximation by expanders and $(X,\mu)$ is non-trivial (this means: for every  $x\in X$, $\mu(\{x\})<1$) then the trivial action of $\Ga$ on $(X,\mu)$ has $h_{\Sigma,\mu}(\tau_X)=-\infty$. 
\end{prop}

\begin{proof}[Proof sketch]
Let $A \subset X$ have $0<\mu(A)\le 1/2$. If $\phi:V_n \to X$ is a microstate then $\phi^{-1}(A)$ should have cardinality approximately $\mu(A)|V_n|$. Moreover the boundary of $\phi^{-1}(A)$ should have small cardinality relative to $|V_n|$ because $A$ is an invariant set. However this contradicts the expander property. So no such microstates exist. More precisely, $\Map(\tau_X,\rho,\cO,F,\delta,\sigma_n)$ is empty if $F \subset \Ga$ contains a generating set, $\delta>0$ is sufficiently small, $\cO$ is sufficiently small and $n$ is sufficiently large. Here $\rho$ is any metric on $X$. 
\end{proof}



\subsubsection{Diffuse sofic approximations}\label{sec:diffuse}
Next we define a sufficient condition on a sofic approximation implying the trivial action has sofic entropy zero.


\begin{defn}
Let $\Sigma=\{\sigma_n\}$ be a sofic approximation to $\Ga$. The {\bf density} of a sequence $\{A_n\}$ of subsets $A_n \subset V_n$ is 
$$\textrm{density}(\{A_n\}) := \lim_{n \to\infty} \frac{|A_n|}{|V_n|}$$
provided the limit exists.

 A sequence $A_n \subset V_n$ is {\bf asymptotically invariant} if for every $g \in \Ga$ the sequence of symmetric differences $\{A_n \vartriangle \sigma_n(g)A_n\}$ has density zero. If $\{A_n\}$ is asymptotically invariant and has positive density then there exist maps $\sigma'_n: \G \to \sym(A_n)$ satisfying
$$\lim_{n\to\infty} \frac{|\{v\in A_n:~\sigma'_n(g)v = \sigma_n(g)v\}|}{|A_n|}  =1$$
for every $g \in \Ga$. The sequence $\Sigma'=\{\sigma'_n\}$  is a {\bf sub-sofic approximation} of $\Si$. It is well-defined only up to edit distance zero (see \S \ref{sec:space}). It is called {\bf proper} if the density of $\{A_n\}$ is strictly less than 1.  The sequence $\Sigma$ is {\bf diffuse} if every sub-sofic approximation $\Sigma'$ admits a proper sub-sofic approximation. 
\end{defn}

\begin{example}
Sofic approximations can be amplified as follows. Let $\{W_n\}$ be a sequence of finite sets. Given $\Si=\{\s_n\}$ as above, define
$$\s'_n:\G \to \sym(V_n \times W_n)$$
by $\s'_n(g)(v,w)=(\s_n(g)v,w)$. If $|W_n| \to \infty$ as $n\to\infty$ then $\Sigma'=\{\sigma'_n\}$ is a diffuse sofic approximation.
\end{example}

\begin{prop}
If $\Sigma$ is diffuse then the trivial action of $\Ga$ on $(X,\mu)$ has zero $\Si$-entropy.  
\end{prop}

\begin{proof}[Proof sketch]

The upper bound can be obtained directly. To prove the lower bound, let $G$ be the set of all numbers $p\in [0,1]$ such that $p$ is the density of $\{A_n\}$ for some asymptotically invariant $\{A_n\}$.  We claim that $G=[0,1]$. To see this, let $p \in [0,1]$ be arbitrary and let $a = \sup \{x \in G:~x \le p\}$. A diagonalization argument shows $G$ is closed. Therefore, $a\in G$ and there exists asymptotically invariant $\{A_n\}$ with density $a$.  

Another diagonalization argument shows there exists an asymptotically invariant sequence $\{B_n\}$ with $A_n \subset B_n$ and $\textrm{density}(\{B_n\})=b$  where $b \ge p$ is minimal subject to these conditions. Note that $B_n \setminus A_n$ is asymptotically invariant. Because $\Si$ is diffuse, either $a=b$ (in which case $a=p \in G$) or there exists an asymptotically invariant sequence $\{C_n\}$ with $A_n \subset C_n \subset B_n$ and
$$a< \textrm{density}(\{C_n\}) < b.$$ 
This contradicts the choice of $a,b$. So $G=[0,1]$ as claimed.
 
 Now let $\cP=\{P_1,\ldots, P_r\}$ be a finite partition of $X$ with $0<\mu(P_i)<1$ for all $i$. By the claim above there exists  an asymptotically invariant sequence $\{Q_{1,n}\}_n$ with density $\mu(P_1)$. Apply the claim again to the complement of $Q_{1,n}$ in $V_n$ to obtain an asymptotically invariant sequence $\{Q_{2,n}\}_n$ with $Q_{1,n} \cap Q_{2,n}=\emptyset$ and density $\mu(P_2)$. Continue in this fashion to obtain a sequence $\{\cQ_n\}$ of partitions $\cQ_n=\{Q_{1,n},\ldots, Q_{r,n}\}$ such that each sequence $\{Q_{i,n}\}_n$ is asymptotically invariant and has density $\mu(P_i)$. This shows there exist microstates (with respect to the partition definition) for the trivial action of $\Ga$ on $X$ and therefore $h_{\Sigma,\mu}(\tau_X)\ge 0$. 

\end{proof}

\subsubsection{The $f$-invariant}\label{sec:f-trivial}

\begin{thm}
Let $(X,\mu)$ be a probability space with finite Shannon entropy $H(X,\mu)$. Then $f_\mu(\tau_X) =  -(r-1)H(X,\mu).$
 \end{thm}

\begin{proof}
Because $H(X,\mu)<\infty$ we may assume $X$ is countable. Let $\cP$ be the partition into points. Then $f_\mu( \tau_X ) = F_\mu( \tau_X,\cP ) = -(r-1) H_\mu(\cP).$
\end{proof}



\subsection{Bernoulli shifts}\label{sec:bernoulli}

\begin{thm}
For any countably infinite group $\G$, if $(K,\k), (L,\l)$ are probability space with the same Shannon entropy then the corresponding Bernoulli shifts $\G \cc (K,\k)^\G$ and $\G \cc (L,\l)^\G$ are measurably conjugate. In particular, if $\G$ is sofic or $h^{\rm{Rok}}_{sup}(\G)=+\infty$ then Bernoulli shifts over $\G$ are completely classified up to measure-conjugacy by Shannon entropy of the base. 
\end{thm}

The first statement is outlined in \S \ref{sec:isom}. The second statement follows from Theorems \ref{thm:bernoulli-sofic} and \ref{thm:seward-bernoulli}. It is unknown whether the second statement holds for all countably infinite groups.


\subsection{Markov chains}\label{sec:markov-chains}       

Recall the definition of a Markov chain over a free group from \S \ref{sec:f-intro}. 

\begin{thm}\label{thm:markov}
Let $\bfX=(X_g)_{g\in \Ga}$ be a stationary process over the free group $\Ga=\langle s_1,\ldots, s_r\rangle$. Suppose the state space $K$ is countable and $H(X_e) < \infty$. Let $\mu =\Law(\bfX)$ and $\G \cc K^\G$ the shift action. Then
$$f_\mu(\Ga \cc K^\Ga) \le F_\mu(\cP)= H(X_e) + \sum_{s\in S} H(X_e|X_{s}) - H(X_e)$$
where $\cP$ is the canonical ''time 0'' partition of $K^\Ga$. Moreover equality holds if and only if $\bfX$ is Markov.
\end{thm}

\begin{remark}
In the case the rank $r=1$, this formula reduces to the well-known formula for the entropy of a Markov chain as the entropy of the present conditioned on the immediate past: $h( \bfX) = H(X_0 | X_{-1})$. 
\end{remark}

\begin{proof}[Proof sketch of Theorem \ref{thm:markov} (details in \cite{bowen-entropy-2010a})]


Since $f_\mu(\Ga \cc K^\Ga)$ is the infimum of $F_\mu(\cQ)$ over all splittings $\cQ$ of $\cP$, the inequality $\le$ is immediate.  So it suffices to show that if $\cQ$ is a simple splitting of $\cP$ then $F_\mu(\cQ)=F_\mu(\cP)$ if and only if $\bfX$ is Markov in the direction of the simple splitting. This can be achieved by direct computation following the steps in the proof sketch of Theorem \ref{thm:f1}.

\end{proof}

\begin{example}[The Ising Model]\label{ex:ising}
Let $K=\{-1,1\}$. Given $\epsilon>0$ we consider the Markov chain $\bfX=(X_g)_{g \in \Ga}$ over $\Ga=\langle s_1,\ldots, s_r\rangle$ with state space $K$ satisfying: 
$$P(X_e = -1)=P(X_e=1)=1/2$$
$$P(X_e = k | X_s = k) = 1-\epsilon, \quad P(X_e \ne k | X_s = k) = \epsilon$$
for every $k \in K$ and $s \in \{s_1,\ldots, s_r, s_1^{-1},\ldots, s_r^{-1}\}$. The $f$-invariant of this process, denoted $f(\bfX)$, is
$$f(\bfX)= -(r-1)\log(2) - r \epsilon  \log(\epsilon) - r(1-\epsilon)\log(1-\epsilon).$$

In the limiting case $\epsilon=0$, the law of $\bfX$ is equally distributed on only two atoms: all $0$'s and all $1$'s. In this case, $f(\bfX)= -(r-1)\log(2)<0$. From the sofic interpretation of the f-invariant \S \ref{sec:interpretation}, this means that for most homomorphisms $\sigma:\Ga \to \sym(n)$, there are no good models/microstates for this process. This is because the graph $(V,E)$ defined by $V=[n]$ and $E=\{ (v, \sigma(s_i)v):~v\in [n], 1\le i \le r\}$ is typically an expander. As in \S \ref{sec:trivial} expansitivity implies there are no good models for a non-ergodic action.

If $\epsilon$ is positive but small and $r>1$ then the f-invariant is still negative and so it cannot be measurably conjugate to a Bernoulli shift. However the Markov chain is mixing. By contrast, every mixing Markov chain over the integers is isomorphic to a Bernoulli shifts \cite{friedman-ornstein}.
\end{example}


\begin{problem}
Classify mixing Markov chains up to measure-conjugacy. If two mixing Markov chains are spectrally isomorphic and have the same $f$-invariant are they measurably conjugate? 
\end{problem}

While it is straightforward to compute the $f$-invariant for Markov chains, there are no known methods for computing the sofic entropy with respect to a given sofic approximation. In particular the following is open:
\begin{question}
Does the sofic entropy of the Ising model depend on the choice of sofic approximation? To avoid trivialities, we assume the approximation is such that the sofic entropy is non-negative. 
\end{question}

\begin{remark}
Markov chains are used for counterexamples. In \S \ref{sec:CPE-spectral} it is shown that the Ising model with small transitive probabilities is uniformly mixing but does not have completely positive entropy (CPE). By contrast it was shown in \cite{rudolph-weiss-2000} that for $\Z$-actions, CPE and uniformly mixing are equivalent properties. In \S \ref{sec:uniqueness}  a topological Markov chain is constructed that has multiple measures of maximal f-invariant. This is impossible if $\G=\Z$.
\end{remark}

\begin{remark}
Markov chains have also been used to obtain positive general results by using that an arbitrary invariant measure can be approximated (in the weak* sense) by measures that are multi-step Markov chains. These results include: a relative entropy theory for the $f$-invariant \S \ref{sec:relative}, a partial Yuzvinskii's formula \S \ref{sec:algebraic}, a formula for the restriction to a subgroup \S \ref{sec:subgroups} and an ergodic decomposition formula \S \ref{sec:ergodic decomposition}. 
\end{remark}

\begin{example}[Tree Lattices]\label{ex:tree}

Let $T_{2d}$ be the $2d$-regular tree and $\Lambda<\Aut(T_{2d})$ a lattice. Because the free group $\F_d$ of rank $d$ is also a lattice in $\Aut(T_{2d})$, it acts by measure-preserving transformations on the quotient $\Aut(T_{2d})/\Lambda$. In general, it is an open problem to compute either the sofic entropy or the $f$-invariant for these homogeneous actions. However there is at least one special case in which the action is isomorphic to a Markov chain and therefore its $f$-invariant can be computed explicitly. For simplicity, let us assume $d=2$. We consider the case $\Lambda=\F_2=\langle a,b\rangle$. 

We assume $\F_2$ acts simply transitively on the vertices of $T_4$. By fixing a vertex $v_0$, we identify $\F_2$ with the set of vertices via the map $g \mapsto gv_0$. Let $\vec{E}(T_4)$ be the set of directed edges of $T_4$ and $L_0:\vec{E}(T_4) \to S$ the standard labeling: $L_0( (gv_0, gsv_0) ) = s$ for $s\in S$. Note that if $e\in \vec{E}(T_4)$ and $\check{e}$ denotes the same edge with the opposite orientation then $L_0(\check{e})=L_0(e)^{-1}$. 

More generally, a {\bf legal labeling} of $T_4$ is any map $L:\vec{E}(T_4) \to S$ satisfying the following.
\begin{itemize}
\item For $v\in V(T_4)$, let $N^+(v)$ and $N^-(v)$ denote the set of edges directed out and directed into $v$ (respectively). Then $L$ is 1-1 on $N^+(v)$ and on $N^-(v)$.  
\item for every $e\in \vec{E}(T_4)$, $L(\check{e}) = L(e)^{-1}$. 
\end{itemize}
$\Aut(T_4)$ acts on the set of legal labelings, denoted $\cL$, by $gL = L \circ g^{-1}$. This action is transitive and the stabilizer of $L_0$ is $\F_2$. Therefore we can identify $\cL$ with $\Aut(T_4)/\F_2$. 


Let $K=\sym(S)$. For $L \in \cL$, let $x_L \in K^\Ga$ be the map 
$$x_L(g) = L_0 \circ (L\resto N^+(gv_0))^{-1}.$$
It is straightforward to verify that $x_{gL} = gx_L$. So the map $L \mapsto x_L$ gives an embedding of $\cL$ into $K^\Ga$. Moreover it pushes forward the Haar measure on $\Aut(T_4)/\F_2 = \cL$ onto a Markov measure on $K^\Ga$ denoted by $\mu$. If $\bfX=(X_g)_{g\in \Ga}$ is a stationary process with law $\mu$ then for any $s\in S$, the pair $X_e$ and $X_s$ is uniformly distributed over the set of pairs of permutations $(\pi_1,\pi_2)\in K \times K$ with the property that if $t \in S$ is such that $\pi_1(t)=s$ then $\pi_2(t^{-1})=s^{-1}$. The number of such pairs is $4! 3!$. So $H(X_e,X_s)=\log(4!3!)$. Moreover $X_e$ is uniformly distributed over $K$. So $H(X_e)=\log(4!)$. So
$$f(\F_2 \cc \Aut(T_4)/\F_2) = f(\bfX) = -3\log(4!) + 2\log(4!3!) = \log(3/2)>0.$$
It is an open problem whether $\F_2 \cc \Aut(T_4)/\F_2$ is isomorphic to a Bernoulli shift or to a factor of Bernoulli shift. 
\end{example}

\subsection{Algebraic dynamics}\label{sec:algebraic}      

Let $X$ denote a compact group and $\Aut(X)$ the group of all automorphisms of $X$. Naturally, any automorphism preserves the Haar measure $\mu$. Therefore any homomorphism $\G \to \Aut(X)$ induces a measure-preserving action $\G \cc (X,\mu)$. The goal of algebraic dynamics is to relate dynamical properties of the action $\G \cc X$ to algebraic and analytic properties of the homomorphism $\G \to \Aut(X)$. Such actions have been explored in great detail when $\G=\Z^d$ (see the book \cite{schmidt-book}). 

A general procedure for computing the entropy of algebraic $\Z$-actions was obtained by Yuzvinskii in the 1960s \cite{MR0201560, MR0194588} and extended to $\Z^d$ by Lind-Schmidt-Ward \cite{LSW-1990}. This procedure rests on Yuzvinskii's addition formula and the special case of principal algebraic actions. These are explained next along with (partial) results in the non-amenable case. We also present recent results on: Pinsker algebras of algebraic actions, CPE, mixing properties and the coincidence of measure entropy and topological entropy. This area is rapidly developing!


\subsubsection{The case $\G=\Z$}

This section is meant to motivate the results for more general groups by explaining the case $\G=\Z$ in detail.

\begin{thm}
Suppose
$$1 \to Y \to X \to X/Y \to 1$$
is an exact sequence of compact metrizable groups and $T_X \in \Aut(X)$ is an automorphism that leaves $Y$ invariant. Then
$$h(T_X) = h(T_Y) + h(T_{X/Y})$$
where $T_Y \in \Aut(Y), T_{X/Y} \in \Aut(X/Y)$ are the induced automorphisms and the entropy is either topological or with respect to Haar measure (these two cases coincide).
\end{thm}
This theorem was first proven by Yuzvinskii \cite{MR0194588}. 

We will apply the above theorem to principal algebraic actions. If $f = \sum_{i=0}^s c_i x^i \in \Z[x]$ is a polynomial and $z \in \T^\Z$ (where $\T=\R/\Z$ is the 1-torus as an additive group) then the convolution $f*z \in \T^\Z$ is defined by
$$(f*z)_n = \sum_{m \in \Z} c_{m} z_{n-m}$$
where we set $c_m=0$ if $m<0$ or $m>s$. Let 
$$X_f = \{ (z_i) \in \T^\Z:~ f^* *z =0\}.$$
Note that $X_f$ is a compact abelian group and the shift action $T_f:X_f \to X_f$, $T_f(x)_i = x_{i+1}$ is an automorphism. 

If $f,g \in \Z[x]$ are non-zero polynomials then there is an exact sequence
$$0 \to X_f \to X_{gf} \to X_g \to 0$$
where the map $X_{gf} \to X_g$ is convolution with $f$. So Yuzvinskii's addition formula implies $h(T_{fg}) = h(T_f) + h(T_g)$. In other words, the map $f \mapsto \exp(h(T_f))$ is multiplicative. There are only a few multiplicative functions on polynomials. For example, the leading coefficient is one. The product of all roots contained in some fixed subset of $\C$ is another. So perhaps it is not too suprising that 
$$h(T_f) = \log |c_s| + \sum_{i=1}^s \max(0, \log |r_i|)$$
where $c_s$ is the leading coefficient of $f$ and $r_1,\ldots, r_s$ are its roots.

This formula is most easily confirmed in the special case that $c_s=1$ and $c_0=\pm 1$ (where $c_0$ is the constant term of $f$). In this case the map 
$$(x_i) \in X_f \mapsto (x_0,\ldots, x_{s-1})\in \T^s$$
induces a measure-conjugacy between the shift $T_f$ and the linear map
$$(x_0,\ldots, x_{s-1}) \mapsto \left(x_1,\ldots, x_{s-1}, - \sum_{i=0}^{s-1} c_ix_i\right).$$
 The latter map has entropy equal to $\sum_{i=1}^s \max(0, \log |\l_i|)$ where $\l_1,\ldots, \l_s$ are its eigenvalues (by Pesin's entropy formula). These eigenvalues are exactly the roots of $f$ because $f$ is its characteristic polynomial.
 
 Using Jensen's formula, we can rewrite the above formula as 
 $$h(T_f) = \int_0^1~\log | f(\exp(2\pi i x)) | ~dx.$$
We will now show how to generalize the formula on the right to arbitrary countable groups.

 We view $f$ as an element of the group ring $\C \Z$ (by identifying $f$ with $\sum_n c_n n$). We can view $\C\Z$ as a subring of $B(\ell^2(\Z))$, the algebra of bounded operators on $\ell^2(\Z)$, via
 $$\sum_n c_n n \mapsto c_n \sigma^n$$
 where $\sum_n c_n n \in \C\Z$ is a formal sum and $\s: \ell^2(\Z) \to \ell^2(\Z)$ is the shift-operator $\s(x)_i=x_{i-1}$. The von Neumann algebra of $\Z$, denoted $L\Z$, is the weak operator closure of $\C\Z$ in $B(\ell^2(\Z))$. The trace on $L\Z$ is defined by
 $$\rm{tr}(x) = \langle x \delta_0, \delta_0\rangle$$
 where $\delta_0\in \ell^2(\Z)$ is the Dirac function supported on $0 \in \Z$.  
 
The Fourier transform $\cF:\ell^2(\Z) \to L^2(\T)$ is the continuous linear map such that $\cF(\delta_n)$ is the function $z \mapsto z^n \in \T$ where $\T= \{z \in \C:~|z|=1\}$ is the unit circle. This is an isomorphism. Next we show that $\cF$ conjugates $L\Z$ to $L^\infty(\T)$.
 
 For $\phi \in L^\infty(\T)$, define the multiplication operator $M_\phi:L^2(\T) \to L^2(\T)$ by
 $$M_\phi(f)(z) = \phi(z) f(z).$$
 The map $\phi \mapsto M_\phi$ embeds $L^\infty(\T)$ into algebra of bounded operators $B(L^2(\T))$. Moreover, $\cF \sigma \cF^{-1}=M_z$. So $\cF (\C\Z )\cF^{-1}$ is the algebra of Laurent polynomials. Its weak operator closure is $L^\infty(\T)$. Thus  $\cF L\Z \cF^{-1}$ is naturally identified with $L^\infty(\T)$.  Moreover for any $x \in L\Z$, the trace of $x$, $\rm{tr}(x)$ equals the integral of $\cF x \cF^{-1}$ over $\T$. In particular, if $f \in \Z[x]$ is a polynomial, then
 $$\int_\T \log |f(z)|~dz = \rm{tr}(\log |f|)$$
 where $\log |f| \in L\Z$ is defined via spectral calculus. The {\bf Fuglede-Kadison determinant} of $f$ is defined by
 $$\det(f) := \exp(\rm{tr}(\log |f|)) = \exp \left( \int_{[0,\infty)} \log(t)~d\zeta_{|f|}(t) \right)$$
 where $\zeta_{|f|}$ is the spectral measure of $|f|$. If $|f|$ has a nontrivial kernel then the determinant is 0. So we modify the definition slightly: the {\bf positive Fuglede-Kadison determinant} of $f$ is 
 $$\rm{det}^+(f) := \exp \left( \int_{(0,\infty)} \log(t)~d\zeta_{|f|}(t) \right).$$
In this way we are naturally led to the conjecture that, for general groups $\G$ and $f \in \Z\G$, the entropy of $\G \cc X_f$ should be the logarithmic positive Fuglede-Kadison determinant of $f$. 

\subsubsection{Group rings}
To discuss algebraic actions, we will need some background on certain group rings. Let $\C \Ga$ denote the complex group ring of $\Ga$. Formally, $\C \Ga$ is the set of all sums $\sum_{g\in \Ga} c_g g $ with $c_g \in \C$ such that all but finitely many of the $c_g$'s are zero. Addition, multiplication and the adjoint operator are given by
$$\sum_{g\in \Ga} c_g g + \sum_{g\in \Ga} c'_g g = \sum_{g\in \Ga} (c_g +c'_g) g$$
$$\left(\sum_{g\in \Ga} c_g g\right) \left(\sum_{g\in \Ga} c'_g g\right) = \sum_{g \in \Ga} \sum_{h\in \Ga}  c_g c'_h gh.$$
$$\left(\sum_{g\in \Ga} c_g g\right)^* = \sum_{g\in \Ga} \overline{c_g} g^{-1}$$
where $\overline{c_g}$ is the complex-conjugate of $c_g$. 

We view $\Z\G$ as a subring of $\C\G$. Also consider $\T^\G$ where $\T=\R/\Z$ is the additive group of the circle. If $x\in \T^\G$ and $f =\sum_g f_g g \in \Z\G$ then $fx, xf \in \T^\G$ are well-defined via:
$$(fx)_g = \sum_{h\in \G} f_{gh} x_{h^{-1}}, \quad (xf)_g = \sum_{h\in \G} x_{gh} f_{h^{-1}}.$$
These are finite sums. The inner product of $f$ with $x$ is defined by
$$\langle f,x \rangle = \sum_g f_g x_g \in \T.$$
Note
$$\langle af,x \rangle = \langle a,xf^* \rangle =  \langle f,a^*x \rangle.$$
(By linearity, it suffices to check this formula when each of $a,x,f$ is supported on a single element). This inner product allows us to identify $\Hom(\Z\G,\T)$ with $\T^\G$.

\subsubsection{Principal algebraic actions}


There is a simple procedure for associating to any $f \in \Z\G$ a dynamical system. First consider the left ideal  $\Z \Ga f$ and the associated quotient $\Z \Ga/\Z \Ga f$. We consider the latter as a countable abelian group on which $\Ga$ acts by automorphisms (namely, left multiplication). Let 
$$X_f = \widehat{\Z \Ga/\Z \Ga f} = \Hom(\Z \Ga/\Z \Ga f, \T)$$
be the Pontryagin dual. By identifying $\Hom(\Z\Ga, \T)$ with $\T^\G$ as above, we can identify $X_f$ with the subgroup
$$X_f = \{ x\in \T^\G:~ x  f^* = 0 \}.$$

For $g\in \G$, let $T_f^g: X_f \to X_f$ be the automorphism  $ (T^g_fx)( p) = x( g^{-1}p)$. Then $T_f=(T^g_f)_{g \in \G}$ is an action on $X_f$ by automorphisms. It preserves the Haar measure, which we denote by $\mu_f$. The action $T_f$ is called a {\bf principal algebraic action}. 

When $\G=\Z^d$, we identify $\Z \G$ with the ring $\Z[u_1^{\pm 1}, \ldots, u_d^{\pm 1}]$ of Laurent polynomials. This way we may view $f \in \Z\G$ as a polynomial function $f:\C^d \to \C$. With this identification, Lind-Schmidt-Ward proved in \cite{LSW-1990} that for non-zero $f$,
\begin{eqnarray}\label{eqn:mahler}
h(T_f) = h_{\mu_f}(T_f) = \int_{\T^d} \log |f(e^{2\pi i \theta})|~d\theta
\end{eqnarray}
where $\T^d=\{(z_1,\ldots, z_d) \in \C^d:~|z_i|=1~\forall i\}$ denotes the $d$-dimensional torus. This extends earlier work of Yuzvinskii in the case $d=1$ \cite{MR0201560, MR0194588}. The right hand side of the equation above is the log-Mahler measure of $f$. 

Christopher Deninger noted that there is a generalization of the Mahler measure to non-abelian $\G$ known as the Fuglede-Kadison determinant \cite{deninger-2006} and conjectured  that, for amenable $\G$, the entropy of $\G \cc X_f$ equals the log of the Fuglede-Kadison determinant. Special cases were confirmed in \cite{deninger-2006, deninger-schmidt, Li-annals-2012} before the general amenable group case was handled in \cite{MR3110799}. The case of expansive principal algebraic actions of residually finite groups was handled in \cite{MR2794944} and extended in \cite{MR2956925} to some non-expansive actions. Then in a stunning breakthrough Ben Hayes obtained the most general result for sofic groups:

\begin{thm}\cite[Theorem 1.1]{hayes-fk-determinants}
Let $\G$ be a sofic group with sofic approximation $\Si$. Let $f\in \Z\G$ and let $\rm{det}^+(f)=\exp(\int_{(0,\infty)} \log(t)~d\zeta_{|f|}(t))$ denote the positive Fuglede-Kadison determinant of $f$ where $\zeta_{|f|}$ denotes the spectral measure of $|f|=(f^*f)^{1/2}$ . Then
\begin{enumerate}
\item $h_\Si(\G \cc X_f)<+\infty$ if and only if $f$ is injective as a convolution operator on $\ell^2(\G)$.
\item If $f$ is injective as an convolution operator on $\ell^2(\G)$, then 
$$h_\Si(\G \cc X_f) = h_{\Si,\mu_f}(\G \cc X_f) = \log  \rm{det} ^+(f).$$
\end{enumerate}
\end{thm}

\begin{remark}
The paper \cite{hayes-fk-determinants} also handles the more general case in which $f$ is a finite-dimensional matrix over $\Z\G$, $I_f \subset \Z(\G)^{\oplus n}$ is its range and $X_f$ is the Pontryagin dual of $\Z(\G)^{\oplus n}/I_f$. 
\end{remark}

\begin{question}
Does the Rokhlin entropy of a principal algebraic action equal the logarithm of the Fuglede-Kadison determinant? 
\end{question}

\subsubsection{Yuzvinskii's addition formula}\label{sec:yuz}

Let $G$ be a compact metrizable group and $\G \cc G$ an action by continuous automorphisms. Suppose that $N \vartriangleleft G$ is a closed $\G$-invariant normal subgroup. We will compare the entropy of $\G \cc G$ with that of the restricted action $\G \cc N$ and the induced action $\G \cc G/N$. More precisely, we say the {\bf addition formula} holds for $(\G \cc G,N)$ if the entropy of $\G \cc G$ equals the sum of the entropy of $\G \cc N$ with the entropy of $\G \cc G/N$. In general, this might depend on which entropy is intended. 

In \cite{MR0194588} Yuzvinskii proved the addition formula for $G=\Z$ with respect to both topological and measure entropy. The proof has been extended to more general kinds of skew-products \cite{MR0293064}, to $\Z^d$  \cite{LSW-1990}, various other amenable groups \cite{MR2533986, MR2457481} and to arbitrary amenable groups \cite{Li-annals-2012} (and independently in unpublished work of Lind-Schmidt). Using this, Li-Thom obtained a general procedure for computing the entropy  of algebraic actions of amenable groups on compact abelian groups \cite{MR3110799} under very mild conditions. The formula shows that entropy can be viewed as $L^2$-torsion.

However it fails for non-amenable groups. Indeed the Ornstein-Weiss example (\S \ref{sec:OW-intro}) is algebraic. Recall that $\F_2=\langle a,b\rangle$ acts on the compact abelian group $G=(\Z/2)^{\F_2}$ with invariant normal subgroup $N \cong \Z/2$ consisting of the constants and quotient $G/N \cong G \times G$. With respect to any sofic approximation, the entropy of $\F_2 \cc G$ is $\log(2)$, the entropy of $\F_2 \cc N$ is $0$ or $-\infty$ and the entropy of $\F_2 \cc G/N$ is $\log(4)$. Rokhlin entropy behaves similarly. Recent work of Bartholdi \cite{bartholdi-kielak} shows that for any non-amenable $\G$ and field $K$ there are $n \in \N$ and injective $K\G$-module homomorphism $K\G^{\oplus (n+1)} \to \K\G^{\oplus n}$. Consider the case when $K$ is finite. Then the action $\G \cc \widehat{K\G^{\oplus n}}$ has entropy $n\log |K|$ while the $\G$-action on the quotient group $\G \cc \widehat{K\G^{\oplus (n+1)}}$ has entropy $(n+1)\log|K|$. Thus the addition formula fails.\footnote{Thanks to an anonymous reviewer for pointing this out.}


In a different direction, Gaboriau and Seward consider the following construction: let $\G$ be a finitely generated group, $\K$ a finite field, $\G \cc \K^\G$ the shift action, $\K \le \K^\G$ the constants. Then:
\begin{thm}\cite{gaboriau-seward-2015}
For any sofic approximation $\Si$,
$$(1 + \beta^1_{(2)}(\G)) \log |\K| \le h_\Si(\G \cc \K^\G/\K) = h_{\Si,\mu}(\G \cc \K^\G/\K) \le \cC_{\sup}(\G) \log |\K| $$
where $\mu$ denotes Haar probability measure on $\K^\G/\K$, $\beta^1_{(2)}(\G)$ is the first $L^2$-Betti number of $\G$ and $\cC_{\sup}(\G)$ is the sup cost of $\G$.
\end{thm}
Conjecturally, $1 + \beta^1_{(2)}(\G) = \cC_{\sup}(\G)$. In spite of these negative results, the following remains open: 
\begin{problem}
Does the $f$-invariant satisfy an addition formula?
\end{problem}
In \cite{bowen-entropy-2010a} a positive answer is claimed. However, there is a serious flaw. It was partially corrected in \cite{bowen-gutman-2014} which proved that indeed the $f$-invariant does satisfy an addition formula whenever $X$ is totally disconnected and there exists a special kind of generating partition. When $X$ is a connected finite-dimensional Lie group, the $f$-invariant of $\G \cc X$ is minus infinity and an addition formula also holds for degenerate reasons.

\begin{problem}
Is there any formula or algorithm for computing the sofic entropy of $\G \cc X_I$ where $I \subset \Z\G$ is a finitely generated ideal and $X_I=\widehat{\Z\G/I}$ is the Pontryagin dual? Such a formula is known when $\G$ is amenable \cite{MR3110799} but unknown for non-amenable groups.
\end{problem}

\subsubsection{Further results}

In recent stunning work, Ben Hayes has shown that for any algebraic action $\G \cc X$ of a sofic group, there is a closed $\G$-invariant normal subgroup $Y \le X$ such that the outer Pinsker algebra is the sigma-algebra of $Y$-invariant Borel subsets (whenever the action satisfies the mild condition of  admitting an lde-convergent sequence of model measures) \cite{hayes-relative-entropy}. Outer Pinsker algebras are defined in \S \ref{sec:outer} below. The proof uses the product formula for outer Pinsker algebras (Theorem \ref{thm:pinsker-sofic}). 

It follows that if  $\G \cc X/Y_0$ has positive outer entropy for every closed normal $\G$-invariant subgroup $Y_0\le X$ with $Y_0 \ne X$, then $\G \cc X$ has completely positive outer entropy. In \cite{MR3641841} Ben Hayes uses local sofic entropy theory (developed in \cite{kerr-li-comb-ind}) to show that if $f \in \Z\G$ is invertible in the group von Neumann algebra $L\G$ but not in $\Z\G$ then every $k$-tuple of points in $X_f$ is a $\Si$-IE-$k$-tuple. This implies positive sofic entropy. So the Fuglede-Kadison determinant of $f$ is $>1$ (answering a question of Deninger) and  the $\Si$-entropy of $\G \cc X_f$ is positive. Moreover, in \cite{hayes-relative-entropy} it is shown how this implies completely positive outer entropy (assuming the same mild condition as above). In \cite{MR3635672} it is shown that completely positive outer entropy implies the Koopman representation embeds into the countable sum of left-regular representations. 

\begin{question}
If $f \in \Z\G$ is invertible in the group von Neumann algebra $L\G$ but not in $\Z\G$ is the action $\G \cc (X_f,\mu_f)$ orbit-equivalent  to a Bernoulli shift? Is it measurably conjugate to a Bernoulli shift? The latter is open even in the special case of amenable $\G$. It is known to be true when $\G=\Z^d$ \cite{schmidt-book}. (More precisely, one should consider $\G/N \cc (X_f,\mu_f)$ where $N$ is the kernel of $\G \cc X_f$. This kernel is finite and therefore trivial if $\G$ is torsion-free.)\end{question}


Finally, Ben Hayes has used Tim Austin's lde-$\Si$-entropy to prove that the topological $\Si$-entropy of $\G \cc X$ agrees with the measure $\Si$-entropy whenever the action satisfies the above-mentioned mild condition \cite{hayes-doubly-quenched}. The case in which $\G$ is amenable was handled earlier in \cite[Theorem 3]{deninger-2006}.

\begin{question}
Suppose $\Ga$ is sofic, $\Sigma$ is a sofic approximation to $\Ga$ and $f \in \Z \Ga$ is such that the action $\G \cc (X_f,\mu_f)$ has completely positive $\Si$-entropy (abbreviated CPE$^\Si$). Is the Haar measure on $X_f$ the unique measure of maximal $\Sigma$-entropy? By \cite[Theorem 8.6]{MR3314515}, if $\G$ is amenable, then Haar measure is the unique measure of maximal entropy if and only if $\G \cc (X_f,\mu_f)$ is CPE. 
\end{question}

\subsection{Gaussian actions}\label{sec:gaussian}       

Associated to any orthogonal representation $\rho:\Ga \to \cO(\cH)$ on a real Hilbert space $\cH$ is a {\bf Gaussian action} $\Ga \cc (X_\rho,\mu_\rho)$. The details of this construction can be found in \cite{Kechris-global-aspects} for example. For intuition, if $\cH$ is finite-dimensional then the Gaussian action is the action of $\G$ on $\cH$ with respect to the standard Gaussian measure. Ben Hayes computed the entropy of Gaussian actions in \cite{MR3693125}:

\begin{thm}
The representation $\rho$ decomposes as $\rho=\rho_1 \oplus \rho_2$ where $\rho_1$ is singular with respect to the left-regular representation (so no nontrivial subrepresentation of $\rho_1$ embeds into the left regular representation) and $\rho_2$ embeds into the countable power of the left-regular representation. Moreover,
\begin{displaymath}
h_{\Sigma,\mu_\rho}(\Ga \cc X_\rho) = \left\{\begin{array}{cc}
-\infty & \textrm{ if } h_{\Sigma,\mu_{\rho_1}}(\Ga \cc X_{\rho_1}) = -\infty \\
0 &  \textrm{ if } \rho_2 = 0 \textrm{ and } h_{\Sigma,\mu_{\rho_1}}(\Ga \cc X_{\rho_1}) \ne -\infty \\
+\infty & \textrm{ if } \rho_2 \ne 0 \textrm{ and } h_{\Sigma,\mu_{\rho_1}}(\Ga \cc X_{\rho_1}) \ne -\infty.
\end{array}\right.\end{displaymath}
\end{thm}

The proof uses a Polish model for the action (this is a continuous action $\G \cc Y$ where $Y$ is a completely metrizable separable space with a measure $\nu$ that is measurably conjugate to the original action). It also uses a weak form of Sinai's factor theorem \cite{MR3773269}: if the Koopman representation of $\G \cc (X,\mu)$ is singular to the left regular representation then the sofic entropy is nonpositive. To see the connection, let  $\s:\G \to \sym(V)$ be given. Then there is an induced map from $\G$ into the unitary group of $\C^V$. If $\s$ is regarded as a sofic approximation to $\G$, then this map approximates the left regular representation.

\subsection{Distal actions}\label{sec:distal}    
\begin{defn}
An action $\Ga\cc^T (X,d)$ by homeomorphisms on a compact metric space is {\bf distal} if $\inf_{g\in \Ga} d(T^gx,T^gy)>0$ for every pair of distinct points $x,y \in X$. Note that profinite, compact and equicontinuous actions are distal \cite[Lemma 4.1]{chung-zhang-expansive}.  A pmp action $\Ga \cc^T (X,\mu)$ is {\bf measure distal} if it is measurably-conjugate to a distal action. 

 \end{defn}

\begin{thm}\cite{kerr-li-comb-ind, MR3635672}
Measure distal actions have zero Rokhlin entropy. Therefore they have  non-positive sofic entropy. Also, distal actions have non-positive topological sofic entropy with respect to every sofic approximation.
\end{thm}

\begin{proof}[Remarks on the proof]
The paper \cite{kerr-li-comb-ind} shows that if the topological sofic entropy of $\G \cc^T X$ is positive (with respect to some sofic approximation) then the action is {\bf Li-Yorke chaotic}. This condition means there is an uncountable subset $Z \subset X$ such that every non-diagonal pair $(x,y) \in Z\times Z$ satisfies
$$\limsup_{G\ni s \to \infty} d(T^sx,T^sy)>0, \quad \liminf_{G\ni s \to \infty} d(T^sx,T^sy)=0.$$
It follows that distal actions have non-positive topological sofic entropy.

\end{proof}
The special case of topological $\Z$-actions was handled earlier by Keynes \cite{MR0274710}. Recall that naive entropy is an upper  bound for sofic entropy. This motivates:
\begin{question}
Do distal actions have zero topological naive entropy? The measure version of this question is also open.
\end{question}
Partial progress has been made by Peter Burton:

\begin{thm}\cite{MR3649602}
If $\Ga$ contains an infinite cyclic subgroup then every distal action has zero topological naive entropy (and therefore non-positive topological sofic entropy).
\end{thm}

\subsection{Smooth actions}\label{sec:smooth}       

The purpose of this section is to show that, under mild conditions, the entropy of an action of a large group by diffeomorphisms on a smooth manifold is nonpositive. The starting point is a bound on the topological entropy of a single Lipschitz map. So let $(X,d)$ be a compact metric space. The {\bf ball dimension} of $(X,d)$ is 
$$\dim_{\rm{ball}}(X,d) = \limsup_{\epsilon \searrow 0} \frac{\log S_\epsilon(X,d) }{|\log \epsilon|}$$
where $S_\epsilon(X,d)$ is the minimum cardinality of an $\epsilon$-spanning subset of $(X,d)$. Given any map $T:X \to X$, the {\bf Lipschitz contant} $\rm{Lip}(T)$ is 
$$\rm{Lip}(T) := \sup_{x \ne y} \frac{ d(Tx,Ty)}{d(x,y)}.$$

\begin{lem}
Let $(X,d)$ be a compact metric space and $T: X \to X$ a continuous map. If $\dim_{\rm{ball}}(X,d)<\infty$ and $\rm{Lip}(T)<\infty$ then
$$h_{\rm{top}}(T) \le \dim_{\rm{ball}}(X,d) \max (0, \log \rm{Lip}(T) )$$
where $h_{\rm{top}}(T)$ is the topological entropy of $T$. In particular, if $X$ is a smooth Riemannian manifold and $T$ is a diffeomorphism then $h_{\rm{top}}(T)<\infty$. 
\end{lem}
\begin{proof}
This is \cite[Theorem 3.2.9]{katok-hasselblatt}.
\end{proof}


The next result shows that if the induced action of an infinite index amenable subgroup has finite entropy then the action of the group has nonpositive entropy.
\begin{lem}
Let $\Ga$ be a group with an infinite index amenable subgroup $\Lambda \le \Ga$. Let $\Ga \cc^T X$ be a continuous action on a compact metrizable space. Also let $\mu$ be an invariant probability measure on $X$ that is ergodic with respect to $T$. Let $T_\La$ denote the restriction of the action to $\La$. Then
$$h_{\rm{top}}(T_\La) < \infty \Rightarrow h_\Si(T) \le 0 = h^{\rm{naive}}(T),$$
$$h_{\mu}(T_\La) < \infty \Rightarrow h_{\Si,\mu}(T) \le 0 = h^{\rm{naive}}_\mu(T)$$
for any sofic approximation $\Si$ to $\Ga$. 
\end{lem}

\begin{proof}

The proofs in the topological and measure settings are similar, so we will just give the proof in the measure setting. Let $\cP$ be any partition of $X$ with finite Shannon entropy. Let $F \subset \Ga$ be any finite set of coset representatives of $\Ga/\Lambda$. So if $f_1 \ne f_2$ with $f_1,f_2 \in F$ then $f_1\Lambda \cap f_2 \Lambda = \emptyset$. Let $K \subset \Lambda$ be an arbitrary finite set. Then
$$\lim_{K \nearrow \Lambda} \frac{1}{|K|} H_\mu(\cP^{FK}) \le h_\mu(T_\Lambda)<\infty$$
where the limit is along any F\o lner sequence for $\Lambda$. So
\begin{eqnarray*}
h^{\rm{naive}}_\mu(T,\cP) &\le & \lim_{K \nearrow \Lambda} \frac{1}{|FK|} H_\mu(\cP^{FK}) = \frac{1}{|F|} h_\mu(T_\Lambda).
\end{eqnarray*}
Taking the supremum over all $\cP$ proves $h^{\rm{naive}}_\mu(T) \le \frac{1}{|F|} h_\mu(T_\Lambda)$. Since $|F|$ is arbitrary, this implies $h^{\rm{naive}}_\mu(T) =  0$. The rest follows from Propositions \ref{prop:naive-Rokhlin2} and \ref{prop:Rokhlinsofic}.

\end{proof}

\begin{thm}
Suppose $\Ga$ is a countable group with an infinite cyclic subgroup of infinite index. If $\Ga \cc X$ is a continuous action by Lipschitz maps on a compact metric space $(X,d)$ that has finite ball dimension then $\Ga \cc X$ has non-positive sofic entropy with respect to every sofic approximation.  In particular, smooth actions on manifolds have non-positive sofic entropy.
\end{thm}

\begin{proof}
This follows immediately from the previous two lemmas.
\end{proof}
There is a similar result for the $f$-invariant \cite[Lemma 3.5]{bowen-gutman-2014}.

\subsection{Nonfree actions}\label{sec:nonfree}    

Given an action $\Ga\cc^T (X,\mu)$ and $x\in X$, let $\Stab_T(x)=\{g\in \Ga:~T^gx=x\}$ be the stabilizer. An action is {\bf non-free} if there is a positive measure set of $x$ such that $\Stab_T(x)$ is non-trivial. There are two results concerning the entropy of non-free actions:

\begin{thm}
Suppose $\Ga \cc^T (X,\mu)$ is an ergodic $\Ga$-action with positive sofic entropy with respect to some sofic approximation. Then $\Stab_T(x)$ is finite for a.e. $x$.
\end{thm}

This result is \cite[Theorem 2.3]{MR3550271}. The special case of amenable groups is handled in a remark in the last section of \cite{weiss-actions-of-amenable-groups}. 


\begin{thm}\cite[Theorem 1.1]{MR3540601}
Let $\F \cc^T (X,\mu)$ be a pmp action of a finitely generated free group. Suppose the action has a finite-entropy generating partition, $f_\mu(T)\ne -\infty$ and $\F \cc^S (Y,\nu)$ is a factor action. Then for $\nu$-a.e. $y\in Y$ either
\begin{itemize}
\item $\Stab_S(y)$ is trivial or
\item $\Stab_S(y)$ has finite-index in $\Ga$ and $y$ is an atom (i.e. $\nu(\{y\})>0$).
\end{itemize}
\end{thm}

An anonymous reviewer pointed out that there are ergodic actions $\G \cc^T (X,\mu)$ with positive Rokhlin entropy such that $\Stab_T(x)$ is infinite for a.e. $x$.  For example, let $\Z \cc^S (X,\mu)$ be any ergodic action with positive Rokhlin entropy. Thinking of $\Z$ as a quotient of $\Z^2$, we obtain an ergodic action $\Z^2 \cc^T (X,\mu)$. Moreover, the Rokhlin entropy of this action is the same as the Rokhlin entropy of $\Z \cc^S (X,\mu)$ since the two actions have the same generating partitions. 

\begin{question}
If $\Ga \cc^T (X,\mu)$ is ergodic and has positive naive entropy (topological or measure) then is $\Stab_T(x)$ finite for a.e. $x$? 
\end{question}

\section{Perturbations}\label{sec:perturbations}       

 \subsection{Perturbing the sofic approximation}\label{sec:dependence}       
 
 When $\Ga$ is amenable the sofic entropy agrees with classical entropy and therefore does not depend on the choice of sofic approximation. This is not true in the non-amenable case even when the action is trivial (\S \ref{sec:trivial}). This subsection provides another explicit counterexample which works for both topological and measure entropy. However, the example is not entirely satisfying because it leaves open a major problem:
 
 \begin{question}
 Suppose $\Ga \cc X$ is a continuous action on a compact metrizable space by homeomorphisms. Let $\Sigma_1,\Sigma_2$ be two sofic approximations of $\Ga$.  Suppose $h_{\Sigma_i}(\Ga \cc X)$ is not minus infinity for $i=1,2$. Is it true that $h_{\Sigma_1}(\Ga \cc X) = h_{\Sigma_2}(\Ga \cc X)$? The measure entropy version of this question is also open. In fact, it is open even in the special case of the Ising model on the free group (\S \ref{sec:markov-chains}). It is possible that if the measure sofic entropy of an action is not minus infinity then it must equal the Rokhlin entropy.
 \end{question}
 
  \begin{thm}
 There exists a sofic group $\Ga$ with sofic approximations $\Sigma_1,\Sigma_2$ and a continuous action $\Ga \cc X$ on a compact metric space such that 
 $$h_{\Sigma_1}(\Ga \cc X) \ne h_{\Sigma_2}(\Ga \cc X).$$
 Moreover there exists an invariant probability measure $\mu$ on $X$ such that
  $$h_{\Sigma_1,\mu}(\Ga \cc X) \ne h_{\Sigma_2,\mu}(\Ga \cc X).$$
 \end{thm}
 
 \begin{proof}
Let $\F_2=\langle a,b \rangle$ be the rank 2 free group and let $\F_2 \cc \Z/2$ be the action in which each generator in $\{a,b\}$ acts non-trivially. Recall that if $\sigma:\F_2 \to \sym(V)$ is any map then the associated graph $G_\sigma$ has vertex set $V$ and edges $\{v,\sigma(a)v\}, \{v,\sigma(b)v\}$ for $v\in V$. 

Let $\Sigma_1=\{\sigma_{1,n}\}_{n\in \N}$ be a sofic approximation whose associated graphs are far from bipartite. To be precise, we require the existence of an $\epsilon>0$ such that if $\cP=\{P_1,P_2\}$ is a partition of $V_n$ with 
$$|P_1| \ge (1-\epsilon)|V_n|/2, \quad |P_2| \ge (1-\epsilon)|V_n|/2$$
then the number of edges $\{v,w\}$ of the associated graph $\Gamma_{\sigma_{1,n}}$ with either $v,w \in P_1$ or $v,w \in P_2$ is at least $\epsilon |V_n|$. For example, we could choose $\sigma_{1,n}:\F_2 \to \sym(V_n)$ uniformly at random (among all homomorphisms from $\F_2$ to $\sym(V_n)$). With probability 1 the resulting sofic approximation satisfies the above property. This can be proven using the $f$-invariant for example or directly with combinatorial estimates.

The strong non-bipartiteness of the graphs of $\Sigma_1$ immediately implies $h_{\Sigma_1}(\Ga \cc X) = -\infty.$ On the other hand, if $\Sigma_2$ is a sofic approximation whose associated graphs are bipartite, then $h_{\Sigma_2}(\Ga \cc X) \ge 0$ since the bipartition of the graphs yields two microstates for the action. In fact, $h_{\Sigma_2}(\Ga \cc X) = 0$ since any microstate must be close to one of the two bipartitions of the graphs. If $\mu$ is the unique invariant probability on $\Z/2$ then 
$h_{\Sigma_i,\mu}(\Ga \cc X)  = h_{\Sigma_i}(\Ga \cc X)$
for $i=1,2$ which proves the last statement.
\end{proof}

\begin{problem}
Does the sofic entropy of either a topological action or a measure-preserving action vary upper semi-continuously with respect to the edit distance on sofic approximations (\S \ref{sec:space})? This seems likely but it has not been worked out. It is unknown whether entropy varies continuously (excluding the value negative infinity).
\end{problem}




\subsection{Subgroups}\label{sec:subgroups}       

It is well-known that if $T$ is an automorphism of $(X,\mu)$ and $n \in \Z-\{0\}$ then 
$$h_\mu(T^n) =|n| h_\mu(T).$$
More generally, if $\Ga$ is an amenable group, $\Ga \cc (X,\mu)$ is a pmp action and $\Ga' \le \Ga$ has finite index then
$$h_\mu(\Ga' \cc X) = [\Ga:\Ga'] h_\mu(\Ga \cc X).$$
See \cite[Corollary 3.5]{danilenko-2001}. This is called the {\bf subgroup formula}. 

In the case of the $f$-invariant, a similar formula holds:
\begin{thm}\cite[Theorem 1.1, Corollary 1.2]{seward-subgroup}
 Let $\Ga$ be a finitely generated free group, $\Lambda\le \Ga$ and $\Ga \cc (X,\mu)$ be a pmp action with a finite-entropy generating partition. If $\Lambda$ has finite index in $\Ga$ then the induced action $\Lambda\cc (X,\mu)$ also has a finite-entropy generating partition and 
$$f_\mu(\Lambda \cc X) = [\Ga:\Lambda] f_\mu(\Ga \cc X).$$
If $[\Ga:\Lambda]=+\infty$, there are infinitely many finite index subgroups $\Ga'$ with $\Lambda<\Ga'<\Ga$ and there is a finite-entropy generating partition for the action $\La \cc (X,\mu)$ then $f_\mu(\Ga \cc X)\le 0$. 
\end{thm}

For Rokhlin and sofic entropy we have:
\begin{thm}\cite[Theorem 1.7]{seward-weak-containment}\label{thm:subgroup1}
Let $\G$ be a sofic group with sofic approximation $\Si$. Suppose that for all finite-index normal subgroups $N \vartriangleleft \G$, the $\Si$-entropy of $\G \cc (\G/N, u_{\G/N})$ is not minus infinity. Then for every finite index subgroup $\La\le \G$ and aperiodic pmp action $\G \cc^T (X,\mu)$,
$$[\G:\La] h_{\Sigma}(T) \le h^{\rm{Rok}}(T\resto \La) \le [\G:\La] h^{\rm{Rok}}(T)$$
where $T\resto \La$ is the restriction of the action to $\La$. 

\end{thm}


\begin{thm}\label{thm:subgroup2}
If $\Lambda\le \Ga$ and $\Ga \cc X$ is a continuous action on a compact metrizable space then
$$h_{\Sigma}(\Ga \cc X) \le h_{\Sigma\resto \Lambda}(\Lambda \cc X)$$
where $\Sigma\resto \Lambda$ denotes the restriction of $\Sigma$ to $\Lambda$. A similar inequality holds in the measure-entropy case.
\end{thm}

\begin{proof}
This follows immediately upon realizing that for any continuous pseudometric $\rho$ on $X$, any finite $F \subset \Ga$, $\delta>0$ and $\sigma:\Ga \to \sym(d)$,
$$\Map(T,\rho,F,\delta,\sigma) \subset \Map(T \resto \Lambda,\rho,F,\delta,\sigma \resto \Lambda).$$
The argument for the measure-entropy case is similar.
\end{proof}

\begin{question}
Are the bounds in Theorems \ref{thm:subgroup1} and \ref{thm:subgroup2} tight?
\end{question}


Next we present a counterexample:

\begin{thm}
There exist a finite-index subgroup $\La \le \Ga$ of a sofic group $\G$, a sofic approximation $\Si$ to $\G$ and a pmp action $\G \cc^T (X,\mu)$ such that
$$h_{\Sigma,\mu}(T) = -\infty, \quad h_{\Sigma \resto \La,\mu}(T \resto \La)=0.$$
\end{thm}

\begin{proof}

Let $\La=\langle a,b\rangle$ be the rank 2 free group, $\Ga=\La \times \Z/2$. We identify $\La$ with the subgroup $\La \times \{0\}$ of $\Ga$. Let $(X,\mu)=(\Z/2,u_2)$ and let $T$ be the trivial action of $\Ga$ on $\Z/2$. Let $\Sigma'=\{\sigma'_i:\F \to \sym(V_i)\}_{i=1}^\infty$ be a sofic approximation to $\F=\La$ by expanders as in \S \ref{sec:expanders}.  Define $\sigma_i:\F \times \Z/2 \to \sym( V_i \times \Z/2)$ by 
$$\sigma_i(f,x)(v,y)=(\sigma'_i(f)v,x+y).$$ 
Then $\Sigma=\{\sigma_i\}_{i=1}^\infty$ is a sofic approximation to $\G=\F\times \Z/2$. 

Because $\Sigma'$ is by expanders, $\Sigma$ is also by expanders. So Proposition \ref{prop:trivial} implies $h_{\Sigma,u_2}(T) = -\infty$. However, the restriction $\Sigma\resto \F$ is the disjoint union of two copies of $\Sigma'$. So the function 
$$\phi: V_i \times \Z/2 \to \Z/2, \quad \phi(v,x)=x$$
is a microstate for the trivial action $\La \cc \Z/2$. This shows that $h_{\Sigma \resto \La,u_2}(T \resto \La) \ge 0.$ The upper bound can be derived directly or via naive entropy.
\end{proof}



\subsection{Co-induction}\label{sec:coinduction}       

\begin{defn}
Let $\Lambda\le \Ga$ be countable discrete groups. Let $X$ be a compact metrizable space and $\Lambda \cc X$ an action by homeomorphisms. Let
$$Y=\{f:\Ga \to X:~f(gh)=h^{-1}f(g) \quad \forall g\in \Ga, h\in \Lambda\} \subset X^\Ga$$
and give $Y$ the subspace topology (and  $X^\Ga$ the product topology). So $Y$ is a compact metrizable space and $\Ga \cc Y$ by
$$(gf)(x)=f(g^{-1}x)\quad \forall g,x\in \Ga, f\in Y.$$
The action $\Ga \cc Y$ is  the {\bf action of $\Lambda \cc X$ co-induced to $\Ga$}.
\end{defn}

For example, if $\Lambda$ is the trivial subgroup, then the co-induced action is a Bernoulli shift over $\G$. 

\begin{thm}\label{thm:coinduced}\cite[Proposition 5.29]{hayes-fk-determinants}
Let $\Lambda\le \Ga$ and $\Sigma$ be a sofic approximation to $\Ga$. Let $\Lambda \cc X$ be an action on a compact metrizable space by homeomorphisms and $\Ga \cc Y$ the co-induced action. Then $h_\Sigma(\Ga \cc Y) = h_{\Sigma \resto \Lambda}(\Lambda \cc X).$
\end{thm}



\begin{question}
Is there a measure entropy version of Theorem \ref{thm:coinduced} or a Rokhlin entropy version? Results have been obtained by Tim Austin for a variant of sofic entropy \cite{MR3542515}.
\end{question}

\subsection{Perturbing the partition}\label{sec:continuity1}       


\begin{defn}[The space of partitions]
Let $(X,\mu)$ be a standard probability space. Given countable partitions $\cP,\cQ$ of $X$ the {\bf relative Shannon entropy} of $\cP$ given $\cQ$ is
$$H_\mu(\cP|\cQ):=H_\mu(\cP \vee \cQ)-H_\mu(\cQ).$$
The {\bf Rokhlin distance} between $\cP$ and $\cQ$ is
$$d(\cP,\cQ):=H_\mu(\cP|\cQ) + H_\mu(\cQ|\cP).$$
This is a metric on the space of (mod 0 equivalence classes of) partitions of $(X,\mu)$ with finite Shannon entropy.

\end{defn}


\begin{notation}
Let $\G \cc^T (X,\mu)$ be a pmp action. Given a measurable partition $\cP$ of $X$, the Mackey Realization Theorem implies the existence of a factor $\G \cc^{T_\cP} (X_\cP,\mu_\cP)$ of $T$ such that if $\Phi:X \to X_\cP$ is the factor map then $\Phi^{-1}(\cB_{X_\cP})$ is the sigma-algebra $\bigvee_{g\in\G} T^g \cP$ (up to measure zero). 
\end{notation}

\begin{prop}
Let $\Sigma$ be a sofic approximation to a countable group $\Ga$ and $\G \cc^T (X,\mu)$ a pmp action.   Then both the sofic entropy $h_{\Si,\mu_\cP}(T_\cP)$ and the Rokhlin entropy $h^{\rm{Rok}}(T_\cP)$ vary upper semi-continuously in $\cP$  with respect to the Rokhlin metric. If $\G$ is amenable then the entropy varies continuously in $\cP$. In the non-amenable case, neither entropy is continuous in general. 
\end{prop}
\begin{proof}
The first statement is contained in \cite[Corollary 6.3]{bowen-jams-2010} (sofic entropy) and \cite{seward-kreiger-2} (Rokhlin entropy). The amenable case is handled in \cite[Proposition 4.3.13]{ollagnier-book}. For the last statement, we perturb the Ornstein-Weiss example as follows. Let $\G=\F_2=\langle a,b\rangle$ be the rank 2 free group, $X=(\Z/2)^\G$ be the full 2-shift and $\mu = u_2^\G$ be the product measure on $X$. So $T$ is the usual shift action $(T^g)x(f)=x(g^{-1}f)$. Let $\phi: X \to \Z/2 \times \Z/2$ be the observable
$$\phi(x)=(x(1_\G)+x(a), x(1_\G)+x(b))$$
and $\cP=\{\phi^{-1}(i,j):~i,j \in \Z/2\}$ be the corresponding partition. For every $n>0$, choose a subset $A_n \subset \{x \in X:~x(1_\G)=0\}$ such that $0<\mu(A_n)<1/n$. Finally, let
$$\cP_n = \cP \vee \{A_n, X \setminus A_n\}.$$

The Ornstein-Weiss example (\S \ref{sec:OW-intro}) shows that $\G \cc (X_\cP,\mu_\cP)$ is isomorphic to the full 4-shift. In particular,
$$h_{\Si,\mu_\cP}(T_{\cP}) = \log 4 = h^{\rm{Rok}}(T_\cP).$$
Moreover, the factor map from $(X,\mu)$ to $(X_\cP,\mu_\cP)$ is 2-1. Indeed, if $\bar{1} \in (\Z/2)^\G$ denotes the constant function then for any $x\in X$, $x$ and $\bar{1}+x$ have the same image in $(X_\cP,\mu_\cP)$.

We claim that $\cP_n$ is generating for all $n$.  It suffices to show that for a.e. $x\in X$ there exists $g \in \G$ such that $T^gx$ and $T^g(\bar{1}+x)$ lie in different parts of $\cP_n$. Since the action is ergodic, there exists $g\in\G$ such that $T^gx \in A_n$. This implies $(T^gx)(1_\G)=1$ and therefore $(T^g(x+\bar{1}))(1_\G)=0$. In particular, $T^g(\bar{1}+x)$ cannot be in $A_n$. So $\cP_n$ is generating as claimed.

Because $\cP_n$ is generating,
$$h_{\Si,\mu_{\cP_n}}(T_{\cP_n}) = \log 2 = h^{\rm{Rok}}(T_{\cP_n})$$
for all $n$. Since $\mu(A_n) \to 0$ as $n\to\infty$,  $\cP_n \to \cP$ in the Rokhlin metric. 

\end{proof}


\subsection{Perturbing the measure}\label{sec:continuity2}       

Let $K$ be a finite set, $\G \cc^T K^\G$ the shift action and $\Prob_\G(K^\G)$ the space of all $T(\G)$-invariant Borel probability measures on $K^\G$ with the weak* topology.

\begin{prop}
Let $\Sigma$ be a sofic approximation to $\Ga$.  Then both the $\Si$-entropy and the Rokhlin entropy of the action $\Ga \cc^T (K^\G, \mu)$ vary upper semi-continuously in $\mu \in \Prob_\G(K^\G)$ with respect to the weak* topology. In general, these are not continuous (even when $\Ga=\Z$).
\end{prop}

\begin{proof}[Proof remarks]
The Rokhlin entropy case is handled in \cite{seward-kreiger-2}. For sofic entropy, a more general statement is proven in \cite{chung-zhang-expansive}: whenever $\G$ acts expansively on a compact metric space $X$ then $\Si$-entropy varies upper semi-continuously with respect to the measure $\mu$. In fact, this is proven for a weak form of expansitivity called asymptotic $h$-expansitivity. 

In \cite{MR2546620} it is shown that if $\F$ is a finite rank free group (including $\Z$) then the set of shift-invariant measures $\mu$ that have finite support are dense in $\Prob_\F(K^\F)$. If a shift-invariant measure has finite support then it has zero naive entropy and therefore non-positive sofic and Rokhlin entropy (by Propositions \ref{prop:naive-Rokhlin1} and \ref{prop:Rokhlinsofic}). So entropy is not a continuous function of $\mu \in \Prob_\F(K^\F)$ (even when $\F=\Z$). 
\end{proof}

\begin{remark}
In \S \ref{sec:d-bar}, we discuss a stronger topology on $\Prob_\G(K^\Ga)$ called the $d$-bar topology. If $\Ga$ is amenable then entropy varies continuously in the $d$-bar topology but there are explicit counterexamples when $\Ga$ is a rank 2 free group.
\end{remark}

\subsection{Orbit-equivalence}

Two actions $\G \cc (X,\mu), \La \cc (Y,\nu)$ are {\bf orbit-equivalent} (abbreviated OE) if there exists a measure space isomorphism $\Phi:X \to Y$ such that $\Phi(\G x)=\La \Phi(x)$ for a.e. $x$. By work of Dye and Ornstein-Weiss \cite{MR0158048, MR0131516, OW80}, all ergodic essentially free pmp actions of infinite amenable groups are OE. In particular, entropy is not an OE-invariant. However, every non-amenable group admits uncountably many non-OE actions \cite{epstein-oe, IKT}. Moreover, there are many groups for which Bernoulli shifts are OE-rigid in the sense that OE implies measure-conjugacy. These include ICC groups with property (T) \cite{MR2231962} and direct products of non-amenable groups with infinite groups that have no nontrivial finite normal subgroups \cite{popa-malleable}. So for these actions, entropy is automatically an OE-invariant.

This leads to the general question:
\begin{question}
Under what conditions on a group action can one conclude that entropy (whether sofic/Rokhlin/naive) is an OE-invariant?
\end{question}
For example in \cite{MR2369194} it is proven that $\G$ is a non-exceptional mapping class group then any pmp action in which all finite-index subgroups acts ergodically is OE-rigid. It follows that entropy is an OE-invariant for such actions. In \S \ref{sec:weaklycompact} below, we show there is a property of actions called weak compactness which is an OE-invariant and implies zero entropy whenever the group is non-amenable.

\subsubsection{Weakly compact actions}\label{sec:weaklycompact}

(The ideas of this section have been gracefully provided by Ben Hayes.)

If $(X,\mu),(Y,\nu)$ are standard probability spaces and $f\colon X\to \C,g\colon Y\to\C$ are measurable then $f\otimes g\colon X\times Y\to\C$ is defined by $(f\otimes g)(x,y)=f(x)g(y).$ If $\Gamma \cc (X,\mu)$  is a pmp action  then $\kappa\colon\Gamma\to \mathcal{U}(L^{2}(\mu))$ (where $\cU(\cdot)$ is the unitary group) denotes the Koopman representation given by
\[(\kappa_{g}\xi)(x)=\xi(g^{-1}x) \quad \forall \xi \in L^2(\mu), g\in \G, x\in X.\]
The map $\kappa_{g}\otimes \kappa_{g}\colon\Gamma\to \mathcal{U}(L^{2}(\mu\times \mu))$ is the Koopman representation of $\Gamma$ on $X\times X.$ Define $\kappa_{g}(\xi), \kappa_{g}\otimes \kappa_{g}(\zeta)$  by the same formula for $\xi \in L^1(\mu), \zeta \in L^1(\mu\times \mu)$.

\begin{defn}\cite{MR2680430}\label{D:weaklycompact}
 A pmp action $\G \cc (X,\mu)$  is {\bf weakly compact} if there is a sequence $\xi_{1}, \xi_2,\ldots \in L^{2}(\mu\times \mu)$ such that each $\xi_{n}\geq 0$,
\begin{itemize}
\item $\|\xi_{n}-(v\otimes \overline{v})\xi_{n}\|_{2}\to 0\mbox{\emph{ as $n\to\infty$ for all $v\colon X\to S^{1}$ measurable}}$,
\item $\|\xi_{n}-(\kappa_{g}\otimes \kappa_{g})(\xi_{n})\|_{2}\to 0\mbox{\emph{ as $n\to\infty$ for all $g\in \Gamma$}}$,
\item $\langle (f\otimes 1)\xi_{n},\xi_{n} \rangle=\int_{X}f\,d\mu=\langle (1\otimes f)\xi_{n},\xi_{n}\rangle,\mbox{\emph{ for all $f\in L^{\infty}(X,\mu)$}.}$
\end{itemize}

\end{defn}

It is shown in \cite[Proposition 3.2]{MR2680430}  that compact actions are weakly compact. Moreover, weak compactness is an OE-invariant  (see the remarks after \cite[Proposition 3.4]{MR2680430}). It is shown below that all weakly compact actions of non-amenable groups have zero Rokhlin entropy. Moreover, because weak compactness is preserved under factors, no factor of a weakly compact action has positive entropy.

Set $\zeta_{n}=\xi_{n}^{2}$ to see that weak compactness is equivalent to saying that there exists $\zeta_{n}\in L^{1}(\mu\times \mu)$ so that $\zeta_{n}\geq 0$ and:
\begin{itemize}
\item $\|\zeta_{n}-(v\otimes \overline{v})\zeta_{n}\|_{1}\to 0 \mbox{ as $n\to\infty$ for all measurable $v\colon X\to S^{1} \subset \C$},$
\item $\|\zeta_{n}-(\kappa_{g}\otimes \kappa_{g})(\zeta_{n})\|_{1}\to 0\mbox{ as $n\to\infty$ for all $g\in\Gamma,$}$
\item $\int f(x)\zeta_{n}(x,y)\,d\mu\times \mu(x,y)=\int_{X}f\,d\mu=\int f(y)\zeta_{n}(x,y)\,d\mu\times \mu(x,y)\mbox{ for all $f\in L^{\infty}(\mu).$}$
\end{itemize}

\begin{prop}\label{prop:wc}
All factors of weakly compact actions are weakly compact.  
\end{prop}

\begin{proof} Let $\Gamma\cc (Y,\nu)$ be a factor of $\Gamma\cc (X,\mu)$ and $\E_{Y\times Y}(f)$ be the conditional expectation of $f\in L^{1}(\mu\times \mu)$ onto $L^{1}(\nu\times \nu)$ (which we identify with the obvious subspace of $L^1(\mu\times \mu)$). Let $\zeta_{n}\in L^{1}(\mu\times \mu)$ be as in the remarks after Definition \ref{D:weaklycompact} and set $\eta_{n}=\E_{Y\times Y}(\zeta_{n}).$ It is straightforward to check that
\[\|\eta_{n}-(v\otimes \overline{v})\eta_{n}\|_{1}\to 0 \mbox{ for all measurable $v\colon Y\to S^{1} \subset\C$},\]
\[\|\eta_{n}-(\kappa_{g}\otimes \kappa_{g})(\eta_{n})\|_{1}\to 0\mbox{ for all $g\in\Gamma,$}\]
\[\int f(x)\eta_{n}(x,y)\,d\nu\times \nu(x,y)=\int_{Y}f\,d\nu=\int f(y)\eta_{n}(x,y)\,d\nu\times \nu(x,y)\mbox{ for all $f\in L^{\infty}(Y,\mu)$}.\]
 Thus $\{\eta_{n}\}$ witnesses that $\Gamma\cc (Y,\nu)$ is weakly compact.
\end{proof}

\begin{prop}\label{prop:wc-b}
 If $\G$ is non-amenable then no  nontrivial Bernoulli action of $\Gamma$ is weakly compact.
\end{prop}

\begin{proof}
 Weak compactness of $\Gamma\cc (X,\mu)$ tautologically implies the product action $\Gamma\cc (X\times X,\mu\times \mu)$ does not have spectral gap. However, if $\G \cc (X,\mu)$ is Bernoulli then so is the product action. It is well-known that  Bernoulli actions of non-amenable groups have spectral gap.
\end{proof}

\begin{cor}
If $\G$ is non-amenable and $\G \cc (X,\mu)$ is essentially free and weakly compact then it has zero Rokhlin entropy. Moreover, every action OE to this action has zero Rokhlin entropy. 
\end{cor}

\begin{proof}
If $\G \cc (X,\mu)$ has positive Rokhlin entropy then it factors onto a Bernoulli shift by Seward's generalization of Sinai's Theorem (Theorem \ref{thm:seward-sinai}). The Corollary now follows from Propositions \ref{prop:wc} and \ref{prop:wc-b}.  
\end{proof}

\begin{question}
Do all weakly compact actions of non-amenable groups have zero naive entropy?
\end{question}

\subsubsection{Orbit equivalence and relative entropy}\label{sec:oe-rel}

Rudolph and Weiss showed that while entropy is not an OE-invariant, entropy relative to the orbit-change sigma-algebra is \cite{rudolph-weiss-2000}. They developed this tool to prove that CPE actions of amenable groups are uniformly mixing. It has since been used to show that CPE actions of amenable groups have countable Lebesgue spectrum \cite{dooley-golodets-2002} and to give an alternative development of Ornstein theory \cite{danilenko-2001, danilenko-2002}. We give the precise statement next.

Suppose essentially free actions $\G \cc (X,\mu)$ and $\La \cc (Y,\nu)$ are orbit-equivalent via an isomorphism $\Phi:X \to Y$. Then we can define cocycles $\alpha:\G \times X \to \La$ and $\b:\La \times Y \to \G$ by
$$\alpha(g,x)\Phi(x) = \Phi(gx), \quad \beta(h,y)\Phi^{-1}(y) = \Phi^{-1}(hy).$$
These satisfy the cocycle equations
$$\alpha(gh,x)=\alpha(g,hx)\alpha(h,x),\quad \beta(gh,y)=\beta(g,hy)\beta(h,y).$$
Let $\cF_X \subset \cB_X$ be the smallest sigma-subalgebra such that $\alpha(g,\cdot):X \to \La$ is $\cF_X$-measurable for all $g \in \G$. If $\cF_Y \subset \cB_Y$ is defined similarly, then $\Phi$ maps $\cF_X$ to $\cF_Y$. These are the {\bf orbit-change} sigma-algebras.

\begin{thm}[Rudolph-Weiss Theorem]\cite{rudolph-weiss-2000}
Let $\G, \La$ be amenable groups and $\G \cc (X,\mu), \La\cc (Y,\nu)$ essentially free pmp actions as above. Then
$$h_\mu(\G \cc X| \cF_X) = h_\nu(\La \cc Y| \cF_Y).$$
\end{thm}
This theorem has been generalized to Rokhlin entropy and arbitrary countable groups \cite{seward-kreiger-1}. It plays a major role in Seward's extension of Krieger's Theorem. It also inspired sofic entropy theory for actions of groupoids \cite{MR3286052}. 

\begin{question}
Is there an analogue of the Rudolph-Weiss Theorem for arbitrary sofic groups if one takes for entropy the supremum over all sofic entropies? Perhaps one  should use Ben Hayes' definition of relative sofic entropy \cite{hayes-relative-entropy}?
\end{question}

\subsubsection{Integrable orbit-equivalence}

With notation as in \S \ref{sec:oe-rel}, suppose that $\G$ and $\La$ are endowed with word lengths $\|\cdot \|_\G, \|\cdot \|_\La$. Then the orbit-equivalence is said to be {\bf integrable} if for every $g\in \G$ and $h \in \La$,
$$\int \|\alpha(g,x)\|_\La~d\mu(x) < \infty, \quad \int \|\beta(h,y)\|_\G ~d\nu(y)<\infty.$$
Tim Austin proved in \cite{MR3579704} that if $\G$ and $\La$ are amenable then entropy is invariant under integrable orbit-equivalence (abbreviated IOE). 

\begin{question}
Is either Rokhlin entropy, $\Si$-entropy, naive entropy or the $f$-invariant IOE- invariant?  Is the supremum of $\Si$-entropies over all $\Si$ an IOE-invariant?
\end{question}

\section{Factors and extensions}\label{sec:factors}       

\subsection{A variant of the Ornstein-Weiss example}\label{sec:OW}       

The Ornstein-Weiss example (\S \ref{sec:OW-intro}) is an entropy-increasing finite-to-1 factor map of a Bernoulli shift action of the free group onto another Bernoulli shift. It was generalized by Gaboriau-Seward via an algebraic construction that applies to all groups (\S \ref{sec:yuz}). Their entropy bounds show that the map is entropy-increasing whenever $\G$ has positive first $L^2$-Betti number. Next we present an example which is even more extreme: a zero entropy action with a finite-to-1 factor map onto a Bernoulli shift.




\begin{thm}\label{thm:variant}
Let $\F=\langle a,b\rangle$. Then there exists an ergodic essentially free pmp action $\F \cc^T (X,\mu)$ such that
\begin{enumerate}
\item $h_{\Sigma,\mu}(T) = f_\mu(T)=h^{\rm{Rok}}(T)=0$ for every sofic approximation $\Sigma$;
\item $\F \cc^T (X,\mu)$ admits a 2-1 factor map onto a Bernoulli shift.
\item Moreover the action is algebraic and the factor map is a continuous homomorphism.
\end{enumerate}
\end{thm}

\begin{proof}
 Let $\Lambda=\langle a \rangle$ denote the subgroup of $\F$ generated by $a$. We regard $(\Z/2)^\F$ as a compact abelian group under pointwise addition. Let $X\le (\Z/2)^\F$ denote the subgroup consisting of all $x\in (\Z/2)^\F$ such that $x_{ga^n} = x_{ga^m}$ for every $g\in \F$ and $n,m \in \Z$. Let $(T^g)_{g\in \G}$ be the shift action on $X$. So $(T^g)x(f)=x(g^{-1}f)$. This action is by group-automorphisms and so preserves the Haar measure.

The action $\F \cc^T X$ is measurably conjugate to the trivial action of $\Lambda$ on $\Z/2$ co-induced to $\F$. Because co-induction preserves topological entropy (Theorem \ref{thm:coinduced}) the action $\F \cc^T X$ has zero topological sofic entropy with respect to every sofic approximation. By the variational principle (Theorem \ref{thm:variational}), the measure sofic entropy is also zero. The time-0 partition is a Markov partition on $X$ (when viewed as a subset of $(\Z/2)^\F$). So its $f$-invariant can be computed directly and shown to equal zero. 

To see that the Rokhlin entropy is zero, define $\phi_n:X \to \{0,1\}$ by $\phi_n(x)=0$ if $x(b^i)=0$ for all $|i|\le n$ and $\phi_n(x)=1$ otherwise.  We claim that $\phi_n$ is generating in the sense that the smallest $\F$-invariant sigma-algebra containing $\phi_n^{-1}(0)$ is the sigma-algebra of all Borel sets (modulo measure zero).

If $x\in X$ is Haar random then for any $g \in\G$, $\{x(gb^n):~n \in \Z\}$ are iid random variables uniformly distributed on $\{0,1\}$. It follows that for a.e. $x\in X$ if $g\in \G$ is such that $x(g)=0$ then there exists $m \in \Z$ such that $\phi_n( T^{a^mg^{-1}}x)=0$ (i.e., $x(ga^{-m}b^i)=0$ for all $|i|\le n$). 

Now suppose $x,y \in X$ satisfy the condition above and $x \ne y$. It suffices to show there exists $g\in \F$ such that $\phi_n(T^gx) \ne \phi_n(T^gy)$. Since $x \ne y$, there exists $h \in \G$ such that $x(h)\ne y(h)$. Without loss of generality, $x(h)=0, y(h)=1$. By the condition above there exists $m \in \Z$ such that $\phi_n(T^{a^mh^{-1}}x) = 0$. However $\phi_n(T^{a^mh^{-1}}y)=1$ since $y$ is constant on $\La$-cosets. Setting $g = a^mh^{-1}$ shows $\phi_n(T^gx)\ne \phi_n(T^gy)$ and so $\phi_n$ is generating. 

Since the measure of $\phi_n^{-1}(0)$ tends to zero as $n\to\infty$, the Shannon entropy of $\{\phi_n^{-1}(0),\phi_n^{-1}(1)\}$ also tends to zero as $n\to\infty$. This shows $h^{\rm{Rok}}(\F \cc (X,\mu))=0$.


Let $\Phi: X \to (\Z/2)^\F$ be the map $\Phi(x)_g = x_g + x_{gb}$. This is a surjective homomorphism with kernel $N=\{ {\bf 0}, {\bf 1}\}$. It is $\F$-equivariant. So it is a 2-1 factor map onto the Bernoulli shift with entropy $\log(2)$. 
\end{proof}

\begin{question}
Is it true that all proper factors of the system described above have positive entropy? If so, then this action has zero sofic entropy but is `almost' CPE (completely positive entropy).
\end{question}


\begin{question}
In \cite{MR508264} D. Rudolph showed that every finite-to-1 extension of a Bernoulli shift over $\Z$ is either Bernoulli or has a nontrivial factor map onto a finite action. By Theorem \ref{thm:variant}, Rudolph's Theorem does not extend to non-abelian free groups. Is there some other classification of finite-to-1 extensions of Bernoulli shifts over a non-abelian free group?
\end{question}

\subsection{Bernoulli factors of Bernoulli shifts}\label{sec:bernoulli-bernoulli}        

This section sketches a proof of:
\begin{thm}\label{thm:ri}\cite{bowen-ri}
Let $\G$ be a non-amenable countable group. Then all Bernoulli shifts of $\G$ factor onto each other.
\end{thm}
The first step is to prove the theorem when $\G=\F_2=\langle a,b \rangle$ is the rank 2 free group (details are in \cite{bowen-ornstein-2011}). Let $\Phi:(\Z/2)^{\F_2} \to (\Z/2\times \Z/2)^{\F_2}$ denote the Ornstein-Weiss map. Now consider the map $\Psi: (\Z/2 \times \Z/2)^{\F_2} \to (\Z/2 \times \Z/2 \times \Z/2)^{\F_2}$ given by $\Psi(x,y) = (x,\Phi(y))$. By composing $\Psi$ with $\Phi$, we see that ${\F_2} \cc (\Z/2)^{\F_2}$ factors onto ${\F_2} \cc (\Z/2 \times \Z/2 \times \Z/2)^{\F_2}$. This construction can be  iterated to show that $(\Z/2)^{\F_2}$ factors onto $((\Z/2)^\N)^{\F_2}$. Since the latter has a diffuse base space, it factors onto all Bernoulli shifts over $\F_2$. Therefore $\F_2 \cc (\Z/2)^{\F_2}$ factors onto all Bernoulli shifts. 

A co-induction argument using Sinai's Factor Theorem for $\Z$ shows that whenever $(K,\k)$ is a probability space with $H(K,\k) \ge \log 2$, the corresponding Bernoulli shift ${\F_2} \cc (K,\k)^{\F_2}$ factors onto $\F_2 \cc (\Z/2,u_2)^{\F_2}$ and therefore factors onto all Bernoulli shifts (the factor map for the ${\F_2}$-action is constructed by applying the factor map for $\Z$-actions to each $\langle a\rangle$ coset). Another co-induction argument (due to Stepin \cite{stepin-1975}) applied to Ornstein's Isomorphism Theorem shows that whenever two probability spaces have the same Shannon entropy then their corresponding Bernoulli shifts over ${\F_2}$ are isomorphic. 

It now suffices to show that for any $\eps>0$ there is a probability space $(K,\k)$ with Shannon entropy $<\eps$ such that the corresponding Bernoulli shift ${\F_2} \cc (K,\k)^{\F_2}$ factors onto a Bernoulli shift with base entropy $\ge \log 2$. This is possible by a variant of the Ornstein-Weiss map. Consider the space $K = \Z/2 \cup \{*\}$ with a probability measure $\k_1$ satisfying $\k_1(\{0\})=\k_1(\{1\})$. Such a measure can be found to have Shannon entropy $<\eps$. Let $L = \Z/2 \times \Z/2 \cup \{*\}$ with the probability measure $\l_1 \in \Prob(L)$ satisfying $\l_1( \{ (i,j)\}) = \k_1(\{1\})/2$ for all $i,j \in \Z/2$.

The factor map $\Phi:K^{\F_2} \to L^{\F_2}$ is defined by: $\Phi(x)(g)=\phi(g^{-1}x)$ where $\phi:K^{\F_2} \to L$ is defined by $\phi(x)=*$ if $x(1_{\F_2})=*$ and otherwise
$$\phi(x) = (x(1_{\F_2}) + x(a^n), x(1_{\F_2})  + x(b^m))$$
where $n, m \in \N$ are the smallest natural numbers such that $x(a^n) \ne *, x(b^m) \ne *$. It can be shown $\Phi$ pushes the product measure $\k_1^{\F_2}$ forward to the measure $\l_1^{\F_2}$. Therefore, entropy increases. Note
$$H(L,\l_1) = H(K,\k_1) + (1-\k_1(\{*\}))\log(2)>H(K,\k_1).$$
If $H(L,\l_1)\ge \log 2$ then we are done. Otherwise, there exists an isomorphism from $\F_2 \cc (L,\l_1)^{\F_2}$ to a Bernoulli shift of the form $\F_2 \cc (K,\k_2)^{\F_2}$ where $\k_2$ is a probability measure satisfying $\k_2(\{0\})=\k_2(\{1\})$. So we can find a new factor map from $\F_2 \cc (K,\k_2)^{\F_2}$ to $\F_2 \cc (L,\l_2)^{\F_2}$ where $H(L,\l_2)>H(K,\k_2)$. By composing, we obtain a factor map from the original action $\F_2 \cc (K,\k_1)^{\F_2}$ onto $\F_2 \cc (L,\l_2)^{\F_2}$. After a finite number of similar steps, ${\F_2} \cc (K,\k_1)^{\F_2}$ factors onto a Bernoulli shift with base space entropy $\ge \log(2)$. This is because after the $i$-th step the entropy of the base increases by $(1-\k_i(\{*\}))\log(2)$ and $\k_i(\{*\})$ is decreasing to zero. This concludes the case $\G={\F_2}$. 

The general case proceeds by a measurable co-induction argument. The key new ingredient is a generalization of the Gaboriau-Lyons Theorem \cite{gaboriau-lyons} to arbitrary Bernoulli shifts. That theorem states that if $\G$ is an arbitrary non-amenable group then {\em there exists} a probability space $(K,\k)$ and there exists an ergodic essentially free pmp action $\F_2 \cc (K,\k)^\G$ whose orbits are contained in the orbits of the Bernoulli action $\G \cc (K,\k)^\G$. The generalization proves the same statement with ``there exists a probability space $(K,\k)$'' replaced by ``for every nontrivial probability space $(K,\k)$''. Details will appear in \cite{bowen-ri}. This generalizes previous work of Ball \cite{ball-factors1} who proved for every non-amenable group $\G$ there exists some Bernoulli  shift with finite base entropy that factors onto all Bernoulli shifts over $\G$.

\subsection{Zero entropy extensions}\label{sec:zero-entropy}
This section sketches a proof of:
\begin{thm}\label{thm:zero-extension}\cite{MR3530042}
Let $\Ga$ be a non-amenable countable group and $\Ga \cc (X,\mu)$ a free ergodic action. Then there exists a free ergodic action $\Ga \cc (\tX,\tmu)$ that factors onto $\Ga \cc (X,\mu)$ and has zero Rokhlin entropy.
\end{thm}

\begin{remark}
Seward's generalization of Krieger's Generator Theorem, (Theorem \ref{thm:krieger}) implies $\Ga \cc (\tX,\tmu)$ admits a generating partition with 2 parts. This improves an earlier result of Seward \cite{seward-small-action} which proved, under the same hypotheses as Theorem \ref{thm:zero-extension}, the existence of an extension $\Ga \cc (\tX,\tmu)$ that admits a generating partition with at most $n$ parts where $n=n(\Ga)$ depends only on $\Ga$. 
 \end{remark}

The next lemma is the key step.  The proof given here is simpler than the one in \cite{MR3530042} (which was written before Theorem \ref{thm:ri} was known).

\begin{lem}\label{cor:inverse limit}
Let $\Ga$ be any countable non-amenable group. There exists a pmp action $\Ga \cc^T (Z,\zeta)$ satisfying:
\begin{itemize}
\item $T$  is an inverse limit of Bernoulli shifts,
\item $h^{\rm{Rok}}(T)=0$
\item $\Ga \cc^T (Z,\zeta)$ factors onto all Bernoulli shifts over $\Ga$.
\end{itemize}
In particular, if $\G$ is sofic then $\G \cc^T (Z,\zeta)$ is not isomorphic to a Bernoulli shift.
\end{lem}

\begin{proof}
Let $(K_n,\k_n)$ be a sequence of probability spaces with Shannon entropy bounds $0<H(K_n,\k_n) < 2^{-n}$. By Theorem \ref{thm:ri}, there exist factor maps $\Phi_i:K_i^\G \to K_{i-1}^\G$ for $i\ge 2$. Let $\Ga \cc^T (Z,\zeta)$ denote the inverse limit of this system. Since $T$ factors onto all the Bernoulli shifts $\G \cc (K_n,\k_n)^\G$, Theorem \ref{thm:ri} implies it factors onto all Bernoulli shifts. So it suffices to show $h^{\rm{Rok}}(T)=0$. This follows from \cite[Corollary 3.9]{seward-kreiger-2}. Alternatively, let $\alpha_n$ be a generating partition of $K_n^\G$ with $H(\alpha_n)<2^{-n}$. By pulling back, we may consider $\alpha_n$ as a partition of $Z$. Let $\beta_m=\bigvee_{n=m}^\infty \alpha_n$. Then $H(\beta_m) < 2^{-m+1}$ and $\beta_m$ is generating for $T$. So $h^{\rm{Rok}}(T)=0$. 
\end{proof}

\begin{proof}[Proof of Theorem \ref{thm:zero-extension}]
By Seward's generalization of Sinai's Factor Theorem (Theorem \ref{thm:seward-sinai}), there exists a Bernoulli factor $\Ga \cc (Y,\nu)$ of $\Ga\cc (X,\mu)$ such that 
$$h^{\rm{Rok}}(\Ga \cc (X,\mu) | \cB_{Y}) = 0$$
where $\cB_Y$ denotes the pullback of the Borel sigma-algebra of $Y$.

Let $\Ga \cc (Z,\zeta)$ be as in Lemma \ref{cor:inverse limit}. Fix a factor map of  $\Ga \cc (Z,\zeta)$ onto $\Ga \cc (Y,\nu)$. Let $\Ga \cc (\tX,\tmu)$ be the independent joining of $\Ga \cc (Z,\zeta)$ and $\Ga \cc (X,\mu)$ over $\Ga \cc (Y,\nu)$. 

It suffices to show $h^{\rm{Rok}}(\Ga \cc (\tX,\tmu))=0$. To see this, let $\epsilon>0$, $\alpha$ be a generating partition of $Z$ with $H(\alpha)<\epsilon$ and $\beta$ be a partition of $X$ with $H(\beta)<\epsilon$ such that $\sigma-\rm{alg}_\Ga(\beta \cup \cB_{Y}) = \cB_X$ (up to measure zero). By pulling back, $\alpha$ and $\beta$ may be thought of as partitions on $\tX$. Clearly $\alpha \vee \beta$ is generating for the action $\Ga \cc (\tX,\tmu)$ and $H(\alpha\vee \beta)<2\epsilon$. Since $\epsilon>0$ is arbitrary, this implies the claim.

\end{proof}

\begin{remark}
The paper \cite{MR3530042} uses Theorem \ref{thm:zero-extension} to prove that the generic measure-preserving action of $\G$ on a fixed probability space $(X,\mu)$ has zero Rokhlin entropy. The special case when $\G$ is amenable was handled earlier by Dan Rudolph (see the Subclaim after Claim 19 in \cite{foreman-weiss}).
\end{remark}

\subsection{Finite-to-1 factors}\label{sec:finite-to-1}       

A factor map $\phi:X \to Y$ is finite-to-1 if for a.e. $y\in Y$, $\phi^{-1}(y)$ is finite. It is well-known that, for $\Z$-actions, finite-to-1 factor maps preserve entropy. This fact readily extends to actions of amenable groups. However it does not hold for non-amenable groups. The Ornstein-Weiss map is one counterexample \S \ref{sec:OW-intro}. For another, suppose $\Ga$ has a sofic approximation $\Sigma$ by expanders (as in \S \ref{sec:expanders}). Then the trivial action on a two point space has entropy $-\infty$ (with respect to $\Sigma$). However the trivial action on a one-point space has entropy zero. Nonetheless there are some positive results:

\begin{thm}\cite{bowen-entropy-2010a}\label{thm:f-finite}
If $\F_r$ denotes the rank r free group and $\F_r \cc (Y,\nu)$ is an $n$-to-1 factor of $\F_r \cc (X,\mu)$  then
$$f_{\nu}(\F_r \cc Y) = (r-1)\log(n) +  f_{\mu}(\F_r \cc X).$$ 
\end{thm}

\begin{proof}[Proof remarks]
The proof is almost immediate from the Abramov-Rokhlin formula for the $f$-invariant (\S \ref{sec:abramov}).
\end{proof}

\begin{remark}
Theorem \ref{thm:f-finite} gives new examples of entropy-increasing factor maps. For example, if $\F_r \cc^T (X,\mu)$ is any pmp action and $c:\F_r \times X \to G$ a cocycle (where $G$ is a finite group) then one can form the skew-product action $\F_r \cc^{S_c} X\times G$ defined by
$$S_c^g(x,h)=(T^gx, c(g,x)h).$$
Since the factor map $X \times G \to X$ is $|G|$-to-1, Theorem \ref{thm:f-finite} implies
$$f_{\mu}(T) = (r-1)\log(|G|) +  f_{\mu \times u_G}(S_c).$$ 
If $T$ is Bernoulli and $G=\Z/2$ say, are there simple conditions on the cocycle $c$ such that $S_c$ is also Bernoulli? Note the Ornstein-Weiss example is of this form.
\end{remark}

\begin{prop}
Assuming ergodicity, Rokhlin entropy does not decrease under a finite-to-1 factor map. 
\end{prop}

\begin{proof}
Let $\G \cc (Y,\nu)$, $\G \cc (Z,\zeta)$ be ergodic pmp actions and  $\pi:Y\to Z$ a finite-to-1 factor map. Let $\cP$ be a generating partition for $\G \cc (Z,\zeta)$. Because $\pi$ is finite-to-1 there exists a measurable partition $\{Y_1,\ldots, Y_n\}$ of $Y$ such that $\pi$ restricted to $Y_i$ is injective for all $i$. For $\eps>0$, let $Y'_i \subset Y_i$ be a subset with $0<\nu(Y'_i)<\eps$ and let $\cQ=\{Y'_1,Y'_2,\ldots, Y'_n, Y \setminus \cup_i Y'_i\}$ be the coarsest partition containing every $Y'_i$. Finally, let
$$\cR=\cQ\vee \pi^{-1}(\cP).$$
We claim that $\cR$ is generating for the action $\G \cc Y$. It suffices to show that for a.e. pair of distinct elements $x,y \in Y$  there exists $g\in \G$ such that $gx,gy$ are in different parts of $\cR$. So let $x,y \in Y$ be distinct. By ergodicity, we may assume $\G x$ intersects $Y'_i$ for some $i$.

If $\pi(x)\ne \pi(y)$ then, because $\cP$ is generating there exists $g\in \G$ such that $g\pi(x), g\pi(y)$ are in different parts of $\cP$ and therefore $gx,gy$ are separated by $\cR$. 

Now assume $\pi(x)=\pi(y)$.  Let $g\in \G$ be such that $gx \in Y'_i$. Because $\pi$ restricted to $Y'_i$ is injective and $\pi(gx)=\pi(gy)$ it follows that $gy \notin Y'_i$. So $\cR$ separates $gx$ and $gy$ and we are done. 

Because $\cR$ is generating,
$$h^{\rm{Rok}}(\G \cc (Y,\nu)) \le H_\nu(\cR) \le H_\zeta(\cP) + H_\nu(\cQ).$$
The partition $\cQ$ depends on $\eps>0$ and $H_\nu(\cQ) \searrow 0$ as $\eps \searrow 0$. So it follows that $h^{\rm{Rok}}(\G \cc (Y,\nu)) \le  H_\zeta(\cP).$ Since $\cP$ is an arbitrary generating partition for $\G \cc (Z,\zeta)$, the proposition follows.
 \end{proof}

\begin{remark}
In work-in-progress by Alpeev-Seward, the Rokhlin entropy of an action is the convex integral of the Rokhlin entropies of its ergodic components. So the previous remains true without the ergodicity assumption.
\end{remark}

\begin{question}
Is there an upper bound for the (sofic or Rokhlin) entropy of a finite-to-1 factor in terms of the entropy of the source?
\end{question}

\begin{prop}\label{prop:finite-to-1}
Sofic entropy does not decrease under a finite-to-1 factor map. More precisely, suppose $\G \cc (Y,\nu), \G \cc (Z,\zeta)$ are pmp actions and there is a finite-to-1 factor map $\pi:Y \to Z$. Then $h_{\Sigma,\zeta}(\Ga \cc Z) \ge h_{\Sigma,\nu}(\Ga \cc Y).$
The Variational Principle (Theorem \ref{thm:variational}) implies a similar result for topological sofic entropy.
\end{prop}

More generally, whenever $\G \cc (Y,\nu)$ is a compact extension of $\G \cc (Z,\zeta)$ then $h_{\Sigma,\zeta}(\Ga \cc Z) \ge h_{\Sigma,\nu}(\Ga \cc Y).$ This follows from \cite[Theorem 1.1]{MR3635672} although it might not be obvious. Ben Hayes graciously provided the following explanation.

For each $g\in \G$ and measurable function $f\colon Z\to \C,$ define a new measurable function $\alpha_{g}(f)\colon Z\to \C$ by
\[\alpha_{g}(f)(h)=f(g^{-1}h).\]
Let $L^{\infty}(Z)\rtimes_{\textnormal{alg}}\G$ be the set of all formal sums $\sum_{g\in \G}f_{g}u_{g}$ where $f_g \in L^\infty(Z)$ and $f_{g}=0$ for all but finitely many $g.$ The set $L^{\infty}(Z)\rtimes_{\textnormal{alg}}\G$  has a $*$-algebra structure: addition is  defined in the obvious manner, multiplication and the $*$-operation are defined by
\[\left(\sum_{g}f_{g}u_{g}\right)\left(\sum_{g}k_{g}u_{g}\right)=\sum_{g}\left(\sum_{h}f_{h}\alpha_{h}(k_{h^{-1}g})\right)u_{g}\]
\[\left(\sum_{g}f_{g}u_{g}\right)^{*}=\sum_{g}\overline{\alpha_{g}(f_{g^{-1}})}u_{g}.\]

We need to consider representations of $L^{\infty}(Z)\rtimes_{\textnormal{alg}}\G.$ A natural one is $\lambda_{Z}\colon L^{\infty}(Z)\rtimes_{\textnormal{alg}}\G\to B(L^{2}(Z\times \G,\zeta\times \eta)),$ (where $\eta$ is the counting measure and $B(\cdot)$ denotes the algebra of bounded operators) defined by
\[(\lambda_{Z}(f)\xi)(z,g)=f(z)\xi(z,g)\mbox{ for $f\in L^{\infty}(Z),\xi\in L^{2}(Z\times \G),z\in Z,g\in \G$}\]
\[(\lambda_{Z}(u_{h})\xi)(z,g)=\xi(h^{-1}z,h^{-1}g)\mbox{ for $\xi\in L^{2}(Z\times \G),z\in Z,h,g\in \G$}.\]
Given an extension $\G\cc (Y,\nu)\to \G\cc (Z,\zeta)$ with factor map $\pi\colon Y\to Z,$ there is  a natural representation $\rho\colon L^{\infty}(Z)\rtimes_{\textnormal{alg}}\G\to B(L^{2}(Y,\nu))$ given by
\begin{equation}\label{E:extension1}
(\rho(f)\xi)(y)=f(\pi(y))\xi(y)\mbox{ for $f\in L^{\infty}(Z),\xi\in L^{2}(Y)$}
\end{equation}
\begin{equation}\label{E:extension2}
(\rho(u_{g})\xi)(y)=\xi(g^{-1}y)\mbox{ for $g\in \G,\xi \in L^{2}(Y).$}
\end{equation}

\begin{defn}
Suppose that for $i=1,2$, $\rho_i: L^{\infty}(Z)\rtimes_{\textnormal{alg}}\G \to B(\cH_i)$ is a representation. A bounded linear map $T:\cH_1\to \cH_2$ is {\bf $L^{\infty}(Z)\rtimes_{\textnormal{alg}}\G$-modular} if
$$T(\rho_1(\phi) \xi)=\rho_2(\phi)T(\xi), \quad \forall \phi \in L^{\infty}(Z)\rtimes_{\textnormal{alg}}\G, ~\xi \in \cH_1.$$
The representations $\rho_1,\rho_2$ are {\bf mutually singular} if every $L^{\infty}(Z)\rtimes_{\textnormal{alg}}\G$-modular map $T:\cH_1\to \cH_2$ equals zero. An exercise shows this definition is symmetric in $\rho_1,\rho_2$.
\end{defn}


We will prove that if $\G\cc (Y,\nu)\to \G\cc (Z,\zeta)$ is a compact extension, then $\rho$ as defined above is mutually singular with respect to $\lambda_{Z}.$ To simplify the proof, we introduce a few definitions. Recall that if $(X,d)$ is a metric space, $A,B\subseteq X,$ and $\varepsilon>0$ then we write $A\subseteq_{\varepsilon}B$ if for every $a\in A,$ there is a $b\in B$ with $d(a,b)<\varepsilon.$ 

\begin{defn}\emph{Let $\mathcal{H}$ be a Hilbert space and $\rho: L^{\infty}(Z)\rtimes_{\textnormal{alg}}\G\to B(\mathcal{H})$ a $*$-representation. We say that $\rho$ is:}
\begin{itemize}
\item {\bf compact over $L^{\infty}(Z)$} \emph{if for every $\xi\in \mathcal{H}$ and every $\varepsilon>0,$ there are $\eta_{1},\dots,\eta_{k} \in \cH$ so that}
\[\rho(\G)\xi\subseteq_{\varepsilon,\|\cdot\|}\left\{\sum_{j=1}^{k}\rho(f_{j})\eta_{j}:f_{j}\in L^{\infty}(Z),\|f_{j}\|_{\infty}\leq 1\right\}.\]
\item {\bf mixing} \emph{if for every $\xi,\eta\in \mathcal{H}$}
\[\lim_{g\to\infty}\sup_{f\in L^{\infty}(Z),\|f\|_{\infty}\leq 1}|\langle \rho(f)\rho(u_{g})\xi,\eta \rangle|=0.\]
\end{itemize}

\end{defn}

The idea for each of these definitions is that we are replacing the usual complex scalars with $L^{\infty}(Z).$ So an element in $L^{\infty}(Z)$ of norm at most one should be thought of as a replacement for a complex number of size at most $1.$ Proposition \ref{prop:finite-to-1} now follows from Theorem 1.1 of \cite{MR3635672} and the next result.

\begin{prop}\label{P:mixing/compact}
Let $\G$ be a countable discrete group, $(Z,\zeta)$ a probability space and $\G\cc (Z,\zeta)$ a measure-preserving action.
\begin{enumerate} 
\item  Let $\mathcal{H}_{j},j=1,2$ be Hilbert spaces and let $\rho_{j}\colon L^{\infty}(Z)\rtimes_{\textnormal{alg}}\G\to B(\mathcal{H}_{j}),j=1,2$ be two $*$-representations. If $\rho_{1}$ is mixing and $\rho_{2}$ is compact over $L^{\infty}(Z),$ then $\rho_{1},\rho_{2}$ are mutually singular.\label{I:disjointness}
\item The representation $\lambda_{Z}$ is mixing. \label{I:left regular}
\item If $\G\cc (Y,\nu)\to \G\cc (Z,\zeta)$ is a compact extension and $\rho\colon L^{\infty}(Z)\rtimes_{\textnormal{alg}}\G\to B(L^{2}(Y))$ is defined by (\ref{E:extension1}),(\ref{E:extension2}), then $\rho$ is compact over $L^{\infty}(Z).$ \label{I:compact extension}
\end{enumerate}
\end{prop}

\begin{proof}

(\ref{I:disjointness}): Let $T\in B(\mathcal{H}_{2},\mathcal{H}_{1})$ be $L^{\infty}(Z)\rtimes_{\textnormal{alg}}\G$-modular. This means that $T(\rho_2(\phi) \xi)=\rho_1(\phi)T(\xi)$ for all $\phi \in L^{\infty}(Z)\rtimes_{\textnormal{alg}}\G, \xi \in \cH_2$. Let $\xi \in\mathcal{H}_{2}$ and $\varepsilon>0.$ Let $\eta_{1},\dots,\eta_{k} \in \cH$ be as in the definition of compact over $L^{\infty}(Z).$ Given $g\in \G,$ there exist $f_{g,1},\dots,f_{g,k}\in L^{\infty}(Z)$ with $\|f_{g,j}\|_{\infty}\leq 1$ for $j=1,\dots,k$ so that
\[\left\|\rho_{2}(u_{g})\xi-\sum_{j=1}^{k}\rho_{2}(f_{g,j})\eta_{j}\right\|<\varepsilon.\]
We have, for any $g\in \G:$
\begin{align*}
\|T(\xi)\|^{2}=\big|\big\langle T(\xi),T(\xi) \big\rangle\big|&=\big|\big\langle \rho_{1}(u_{g^{-1}})T(\rho_{2}(u_{g})\xi),T(\xi)\big\rangle \big|\\
&\leq \varepsilon\|T\|\|T(\xi)\|+\left|\sum_{j=1}^{k}\big\langle \rho_{1}(u_{g^{-1}})\rho_{1}(f_{g,j})T(\eta_{j}),T(\xi)\big\rangle \right|\\
&= \varepsilon\|T\|\|T(\xi)\|+\left|\sum_{j=1}^{k}\big\langle \rho_{1}(\alpha_{g^{-1}}(f_{g,j}))\rho_{1}(u_{g^{-1}})T(\eta_{j}),T(\xi)\big\rangle \right|\\
&\leq \varepsilon\|T\|\|T(\xi)\|+\sum_{j=1}^{k}\sup_{f\in L^{\infty}(Z):\|f\|_{\infty}\leq 1}\big|\big\langle \rho_{1}(f)\rho_{1}(u_{g^{-1}})T(\eta_{j}),T(\xi)\big\rangle \big|.
\end{align*}
Letting $g\to\infty$ we find that
\[\|T(\xi)\|^{2}\leq \varepsilon\|T\|\|T(\xi)\|.\]
Letting $\varepsilon\to 0$ proves that $T(\xi)=0$ and, as $\xi$ was arbitrary, that $T=0.$

(\ref{I:left regular}) Given $\xi\in L^{2}(Z),\eta\in \ell^{2}(\G)$ we define $\xi\otimes \eta\in L^{2}(Z\times \G)$ by $(\xi\otimes \eta)(z,g)=\xi(z)\eta(g).$ Let 
\[D=\Span\{\xi\otimes \delta_{g}:\xi\in L^{2}(Z),g\in \G\}.\]
By the fact that $\lambda_{Z}\big|_{L^{\infty}(Z)}$ is contractive and the density of $D,$ it is enough to show that for every $\xi,\eta\in D$ we have
\[\lim_{g\to\infty}\sup_{f\in L^{\infty}(Z):\|f\|_{\infty}\leq 1}\big|\big\langle \lambda_{Z}(f)\lambda_{Z}(u_{g})\xi,\eta\big\rangle \big|=0.\]
Let $\xi,\eta\in D$ and write $\xi=\sum_{h}\xi_{h}\otimes \delta_{h},\eta=\sum_{h}\eta_{h}\otimes \delta_{h}.$ Let $E$ be a finite subset of $\G$ so that $\xi_{h}=0$ and $\eta_{h}=0$ if $h\in \G \setminus E.$ It is then straightforward to see that if $g\in \G \setminus EE^{-1},$ then 
\[\big\langle \lambda_{Z}(f)\lambda_{Z}(u_{g})\xi,\eta\big\rangle =0\]
for any $f\in L^{\infty}(Z).$
This proves $(\ref{I:left regular}).$

(\ref{I:compact extension}): $L^{2}(Y|Z)$ is the set of $\xi\in L^{2}(Y)$ with $\E_{Z}(|\xi|^{2})\in L^{\infty}(Z)$. It is a Banach space under the norm
\[\|\xi\|_{L^{2}(Y|Z)}=\sqrt{\|\E_{Z}(|\xi|^{2})\|_{\infty}}.\]
See \cite[Chapter 3]{MR3616077} for more detail.

Given $f\in L^{\infty}(Y)$ and $\varepsilon>0,$ compactness of the extension $\G\cc (Y,\nu)\to \G\cc (Z,\zeta)$ implies the existence (see \cite[Definition 3.8]{MR3616077}) of vectors $\zeta_{1},\dots,\zeta_{k}\in L^{2}(Y|Z)$, a subset $Y_{0}\subseteq Y$ with $\nu(Y_{0})>1-\varepsilon$ and $Z_0 \subset Z$ with $\zeta(Z_0) > 1-\eps$ so that
\[\chi_{Z_0}\rho(\G)\left(\chi_{Y_{0}}f\right)\subseteq_{\varepsilon,\|\cdot\|_{L^{2}(Y|Z)}}\left\{\sum_{j=1}^{k}\rho(k_{j})\zeta_{j}:k_{j}\in L^{\infty}(Z),\|k_{j}\|_{\infty}\leq 1,j=1,\dots,k\right\}.\]
Since
\[\|\rho(u_{g})(f-\chi_{Y_{0}}f)\|_{2}=\|\chi_{Y\setminus Y_{0}}f\|_{2}\leq \sqrt{\varepsilon}\|f\|_{\infty}\]
and $\|\cdot\|_{2}\leq \|\cdot\|_{L^{2}(Y|Z)}$ we have
\[\rho(\G)f\subseteq_{\varepsilon+\sqrt{\varepsilon}\|f\|_{\infty},\|\cdot\|_{2}}\left\{\sum_{j=1}^{k}\rho(k_{j})\zeta_{j}:k_{j}\in L^{\infty}(Z),\|k_{j}\|_{\infty}\leq 1,j=1,\dots,k\right\}.\]
Since $\varepsilon>0$ is arbitrary and $L^{\infty}(Y)$ is dense in $L^{2}(Y),$ this proves part (\ref{I:compact extension}).

\end{proof}






\section{Combinations}\label{sec:combinations}



\subsection{Ergodic decomposition}\label{sec:ergodic decomposition}       

For any pmp action $\Ga \cc (X,\mu)$ there exists a map $x \mapsto \nu_x$ from $X$ to $\Prob_\Ga(X)$, the space of $\Ga$-invariant ergodic probability measures on $X$, such that
$$\mu  = \int \nu_x~d\mu(x).$$
See \cite{MR1784210} for example. If $\Ga$ is amenable then it is well-known that
$$h_\mu(\Ga \cc X) = \int h_{\nu_x}(\Ga \cc X)~d\mu(x).$$
This is known as the {\bf ergodic decomposition formula} \cite{ollagnier-book}. In general such a formula cannot hold for non-amenable $\Ga$ with an arbitrary sofic approximation. For example,  trivial actions can have entropy minus infinity (\S \ref{sec:trivial}). Nonetheless there are some positive results for uniformly diffuse sofic approximations and for the $f$-invariant as explained next.

\begin{defn}\label{defn:uniform-diffuse}
Let $\Sigma=\{\sigma_n\}_{n\in \N}$ be an arbitrary sofic approximation to $\Ga$. Also let $\{p_n\}_{n\in \N}$ be a sequence of positive integers with $\lim_{n\to\infty} p_n = +\infty$. For each $n$, define $\sigma_n^{\oplus p_n}:\Ga \to \sym( [d_n] \times [p_n])$ by 
$$\sigma_n^{\oplus p_n}(g)(j,k) = (\sigma_n(g)j,k).$$
Then $\Sigma':=\{\sigma_n^{\oplus p_n}\}_{n\in \N}$ is a {\bf uniformly diffuse} sofic approximation. This is stronger than being diffuse (which was considered in \S \ref{sec:diffuse}).
\end{defn}

\begin{exercise}
Let $\Ga \cc (X,\mu), \Sigma,\Sigma'$ be as above. Then
$$h_{\Sigma',\mu}(\Ga \cc X) = \int h_{\Sigma',\nu_x}(\Ga \cc X)~d\mu(x) = \int h_{\Sigma,\nu_x}(\Ga \cc X)~d\mu(x).$$
\end{exercise}

In \S \ref{sec:variants} below, a variant of sofic entropy, called average-local sofic entropy is introduced. It is almost immediate that it satisfies the ergodic decomposition formula. It also agrees with $h_{\Sigma',\mu}(\Ga \cc X)$ which proves the exercise.

The next result shows that the $f$-invariant satisfies the ergodic decomposition formula with a correction term.
\begin{thm}\cite[Theorem 1.4]{MR3540601}\label{thm:f-ergodic-decomp}
Let $\F \cc (X,\mu)$ be a pmp action of a rank $r$ free group. Assume $\F \cc (X,\mu)$ admits a finite-entropy generating partition. Then $\F \cc (X,\nu_x)$ admits a finite-entropy generating partition for $\mu$-a.e. $x$ and 
$$f_\mu(\F \cc X) = \int f_{\nu_x}(\F \cc X)~d\mu(x) - (r-1)H(\tau)$$
where $\tau \in \Prob(\Prob(X))$ is the law of $\nu_x$ (so $\tau = \int \d_{\nu_x}~d\mu(x).$)
\end{thm}

\begin{proof}[Proof sketch]
The statement is directly verified for Markov chains. The general case can be obtained from approximating by Markov chains.
\end{proof}

In work-in-progress, Alpeev and Seward have proven that Rokhlin entropy satisfies the ergodic decomposition formula.

\subsection{Direct products}\label{sec:direct products} 

It is well-known that if $\Ga$ is amenable then topological and measure entropy are additive under direct products: 
$$h_{\rm{top}}(T\times S) = h_{\rm{top}}(T) + h_{\rm{top}}(S)$$
$$h_{\mu\times \nu}(T\times S) = h_{\mu}(T) + h_{\nu}(S)$$
where  $\G \cc^T (X,\mu)$, $\G \cc^S (Y,\nu)$ and $T\times S$ is the action on $X\times Y$ defined by
$$(T\times S)^g(x,y)=(T^gx, S^gy)\quad \forall g\in \G, x \in X, y\in Y.$$
Here is a brief summary of this section: topological sofic entropy is also additive under direct products. However, measure sofic entropy is only subadditive and there are explicit counterexamples to additivity. The $f$-invariant is additive under the restriction that all actions involved have finite generating partitions. However, it is unknown whether Rokhlin entropy is additive. There are several variants of sofic entropy introduced by Tim Austin, one of which is additive under direct products. In fact, these variants ``explain'' how additivity can fail.

\subsubsection{Direct products and topological entropy}

\begin{thm}[Topological entropy product formula]
Let $\Ga$ be a sofic group with a sofic approximation $\Sigma$ and fix a nonprincipal ultrafilter $\cU$ on $\N$. For continuous actions $\Ga \cc X, \Ga\cc Y$ on compact metrizable spaces define $(\Si,\cU)$-entropy by
$$h_{\Sigma,\cU}(\Ga \cc X)  =  \sup_{\epsilon>0} \inf_{F \Subset \Ga} \inf_{\delta >0} \lim_{i\to\cU} |V_i|^{-1}\log \left(N_\epsilon( \Map(T,\rho,F,\delta,\sigma_i), \rho_\infty)\right)$$
for any continuous generating pseudo-metric $\rho$ on $X$. That is, the definition of sofic entropy is modified by replacing $ \limsup_{i\to\infty}$ with the ultralimit along $\cU$. Then  
$$h_{\Sigma,\cU}(\Ga \cc X \times Y) = h_{\Sigma,\cU}(\Ga \cc X) + h_{\Sigma,\cU}(\Ga \cc Y)$$
where, by convention, $+\infty + (-\infty) = -\infty$. 
\end{thm}

The proof is a straightforward exercise. The crux of the argument is that if $\phi:V \to X, \psi:V \to Y$ are good microstates then $\phi\times \psi:V \to X \times Y$ is also a good microstate.

\subsubsection{Subadditivity}

\begin{thm}[Subadditive product formula]
Let $\Ga$ be a sofic group with a sofic approximation $\Sigma$ and pmp actions $\Ga \cc (X,\mu), \Ga\cc (Y,\nu)$. Then  
$$h_{\Sigma,\mu\times \nu}(\Ga \cc X \times Y) \le h_{\Sigma,\mu}(\Ga \cc X) + h_{\Sigma,\nu}(\Ga \cc Y).$$
If one of these actions is a Bernoulli shift then equality holds.
\end{thm}

The proof is a direct exercise. The key observation is that if $\psi:V \to X\times Y$ is a good microstate for $\mu\times \nu$ then the projections $\psi_X:V\to X$ and $\psi_Y:V \to Y$ are good microstates for the marginals. The special case of actions with finite-entropy generating partitions was handled in \cite{bowen-jams-2010}. The general case can be proven similarly using the partition definition of sofic entropy (\S \ref{sec:partition}). 

\subsubsection{Direct products and the $f$-invariant}

It is a brief exercise to show that the $f$-invariant is additive under direct products in the following sense: if $\F_r \cc^T (X,\mu)$, $\F_r \cc^S (Y,\nu)$ are both pmp actions with finite-entropy generating partitions then
$$f_{\mu\times \nu}(T\times S) = f_\mu(T) + f_\nu(S).$$ 
\begin{question}
Does this formula extend to actions that do not have finite-entropy generating partitions? In this case the $f$-invariant is defined via a random sofic approximation as in \S \ref{sec:interpretation}. Examples below show there is reason to be cautious.
\end{question}

In the next two examples, we consider sofic entropy with respect to the random sofic approximation $\P$ defined in \S \ref{sec:interpretation}. Recall that $\P$-entropy extends the $f$-invariant to actions that need not admit finite generating partitions.

\begin{example}[The infinite entropy Bernoulli shift and the trivial action]
The $\P$-entropy of the Bernoulli shift $\F_r \cc ([0,1],\textrm{Leb})^{\F_r}$ is $+\infty$. The $\P$-entropy of the trivial action of $\F_r$ on $([0,1],\rm{Leb})$ is $-\infty$ (for example this follows from Theorem \ref{thm:f-finite}). The direct product of these two actions has $\P$-entropy $-\infty$. The reason is that the number of $\eps$-separated microstates $\phi:[n] \to [0,1]^{\F_r} \times [0,1]$ for this action is roughly bounded by $\exp(n/\eps)$. On the other hand, with high probability a random sofic approximation $\s:\F_r \to \sym(n)$ admits no microstates at all. 
\end{example}

\begin{example}[Countable direct products]
The $\P$-entropy is not additive under countable direct products. To see this recall that the trivial action $\F_2 \cc (\Z/2, u_2)$ has $f$-invariant $-\log(2)$ while the Bernoulli shift $\F_2 \cc (\Z/2,u_2)^{\F_2}$ has $f$-invariant $\log(2)$. So the direct product action
$$\F_2 \cc (\Z/2\times \Z/2^{\F_2}, u_2 \times u_2^{\F_2})$$
has $f$-invariant $0$. The infinite direct power of this action is measurably conjugate to the direct product of the trivial action $\F_r \cc ([0,1],\rm{Leb})$ with the Bernoulli shift $\F_r \cc ([0,1],\textrm{Leb})^{\F_r}$. By the previous example, the $\P$-entropy of this action is $-\infty$. 
\end{example}

\subsubsection{A counterexample to additivity}

\begin{thm}
There exists a sofic group $\Ga$ with a sofic approximation $\Lambda$ and a pmp action $\Ga \cc^T (X,\mu)$ such that 
$$h_{\Lambda,\mu\times \mu}(T\times T) = -\infty \ne 0= 2 h_{\Lambda,\mu}(T).$$
\end{thm}

\begin{proof}
Let $\F_2=\langle a,b \rangle$ be a rank 2 free group and $\Si=\{\s_i:\F_2 \to \sym(V_i)\}$ be a sofic approximation to $\F_2$ by expanders as in \S \ref{sec:expanders}. For each $i$, define $\s^{\oplus 2}_i:\F_2 \to \sym(V_i \times \{0,1\})$ by
$$\s^{\oplus 2}_i(g)(v,i) = (\s(g)v,i)$$
(as in Definition \ref{defn:uniform-diffuse}) and let $\Lambda=\Si^{(2)}=\{\s_i^{\oplus 2}\}$.

Let $T$ be the trivial action of $\F_2$ on $\{0,1\}$ (so $T^gx=x$ for all $g\in \F_2, x\in \{0,1\}$).  By Proposition \ref{prop:trivial}, $h_{\Si,\mu}(T)=-\infty$ where $\mu$ is the uniform probability measure. An exercise shows that $h_{\Si^{(2)},\mu}(T)=0$.  An argument similar to Proposition \ref{prop:trivial} shows $h_{\Si^{(2)},\mu\times \mu}(T\times T)=-\infty$. 



\end{proof}

\subsubsection{Variants of sofic entropy and direct products}\label{sec:variants}

In \cite{MR3542515} Tim Austin introduced two variants of sofic entropy, one of which is additive under direct products. Both variants replace the number of microstates with covering numbers of measures on model spaces. The idea to use measures on model spaces goes back to \cite{MR2794944} where yet another variant of sofic entropy was introduced. 

The starting point is to assume our action has the form $\G  \cc (\cX^\G, \mu)$ where $(\cX,d_\cX)$ is a compact metric space and $\mu$ is a shift-invariant measure. This special form does not lose generality: any pmp action has a topological model and therefore we can assume it has the form $\G \cc (\cX,\nu)$ where $(\cX,d_\cX)$ is as above. We can equivariantly embed $\cX$ into $\cX^\G$ via $x \mapsto (g^{-1}x)_{g\in \G}$. This pushes $\nu$ forward to a shift-invariant measure on $\cX^\G$. 

Let $V$ be a finite set. A point in the product space $\cX^V$ is represented as $\bfx=(x_v)_{v\in V}$. This space is endowed with the normalized Hamming metric
$$d^{(V)}_\cX(\bfx, \bfy) = |V|^{-1} \sum_{v\in V} d_\cX(x_v, y_v).$$
Given $\bfx\in \cX^V$ and a map $\s:\G \to \sym(V)$, the {\bf pullback name of $\bfx$ at $v$} is
$$\Pi^\s_v(\bfx):=(x_{\s(g)^{-1}v})_{g\in \G} \in \cX^\G.$$
This defines a map $\Pi^\s_v: \cX^V \to \cX^\G$. The {\bf empirical distribution} of $\bfx$ is
$$P^\s_{\bfx} = |V|^{-1} \sum_{v\in V} \d_{\Pi^\s_v(\bfx)} \in \Prob(\cX^\G).$$
If $\s$ is a homomorphism then $P^\s_{\bfx}$ is shift-invariant. In general if $\cO \subset \Prob(\cX^\G)$ is any weak* open neighborhood of the subspace $\Prob_\G(\cX^\G)$ of shift-invariant probability measures, $\Si=\{\s_n\}$ is a sofic approximation and $\bfx_n \in \cX^{V_n}$ then $P^{\s_n}_{\bfx_n} \in \cO$ for all sufficiently large $n$. 

If $\cO$ is any weak* open neighborhood of $\mu$ in $\Prob(\cX^\G)$, then we let $\Omega(\cO,\s_n)$ denote the set of all microstates $\bfx\in \cX^{V_n}$ such that $P^{\s_n}_\bfx \in \cO$. The metric space $(\Omega(\cO,\s_n), d^{(V_n)}_{\cX}) $ is a {\bf model space}. It plays a role here that is very similar to the role played by $\Map(\cdots)$ in \S \ref{sec:pseudometric}. 

In order to avoid taking limsups and liminfs, we will work with ultralimits. So let $\cU$ be a nonprincipal ultrafilter on $\N$. A sequence $\{\mu_n\}$ of probability measures on $\cX^{V_n}$ is said to converge to $\mu$ with respect to ($\Si,\cU$)
\begin{itemize}
\item {\bf locally on average} if $\lim_{n\to\cU} \int P^{\s_n}_{\bfx}~d\mu_n(\bfx) = \mu$,
\item {\bf locally} if for every weak* open neighborhood $\cO$ of $\mu$ it is the case that
$$\lim_{n\to\cU} |V_n|^{-1}\#\{v\in V_n:~(\Pi^{\s_n}_v)_*\mu_n \in \cO\}=1,$$
\item {\bf locally and empirically} if it converges locally and for every weak* open neighborhood $\cO$ of $\mu$
$$\lim_{n\to\cU} \mu_n(\Omega(\cO,\s_n))=1,$$
\item {\bf locally and doubly-empirically} if $\mu_n\times \mu_n$ locally and empirically converges to $\mu\times \mu$.
\end{itemize}
The notions above are listed in order of increasing strength. In \cite{MR3542515}, local and empirical convergence is called quenched convergence. The new terminology has been chosen to avoid a conflict with common statistical physics language.

Every notion of convergence above corresponds to a variant of sofic entropy in which the number of microstates in the usual formula for entropy is replaced with a covering number. To explain, suppose $(\cY,d_\cY)$ is a metric space and $\d>0$. Then the {\bf $\d$-covering number} of $(Y,d_\cY)$, denoted $\rm{cov}_\d(\cY,d_\cY)$, is the minimum cardinality of a subset $S \subset \cY$ whose $\d$-neighborhood is all of $\cY$. If $\mu$ is a probability measure on $\cY$ then $\rm{cov}_{\eps,\d}(\mu,d_\cY)$ is the infimum of $\rm{cov}_\d(\cZ,d_\cY\resto \cZ)$ over all subsets $\cZ \subset \cY$ with measure $\mu(\cZ) \ge 1-\eps$. 

The {\bf local-on-average} sofic entropy of $\G \cc (\cX^\G,\mu)$ (with respect to $(\Si,\cU)$) is defined by
$$h^{loc-avg}_{\Si,\cU}(\G\cc (\cX^\G,\mu)) = \sup \left\{ \sup_{\eps,\d>0} \lim_{n\to\cU} |V_n|^{-1} \log \rm{cov}_{\eps,\d}\left(\mu_n, d^{(V_n)}_\cX\right)\right\}$$
where the first supremum is over all sequences $\{\mu_n\}$ that locally-on-average converge to $\mu$ with respect to $(\Si,\cU)$. The definitions of {\bf local sofic entropy, local and empirical sofic entropy} and {\bf local and doubly-empirical sofic entropy} (denoted $h^{loc}_{\Si,\cU}(\cdot)$, $h^{le}_{\Si,\cU}(\cdot)$, $h^{lde}_{\Si,\cU}(\cdot)$) are similar. Local and empirical is  abbreviated to {\bf le-} and local and doubly-empirical to {\bf lde-}.

\begin{prop}
The four notions of entropy defined above are measure-conjugacy invariants.
\end{prop}

\begin{proof}[Remarks on the proof]
The fact that le- and lde-sofic entropy are measure-conjugacy invariants is proven in \cite{MR3542515}. The proof there generalizes to local-on-average and local sofic entropy. The main idea is that any factor map  gives rise to a sequence of ``almost Lipschitz'' maps between model spaces. These maps essentially push-forward a sequence $\{\mu_n\}$ converging to $\mu$ to a new sequence $\{\nu_n\}$ converging to $\nu$ in such a way that the type of convergence is preserved.
\end{proof}
Any sequence $\{\mu_n\}$ of probability measures $\mu_n$ on $\Omega(\cO_n,\s_n)$ converges locally-on-average to $\mu$ whenever $\{\cO_n\}$ is a sequence of weak* open sets that decrease to $\mu$. Therefore sofic entropy lower bounds $h^{loc-avg}_{\Si,\cU}(\cdot)$. On the other hand, the increasing strength of the notions of convergence imply
$$h^{loc-avg}_{\Si,\cU}(\cdot) \ge h^{loc}_{\Si,\cU}(\cdot)  \ge h^{le}_{\Si,\cU}(\cdot) \ge h^{lde}_{\Si,\cU}(\cdot).$$

\begin{prop} Let $\G \cc^T (X,\mu)$ be a pmp action.
\begin{itemize}
\item If $T$ is ergodic then $h_{\Si,\cU,\mu}(T) = h^{loc-avg}_{\Si,\cU, \mu}(T)$. 
\item If $T$ is ergodic then any sequence $\{\mu_n\}$ that locally converges to $\mu$ must also le-converge to $\mu$. So $h^{loc}_{\Si,\cU, \mu}(T) = h^{le}_{\Si,\cU, \mu}(T)$. 
\item If $T$ is weakly mixing then any sequence $\{\mu_n\}$ that locally converges to $\mu$ must also lde-converge to $\mu$. So $h^{loc}_{\Si,\cU, \mu}(T) = h^{le}_{\Si,\cU, \mu}(T)=h^{lde}_{\Si,\cU, \mu}(T)$. 
\end{itemize} 
\end{prop}

\begin{proof}
The first statement is proven in \cite{MR2794944} (under unnecessarily restrictive simplifying hypotheses). In fact the proof shows that if $\mu_n$ locally-on-average converges to $\mu$ then for any weak* open neighborhood $\cO$ of $\mu$, $\mu_n(\Omega(\cO,\s_n)) \to 1$ as $n\to \cU$. This also explains the second statement. The third statement follows from the second since weak mixing implies that $\mu\times \mu$ is ergodic and it can be shown directly that if $\mu_n \to \mu$ locally then $\mu_n\times \mu_n \to \mu\times \mu$ locally. The second and third statements are proven in \cite{MR3542515}. 
\end{proof}

\begin{thm}\cite[Theorems B and C]{MR3542515}
Let $\G \cc^T (\cX^\G,\mu), \G \cc^S (\cY^\G,\nu)$ be pmp actions. As above, let $\Si$ be a sofic approximation to $\G$ and $\cU$ a non-principal ultrafilter on $\N$. Then
$$h_{\Si,\cU,\mu\times\nu} (T\times S) \ge h^{lde}_{\Si,\cU,\mu} (T) + h_{\Si,\cU,\nu} (S)$$
$$h^{lde}_{\Si,\cU,\mu\times\nu} (T\times S) = h^{lde}_{\Si,\cU,\mu} (T) + h^{lde}_{\Si,\cU,\nu} (S).$$
\end{thm}

\begin{proof}[Remarks on the proof]
The statements of Theorems B and C  in \cite{MR3542515} differ from the above. Instead of using ultrafilters one quantifies over all subsequences. The proof of the version above can be derived from the proof in \cite{MR3542515} with only minimal changes. The main idea is that if $\{\mu_n\}$ is a sequence of measures on model spaces that locally and doubly-empirically converges to $\mu$ then for any weak* open neighborhood $\cN$ of $\mu\times \nu$ in $\Prob(\cX^\G \times \cY^\G)$ there is a weak* open neighborhood $\cO$ of $\nu$ in $\Prob(\cY^\G)$ such that 
$$\inf_{y \in \Omega(\cO, \s_n)} \mu_n\big(\{\bfx \in \cX^{V_n}:~(\bfx,\bfy) \in \Omega(\cN, \s_n)\}\big) \to 1$$
as $n\to\cU$. This is reminiscent of the well-known fact that if $\mu\times \mu$ is ergodic and $\nu$ is ergodic then $\mu\times \nu$ is ergodic. 
\end{proof}

Finally, in order to further justify the notion of lde-entropy we have:

\begin{thm}\cite[Theorem D]{MR3542515}
Given a pmp action $\G \cc^T (X,\mu)$, sofic approximation $\Si$ and a non-principal ultrafilter $\cU$ on $\N$, the {\bf power-stabilized $(\Si,\cU)$-entropy} is defined by
$$h_{\Si,\cU,\mu}^{ps}(T)  = \lim_{n\to\infty} \frac{1}{n} h_{\Si,\cU,\mu^{\times n}}(T^{\times n}) $$
where $\G \cc^{T^{\times n}}(X,\mu)^{\times n}$ is the $n$-fold Cartesian power of $T$. The limit exists by sub-additivity. Then
$$h_{\Si,\cU,\mu}^{ps}(T) \ge h_{\Si,\cU,\mu}^{lde}(T)$$
and equality holds if $\G \cc^T (X,\mu)$ admits a finite-entropy generating partition. 
\end{thm}

One further justification:
\begin{thm}\cite[Corollary D']{MR3542515}
Given a pmp action $\G \cc^T (\cX^\G,\mu)$ with a finite-entropy generating partition, sofic approximation $\Si$ and a non-principal ultrafilter $\cU$ on $\N$, 
$$h_{\Si,\cU,\mu}^{lde}(T) = \inf_S h_{\Si,\cU,\mu\times\nu}(T\times S)  - h_{\Si,\cU,\nu}(S) $$
where the infimum is over all pmp actions $\G \cc^S (\cY^\G,\nu)$. 
\end{thm}

\section{Ornstein Theory}\label{sec:ornstein theory}

In 1970, Donald Ornstein introduced a powerful set of tools, known collectively as the ``Ornstein machine'', for proving that a given automorphism is measurably conjugate to a Bernoulli shift \cite{ornstein-1970a, ornstein-1970b, ornstein-1970c}. This machine also unifies the proofs of the following major results: 
\begin{itemize}
\item (Sinai's Factor Theorem): every ergodic automorphism $T$ with positive entropy factors onto every Bernoulli shift $\Z \cc (K,\k)^\Z$ with $H(K,\k)\le h_\mu(T)$ \cite{sinai-weak};
\item (Krieger's Generator Theorem): every ergodic automorphism $T$ admits a generating partition $\cP$ with $|\cP| < 1+ \exp( h_\mu(T))$ \cite{krieger-generator}; 
\item (Ornstein's Isomorphism Theorem): Bernoulli shifts are classified up to measure-conjugacy by entropy.
\end{itemize}
See \cite{glasner-joinings-book, down-book, petersen-book, rudolph-book} for modern treatments. 


These results were generalized by Ornstein and Weiss \cite{OW80} to arbitrary countable amenable groups via quasi-tiling machinery. Alternatively, this generalization can be made via orbit-equivalence theory \cite{danilenko-2001, danilenko-2002}. 

In recent work, all three major results have been partially generalized to all countable groups. These generalizations are discussed next, followed by a section on the $\bd$-metric that plays a crucial role in Ornstein theory.



\subsection{The Isomorphism Theorem}\label{sec:isom}

Ornstein's Isomorphism Theorem has recently been extended to all countable groups. The final piece was put in by Seward in work that is still in progress.

\begin{thm}
Let $\G$ be a countably infinite group. Let $(K,\k), (L,\l)$ be two probability spaces with the same Shannon entropy. Then the corresponding Bernoulli shifts $\Ga \cc (K,\kappa)^\Ga, \Ga \cc (L,\lambda)^\Ga$ are measurably conjugate.
\end{thm}

\begin{proof}[Proof sketch]
Following Stepin, we say that a group $\Ga$ is {\bf Ornstein} if whenever $(K,\kappa), (L,\lambda)$ are any two probability spaces with the same Shannon entropy then the Bernoulli shifts $\Ga \cc (K,\kappa)^\Ga, \Ga \cc (L,\lambda)^\Ga$ are measurably conjugate. So our goal is to prove that all countably infinite groups are Ornstein.

Stepin showed that if $\Ga$ contains an Ornstein subgroup then $\Ga$ must be Ornstein itself \cite{stepin-1975}. This is because if $H\le \Ga$ is an Ornstein subgroup then an isomorphism from  $\Ga\cc (K,\kappa)^\Ga$ to $\Ga\cc (L,\lambda)^\Ga$ can be built out of an isomorphism from $H\cc (K,\kappa)^H$ to $H\cc (L,\lambda)^H$ coset-by-coset. In other words, it is a co-induction argument (the details are spelled out in \cite{bowen-ornstein-2011}). By Ornstein-Weiss \cite{OW80} all infinite amenable groups are Ornstein. So any group that contains an infinite amenable subgroup must be Ornstein. 

On the other hand, Ol'shankii proved the existence of countable non-amenable groups that contain no proper infinite subgroups \cite{olshankii-book}. Stepin's trick cannot be directly applied to such groups. Nonetheless there is a measurable version of Stepin's trick that works and is applied in \cite{bowen-ornstein-2012}. 

Suppose there is a non-trivial probability space $(M,\mu)$ such that both $(K,\kappa)$ and $(L,\lambda)$ factor onto $(M,\mu)$. Nontrivial means that $\mu(\{x\})<1$ for every $x\in M$. Then the Bernoulli shifts $\Ga \cc (K,\kappa)^\Ga$ and $\Ga \cc (L,\lambda)^\Ga$ both factor onto $\Ga \cc (M,\mu)^\Ga$. By Zimmer \cite{zimmer-1984}, there exists an ergodic automorphism $T$ of $(M,\mu)^\Ga$ whose orbits are contained in the $\Ga$-orbits. We lift $T$ to automorphisms $\tT_K, \tT_L$ of $(K,\kappa)^\Ga, (L,\lambda)^\Ga$ respectively. Using  Thouvenot's relative version of Ornstein's Isomorphism Theorem \cite{thouvenot-rel-isom-1975} we see that $\tT_K$ and $\tT_L$ are isomorphic via an isomorphism compatible with $T$. This isomorphism is used in a manner similar to Stepin's trick to build an isomorphism from $\Ga \cc (K,\kappa)^\Ga$ to $\Ga \cc (L,\lambda)^\Ga$ (see \cite{bowen-ornstein-2012} for details).

Next suppose that neither $(K,\k)$ nor $(L,\l)$ is a 2-atom space. For example, this means that for any $k_0,k_1 \in K$, $\k(\{k_0,k_1\})<1$.  Then an elementary argument shows the existence of a third probability space $(N,\nu)$ such that
$$H(K,\kappa)=H(L,\lambda)=H(N,\nu)$$
and $(K,\kappa)$ has a nontrivial common factor with $(N,\nu)$ which has a nontrivial common factor with $(L,\lambda)$. So the previous result shows that $\Ga \cc (K,\kappa)^\Ga$ is isomorphic to $\Ga \cc (L,\lambda)^\Ga$.

The final piece to the puzzle is to handle the case when $K$ is a 2-atom space. This is handled in work-in-progress by Brandon Seward. The main idea is to find a common factor between $\Ga \cc (K,\k)^\G$ and $\Ga \cc (L,\l)^\G$ for a specific choice of $(L,\l)$ and then apply arguments similar to the above. 
\end{proof}

By contrast, it is trivial to check that no finite group is Ornstein. So among countable groups, the Ornstein property characterizes the infinite groups.


\subsubsection{Non-Bernoulli factors of Bernoulli shifts}

In the special case $\G=\Z$, many natural actions are known to be isomorphic to Bernoulli shifts. This includes mixing Markov chains, inverse limits of Bernoulli shifts, factors of Bernoulli shifts, algebraic actions with completely positive entropy, hyperbolic toral automorphisms, the time 1 map of geodesic flow on a negatively curved surface of finite volume. All of these results were obtained using Ornstein theory. By contrast, when $\G=\F_2$ we gave in \S \ref{sec:markov-chains}  Example \ref{ex:ising} an example of a mixing Markov chain that is not Bernoulli and in \S \ref{cor:inverse limit} an example of an inverse limit of Bernoulli shifts that is not Bernoulli. Next we give an example, due to Popa-Sasyk of a factor of a Bernoulli shift that is non-Bernoulli. Moreover this factor is algebraic and has completely positive sofic entropy by the main result of \cite{kerr-cpe}.

\begin{thm}\label{thm:popa}\cite{popa-sasyk, popa-cohomology1, MR3543677}
Let $\Ga$ be a countably infinite group with an infinite normal subgroup $H$ such that $(\Ga,H)$ has relative property (T). Then $\Ga$ admits a Bernoulli action with a non-trivial factor that is non-Bernoulli. In fact, the factor action is not even orbit-equivalent to a Bernoulli shift.
\end{thm}

\begin{proof}[Proof outline]
The idea is to compute the cohomology of the actions taking values in the circle $\R/\Z$. To be precise, let $(X,\mu)$ be a standard measure space with an action $\G \cc X$. A {\bf 1-cocycle} is a map $c: \Ga \times X \to \R/\Z$ satisfying
$$c(g_2g_1,x) = c(g_2,g_1x)+ c(g_1,x)$$
for $g_1,g_2 \in \G$ and a.e. $x\in X$. A cocycle $c$ is a {\bf coboundary} if there is a map $\phi:X \to \R/\Z$ such that
$$c(g,x) = \phi(gx) - \phi(x).$$
Two cocycles are {\bf cohomologous} if their difference is a coboundary.

The first cohomology group is $H^1(\G \cc (X,\mu)) = Z(\G \cc (X,\mu))/B(\G \cc (X,\mu))$ where $Z(\cdot)$ is the additive group of 1-cocycles and $B(\cdot)$ is the additive group of coboundaries. 

In \cite{popa-sasyk}, it is shown that for any Bernoulli shift over $\G$, every cocycle is cohomologous to a homomorphism $\G \to \R/\Z$. Therefore the cohomology group is isomorphic to $\Hom(\G, \R/\Z)$. This has been vastly generalized by Popa's cocycle superrigidity Theorem \cite{popa-cocycle-superrigidity-2007}. 

Now let $K$ be a compact abelian group. $K$ is embedded into $K^\G$ by $k\mapsto$ the constant function $(g \mapsto k)$. This is a closed $\G$-invariant subgroup. So $\G \cc K^\G/K$ is an algebraic action (called an Ornstein-Weiss factor in \cite{gaboriau-seward-2015} where bounds on its entropy are obtained). In \cite{popa-cohomology1}, it is shown that the cohomology group of $\G \cc K^\G/K$ (with respect to Haar measure on $K^\G/K$) is isomorphic to $\Hom(\Ga, \R/\Z) \times \Hom(K,\R/\Z)$. So if $\Hom(\Ga, \R/\Z) \times \Hom(K,\R/\Z)$ is not isomorphic to $\Hom(\Ga, \R/\Z)$, then this action cannot be orbit-equivalent to a Bernoulli shift, let alone measurably conjugate to one.

The group $\Hom(\Ga,\R/\Z)$ is compact while $\Hom(K,\R/\Z)$ is countable. As $K$ varies over all compact abelian groups, $\Hom(K,\R/\Z)$ varies over all countable abelian groups. Since there are uncountably many countable abelian groups, there are uncountably factors of Bernoulli shifts over $\G$ that are not Bernoulli.

The proof in \cite{MR3543677} differs from the above. It assumes $\G$ is sofic and shows that the model spaces of Bernoulli shifts have an asymptotic connectivity property that the model spaces of $\G \cc K^\G/K$ lacks.
\end{proof}

\begin{question}
Does there exist a non-Bernoulli factor of a Bernoulli shift over the free group? 
\end{question}

\begin{question}
If a factor of a Bernoulli shift is orbit-equivalent to a Bernoulli shift, must it be Bernoulli? This is true if $\G$ is an ICC property (T) group by Popa's cocycle super-rigidity Theorem \cite{popa-cocycle-superrigidity-2007}. 
\end{question}

\begin{question}
Let $K$ be a finite set and $\Prob_\Z(K^\Z)$ denote the space of shift-invariant Borel probability measures on $K^\Z$ with the weak* topology. Let $0<c<\log |K|$ and consider the subset $X_c \subset \Prob_\Z(K^\Z)$ of all measures with entropy rate $\ge c$. By upper semi-continuity, $X_c$ is closed. In unpublished work, Dan Rudolph proved the subset $X'_c\subset X_c$ of all measures $\mu \in X_c$ such that the shift action $\Z \cc (K^\Z,\mu)$ is Bernoulli and has entropy $=c$ is a dense $G_\delta$ subset of $X_c$. Density is a consequence of Rokhlin's Lemma. The statement that $X'_c$ is $G_\delta$ can be derived from the fact that a process is Bernoulli if and only if it is finitely determined. It uses the full strength of Ornstein theory. Is there an analogous result for any or every non-amenable group?
\end{question}

\subsection{Krieger's Generator Theorem}

Seward generalized Krieger's Generator Theorem to all countable groups using Rokhlin entropy:

\begin{thm}\cite{seward-kreiger-1}\label{thm:krieger}
Let $\Ga \cc^T (X,\mu)$ be a countably infinite group acting ergodically, but not necessarily freely, by measure-preserving bijections on a non-atomic standard probability space $(X,\mu)$. If $\bp=(p_i)$ is any finite or countable probability vector with
$$h^{\rm{Rok}}(T) < H(\bp)=-\sum_i p_i \log p_i,$$
then there is a generating partition $\alpha = \{A_i\}$ with $\mu(A_i) = p_i$ for all $i$. 
\end{thm}
The proof works almost entirely within the pseudo-group of the orbit equivalence relation of the action. It also uses previous (very accessible) work of Seward \cite{seward-krieger-0} to show that there exists a finite generating partition of the action whenever it has finite Rokhlin entropy. There is also a relative version of Theorem \ref{thm:krieger} in \cite{seward-kreiger-1} and a non-ergodic version is being written \cite{alpeev-seward}.

\subsection{Sinai's Theorem}  





Seward recently generalized Sinai's Factor Theorem:
\begin{thm}\cite{seward-sinai}\label{thm:seward-sinai}
Let $\Ga$ be an arbitrary countable group and $\Ga \cc^\a (X,\mu)$ an ergodic essentially free pmp action. Let $(K,\kappa)$ be a probability space and suppose 
$$0<H(K,\kappa) \le h^{\rm{Rok}}(\alpha).$$
Then the action $\Ga \cc^\alpha (X,\mu)$ factors onto the Bernoulli shift $\Ga \cc (K,\kappa)^\Ga$.
\end{thm}

\begin{proof}[Remarks on the proof]
This is only a light sketch of the proof of this deep result in the special case in which $H(K,\kappa) < h^{\rm{Rok}}(\a)$. By \cite{seward-tucker-drob} there exists an essentially free action $\Ga \cc^\b (Y,\nu)$ with 
$$h^{\rm{Rok}}( \b) < h^{\rm{Rok}}( \a) - H(K,\kappa)$$
such that $\Ga \cc^\a (X,\mu)$ factors onto $\Ga \cc^\b (Y,\nu)$. Let $\cB_Y$ denote the Borel sigma-algebra of $Y$, which we consider to be a sub-sigma-algebra of $\cB_X$ via this factor map. 

There exists a Borel map $f: Y \to [0,1]$ and an aperiodic automorphism $T \in \Aut(Y,\nu)$ such that $f_*\nu$ is Lebesgue measure and each orbit of $T$ is the intersection of an orbit of $\Ga$ with the preimage of a point under $f$. Let $\tf:X \to [0,1]$ and $\tT \in \Aut(X,\mu)$ denote the lifts of $\tf$ and $\tT$ respectively. Note that every $\tT$-orbit is the intersection of an orbit of $\Ga$ with the preimage of a point under $\tf$. 

Given a countable partition $\xi$ of $X$ and $S \subset [0,1]$ let
$$\xi_S = \xi \resto^X \tf^{-1}(S)$$
denote its quasi-restriction (this is the smallest partition of $X$ containing all sets of the form $Z \cap \tf^{-1}(S)$ for $Z \in \xi$). If $\cF$ is a sigma-algebra of $X$, we define $\cF_S$ similarly. Also let
$$\cP_\xi = \cB_Y \vee \bigvee_{t\in [0,1]} \left( \sigma-\rm{alg}_\Ga(\xi_{[0,t)}) \resto^X \tilde{f}^{-1}([t,1])\right).$$
This is called the {\bf external past} of $\xi$. It is $\tT$-invariant and importantly
$$h^{\rm{Rok}}(\a \resto \sigma-\rm{alg}_\Ga(\xi) | \cB_Y) \le h_{\mu}(\tT, \xi| \cP_\xi)$$
where $\a\resto \sigma-\rm{alg}_\Ga(\xi)$ denotes the factor of the action $\a$ generated by the partition $\xi$. Therefore 
$$H(K,\kappa) \le h_{\mu}(\tT, \xi| \cP_\xi).$$
So we can apply the Relative Sinai Factor Theorem (due to Thouvenot \cite{thouvenot-rel-isom-1975}) to $\tT$ relative to $\cP_\xi$ to obtain a Bernoulli factor for $\tT$ that is independent of $\cP_\xi$. This Bernoulli factor can be ``put together'' to obtain a Bernoulli factor for the $\Ga$-action. 
\end{proof}

As spectacular as the result above is; it might not be the `best possible'. As explained in \S \ref{sec:OW} there exist actions with zero Rokhlin entropy that factor onto Bernoulli shifts. This leads us to the following:

\begin{question}
Suppose $\G \cc (X,\mu)$ is an ergodic pmp action and $(K,\kappa)$ is a probability space whose Shannon entropy lower bounds the naive entropy of $\G \cc (X,\mu)$. Then does $\G \cc (X,\mu)$ factor onto the Bernoulli shift $\G \cc (K,\kappa)^\G$? By Sinai's Theorem, if $\G$ is amenable then the answer is `yes'. When $\G$ is non-amenable then the naive entropy of $\G \cc (X,\mu)$ is in $\{0,\infty\}$. When the naive entropy is zero, the action cannot factor onto any Bernoulli shift since naive entropy is monotone under factor maps. 
\end{question}


\subsubsection{Spectral theory implications}

Given a measure space $(X,\mu)$, let $L_0^2(X,\mu) \subset L^2(X,\mu)$ denote the orthogonal complement of the constant functions. Given a pmp action $\G \cc^T (X,\mu)$ there is a corresponding homomorphism of $\G$ into the unitary group of $L^2_0(X,\mu)$ given by
$$\k:\G \to L^2_0(X,\mu),\quad \k_gf = f\circ g^{-1}.$$
This is called the {\bf Koopman representation}. It is well-known that the Koopman representation of a Bernoulli shift is isomorphic to the countable sum of left regular representations which means that it has {\bf countable Lebesgue spectrum}. So it follows from Sinai's Factor Theorem that the Koopman representation of a positive-entropy action of an amenable group $\Ga$ necessarily contains a subrepresentation isomorphic to a countable sum of left-regular representations. A more difficult result to obtain is that any action with completely positive entropy has countable Lebesgue spectrum. This was first proven for $\Ga=\Z$ using the Rokhlin-Sinai Theorem and then extended to all amenable groups by \cite{dooley-golodets-2002} using orbit-equivalence techniques. 

We now have versions of these results for arbitrary groups:

\begin{thm}\cite{seward-sinai}
Suppose $\Ga \cc (X,\mu)$ is an essentially free ergodic pmp action. Let $\cH \subset L^2_0(X,\mu)$ be a $\Ga$-invariant closed subspace and $\cF \subset \cB_X$ be the smallest $\Ga$-invariant sigma-algebra such that all functions in $\cH$ are $\cF$-measurable. If $\cH$ has no non-zero subrepresentation that embeds into the left regular representation $\Ga \cc \ell^2(\Ga)$, then the factor corresponding to $\cF$ has zero Rokhlin entropy. 
\end{thm}

The sofic entropy version of this theorem was obtained previously by Hayes \cite{MR3773269} with a completely different proof relying on von Neumann algebra machinery.

\begin{cor}\cite{seward-sinai}\label{cor:rokhlin-cpe-spectral}
Suppose $\Ga \cc (X,\mu)$ is an essentially free ergodic action with completely positive outer Rokhlin entropy as defined in \S \ref{sec:outer-rokhlin}. Then the Koopman representation $\Ga \cc L^2_0(X,\mu)$ is unitarily isomorphic to the countable sum of left regular representations. 
\end{cor}

The sofic entropy version of the above corollary is stated as Theorem \ref{thm:hayes-cpe} below.

\subsubsection{Markov chains factor onto Bernoulli shifts}

We do not know whether every mixing Markov chain over a free group has positive Rokhlin entropy. So Theorem \ref{thm:seward-sinai} cannot be applied. However, the existence of a Bernoulli factor for these systems can be obtained directly:

\begin{thm}\label{thm:markov-factor}
Every mixing Markov chain over a non-abelian free group factors onto a Bernoulli shift.
\end{thm}

\begin{proof}
Let $\F=\langle S \rangle$ be a non-abelian free group and $\bfX=(X_g)_{g\in \F}$ be a mixing Markov chain taking values in a finite or countable state space $K$. Let $\mu$ be the law of $\bfX$ so that $\F \cc (K^\F,\mu)$ is a pmp action. Without loss of generality, we assume that the law of $X_{e}$ is fully supported on $K$. 


Let  $a, b \in S$ be distinct elements. Let $\Lambda \le \F$ be the cyclic subgroup generated by $a$. The law of $(X_g)_{g\in \Lambda}$ conditioned on $X_{e}=k$ is non-atomic (for any $k \in K$). Denote this law by $\mu_k \in \Prob(K^\Lambda)$. So there exists a measure-space isomorphism 
$$\phi_k: (K^\Lambda,\mu_k) \to (\T,\lambda)$$
where the latter denotes the circle with Haar measure. Define $\phi:K^\Lambda \to \T$ by $\phi(x)=\phi_k(x)$ where $k = x_{e}$. 

Define processes $\bfY=(Y_g)_{g\in \F}$ and $\bfZ=(Z_g)_{g\in \F}$ by
$$Y_g = \phi( a^n \mapsto X_{ga^n} ), \quad Z_g = Y_g + Y_{gb}.$$
Since $\bfZ$ is a factor of $\bfX$, it suffices to show that it is iid (independent and identically distributed). Before doing that, let us consider some general properties of the variables $Y_g$ and $Z_g$.

The Markov property of $\bfX$ implies that if $U \subset \G$ is any set of left-coset representatives of $\La$ then $(Y_g)_{g\in U}$ is iid. The former condition means that the cosets $\{g\La\}_{g\in U}$ are pairwise disjoint. This is because for any $g \in U$, $(X_{ga^n})_{n\in \Z}$ is independent of $(X_{ha^n})_{h \in U\setminus \{g\}, n \in \Z}$ relative to $X_g$. Therefore, $Y_g$ is independent of $(Y_{h})_{h \in U\setminus \{g\}}$ relative to $X_g$. Since the law of $Y_g$ conditioned on $X_g$ is Haar measure, $Y_g$ is independent of $X_g$. So $Y_g$ is independent of $(Y_{h})_{h \in U\setminus \{g\}}$.

 In general, if $A,B,C$ are random variables satisfying: (1) $A,B$ are independent and (2) $B,C$ take values in the circle then $A$ and $B+C$ are independent.  It follows that if $U \subset \G$ and $g \in \G$ are such that 
$$g\La \cup gb\La \nsubseteq U\La \cup Ub \La$$
then $Z_g$ is independent of $(Z_h)_{h \in U}$. 

Now let $W \subset \F$ be a finite set such that the induced subgraph of $W$ in the Cayley graph $\Cay(\F,S)$ is connected. To finish the proof, it suffices to show that $(Z_g)_{g\in W}$ is iid. 


Let $f,g \in W$ be such that there exists $s \in S\cup S^{-1}$ with $fs=g$ and the induced subgraph of $U:=W \setminus \{g\}$ is connected. By induction, we may assume that the variables $(Z_g)_{g\in U}$ are iid. 

We claim that 
$$g\La \cup gb\La \nsubseteq U\La \cup Ub \La.$$
Indeed, if $g\La \cup gb\La \subseteq U\La \cup Ub \La$, then since the induced subgraph of $g\La \cup gb\La$ consists of two lines connected by the single edge $\{g,gb\}$ and the induced subgraph of $U \cup Ub$ is connected, it follows that $U \cup Ub \supset \{g,gb\}$. Since $g \notin U$, this implies $g \in Ub$ and $gb \in U$. Therefore $gb^{-1}$ and $gb \in U$. But this implies the induced subgraph of $W \setminus \{g\}$ is disconnected (since $gb^{-1}$ and $gb$ are in different components). 

This contradiction implies $g\La \cup gb\La \nsubseteq U\La \cup Ub \La$, which as explained previously, implies $Z_g$ is independent of $(Z_h)_{h \in U}$ as required. 

\end{proof}

\subsection{The $\bd$ metric}\label{sec:d-bar}

Let $(K,d_K)$ be a compact metric space. The metric $d_K$ induces a new metric, denoted $\bd$, on the space of $\Ga$-invariant probability measures on $K^\Ga$. Intuitively, this new metric measures how closely two measure $\mu,\nu$ can be ``joined''. More precisely, recall that a {\bf joining} between measures $\mu, \nu \in \Prob_\Ga(K^\Ga)$ is a $\Ga$-invariant Borel probability measures $\lambda$ on the product space $K^\Ga \times K^\Ga$ whose marginals are $\mu$ and $\nu$. The $\bd$-distance between $\mu$ and $\nu$ is
$$\bd(\mu,\nu) = \inf_\lambda \int d_K(x_e,y_e)~d\lambda(x,y)$$
where the infimum is over all joinings $\lambda$ of $\mu$ and $\nu$. If $K$ is finite then it is usually assumed that $d_K$ is the trivial metric $d_K(k_1,k_2)=1$ if $k_1 \ne k_2$. This metric plays a key role in most developments of Ornstein theory.

\begin{thm}\label{thm:dbar}
If $\Ga$ is amenable and $K$ is finite then:
\begin{enumerate}
\item entropy function $\mu \mapsto h_\mu(\Ga \cc K^\Ga)$ is continuous in the topology induced by $\bd$,
\item the set of Bernoulli measures in $\Prob_\Ga(K^\Ga)$ is $\bd$-closed.
\end{enumerate}
\end{thm}

\begin{proof}
The first statement is an exercise in \cite{rudolph-book} (for $\Ga=\Z$). The last statement is contained in \cite{shields-thouvenot-1975} (again for $\Ga=\Z$) although it is also follows from the characterization of Bernoulli shifts as finitely determined processes (see e.g. \cite{rudolph-book} for details). 
\end{proof}

Here we will show that both statements above fail for at least some non-amenable groups. This is interesting because the first statement is a key ingredient in Ornstein theory and the last is a consequence. The next theorem is due to Tim Austin. It improves on an earlier example due to myself and Brandon Seward.

\begin{thm}
If $\Ga$ contains a non-abelian free group then there exists a sequence $\{\mu_n\}_{n=1}^\infty$ of $\Ga$-invariant measures $\mu_n \in \Prob_\Ga( K^\Ga)$ for some finite $K$ such that
\begin{itemize}
\item $\Ga \cc (K^\Ga,\mu_n)$ is isomorphic to the Bernoulli shift   $\Ga \cc ( (\Z/2), u_2)^\Ga$ for all $n$
\item the sequence $\{\mu_n\}$ converges in the $\bd$-metric to a measure $\mu_\infty$ (as $n\to\infty$)  such that $\Ga \cc (K^\Ga,\mu_\infty)$ is isomorphic to the Bernoulli shift   $\Ga \cc ( (\Z/2 \times \Z/2), u_2 \times u_2)^\Ga$.
\end{itemize}
In particular, neither sofic entropy, the $f$-invariant nor Rokhlin entropy is continuous in the $\bd$-metric.
\end{thm}

\begin{proof}
Let $K=(\Z/2)^4$. Let $\beta_n:(\Z/2)^{\Ga} \to \{0,1\}$ be any sequence of measurable functions such that 
$$\lim_{n\to\infty} u_2^{\Ga}(\beta_n^{-1}(1)) = 0$$
and $u_2^{\Ga}(\beta_n^{-1}(1))>0$ for all $n$. 

Let $a,b \in \Ga$ generate a rank 2 free subgroup. Define factor maps $\alpha_n:(\Z/2)^{\Ga} \to K^{\Ga}$ by 
\begin{displaymath}
\alpha_n(x)_g=\left\{ \begin{array}{cc}
( x_g + x_{ga}, x_g + x_{gb}, 0,0) & \textrm{ if } \beta_n(g^{-1}x)=0 \\
( x_g + x_{ga}, x_g + x_{gb}, x_g,1) & \textrm{ if } \beta_n(g^{-1}x)=1
\end{array}\right.\end{displaymath}
 Let $\mu_n = (\alpha_n)_*u_2^{\Ga}$. To verify the conclusions, let $\pi:K^\Ga \to (\Z/2 \times \Z/2)^\Ga$ be the projection map onto the first two coordinates. Observe that $\pi \alpha_n$ is the Ornstein-Weiss factor map. Since this map is 2-1 (when restricted to the free subgroup generated by $a,b$) and $u_2^{\Ga}(\beta_n^{-1}(1))>0$ it follows that $\alpha_n$ is an isomorphism onto its image. 
 
Define $\alpha_\infty:(\Z/2)^{\Ga} \to K^{\Ga}$ by $\alpha_\infty(x)_g=( x_g + x_{ga}, x_g + x_{gb}, 0,0).$ The $\bd$-limit of $\mu_n$ is the measure $\mu_\infty:=\alpha_{\infty *} u_2^\Ga$. The Ornstein-Weiss example shows that $\Ga \cc (K^\Ga,\mu_\infty)$ is isomorphic to the Bernoulli shift over the base space with entropy $\log(4)$.  
 \end{proof}
 
 \begin{question}
 Is the $f$-invariant finitely observable in the sense of \cite{ornstein-weiss-finitely-observable}? The Theorem above suggests the answer may be `no'.
 \end{question}

\begin{thm}
Let $\Ga$ be an infinite property (T) group. Then there exists a finite set $K$ and a sequence $\{\mu_n\}$ of $\Ga$-invariant measures on $K^\Ga$ such that 
\begin{itemize}
\item $\Ga \cc (K^\Ga,\mu_n)$ is isomorphic to a Bernoulli shift for all $n$ 
\item the sequence $\{\mu_n\}_n$ converges in the $\bd$-topology to a measure $\mu_\infty$ (as $n\to\infty$) such that $\Ga \cc (K^\Ga,\mu_\infty)$ is isomorphic to a non-Bernoulli factor of a Bernoulli shift.
\end{itemize}
\end{thm}

\begin{proof}
This example is similar to the previous one. To be precise, let $S \subset \Ga$ be a finite generating set and $K=(\Z/2)^S \times \Z/2 \times \Z/2$. Let $\beta_n:(\Z/2)^{\Ga} \to \{0,1\}$ be any sequence of measurable functions such that 
$$\lim_{n\to\infty} u_2^{\Ga}(\beta_n^{-1}(1)) = 0$$
and $u_2^{\Ga}(\beta_n^{-1}(1))>0$ for all $n$. 

Define factor maps $\alpha_n:(\Z/2)^{\Ga} \to K^{\Ga}$ by 
\begin{displaymath}
\alpha_n(x)_g=\left\{ \begin{array}{cc}
(( x_g + x_{gs})_{s \in S}, 0,0) & \textrm{ if } \beta_n(g^{-1}x)=0 \\
(( x_g + x_{gs})_{s \in S}, x_g,1) & \textrm{ if } \beta_n(g^{-1}x)=1
\end{array}\right.\end{displaymath}
 Let $\mu_n = (\alpha_n)_*u_2^{\Ga}$.  To verify the conclusions, let $\pi:K^\Ga \to ((\Z/2)^S)^\Ga$ be the projection map. Observe that $\pi \alpha_n$ is the Ornstein-Weiss factor map. Since this map is 2-1 and $u_2^{\Ga}(\beta_n^{-1}(1))>0$,  it follows that $\alpha_n$ is an isomorphism onto its image. 
 
Define $\alpha_\infty:(\Z/2)^{\Ga} \to K^{\Ga}$ by $\alpha_\infty(x)_g=(( x_g + x_{gs})_{s \in S}, 0,0).$ The $\bd$-limit of $\mu_n$ is the measure $\mu_\infty:=\alpha_{\infty *} u_2^\Ga$ and $\Ga \cc (K^\Ga, \mu_\infty)$ is isomorphic to the Ornstein-Weiss factor $\Ga \cc (\Z/2)^\Ga/(\Z/2)$. 

By the proof of Theorem \ref{thm:popa}, the cohomology group of any Bernoulli shift over $\Ga$ with values in the circle $\R/\Z$ is $\Hom(\Ga,\R/\Z)$ while the cohomology group of the Ornstein-Weiss factor $\Ga \cc (\Z/2)^\Ga/(\Z/2)$ is $\Hom(\Ga,\R/\Z) \times \Z/2$. Since $\Ga$ has property (T), $\Hom(\Ga,\R/\Z)$ is finite. Therefore, these cohomology groups are non-isomorphic and so 
$\Ga \cc (K^\Ga, \mu_\infty)$ is not orbit-equivalent to a Bernoulli shift and so cannot be measurably conjugate to one.
 \end{proof}

\begin{remark}
Corollary \ref{cor:inverse limit}  shows that for any non-amenable group there exists an inverse limit of factors of Bernoulli shifts that has zero Rokhlin entropy. This inverse limit can be realized as a $\bd$-limit of factors of Bernoulli shifts (with $K=[0,1]$ for example). This provides another example of a non-Bernoulli $\bd$-limit of Bernoulli shifts although in this case, $K$ is infinite. Is it possible to modify this example using finite $K$?
\end{remark}

 \section{The variational principle}\label{sec:variational}     

\begin{thm}[The variational principle]\label{thm:variational}
Let $\G \cc X$ be a continuous action on a compact metrizable space. Let $\Sigma$ be a sofic approximation of $\G$. Then
$$h_\Sigma(\G \cc X) = \sup_\mu h_{\Sigma,\mu}(\G \cc X)$$
where the supremum is over all $\G$-invariant Borel probability measures $\mu$ on $X$. In particular, if there does not exist a $\G$-invariant probability measure on $X$ then $h_\Sigma(\G \cc X) =-\infty$.
\end{thm}

\begin{remark}
For $\Z$-actions, this theorem was obtained over several papers \cite{goodwyn-1969, goodwyn-1972, dinaburg-1970, goodman-1971}. The proof that now appears in most textbooks is due to Misiurewicz \cite{MR0444904}. A number of other variational principles in entropy theory are provided in \cite{down-book}. The case of general sofic groups is \cite[Theorem 6.1]{kerr-li-variational}. There are versions of this result for sofic pressure \cite{chung-pressure}, sofic groupoids \cite{MR3286052} and a local version (with respect to a finite open cover) in \cite{zhang-local-variational}.
\end{remark}

\begin{proof}[Proof sketch]
The inequality
 $$h_\Sigma(\G \cc X) \ge \sup_\mu h_{\Sigma,\mu}(\G \cc X)$$
 is immediate from the pseudo-metric definition  of sofic entropy (\S \ref{sec:pseudometric}). The opposite inequality is achieved in the following way. Fix a finite partition $\cP$ of $\Prob(X)$, the space of probability measures on $X$. Fix a scale $\eps>0$. Each microstate $\phi:V_n \to X$ has an empirical measure
 $$ P^{\s_n}_\phi = |V_n|^{-1} \sum_{v\in V_n} \delta_{\phi(v)}\in \Prob(X).$$
By pulling back the partition $\cP$, we obtain a finite partition $\tilde{\cP}$ on the space of topological microstates. It follows that the exponential growth rate of the maximum cardinality of an $\eps$-separated subset of topological microstates is approximated by the same growth rate only restricted to microstates  whose empirical measures lie in a fixed part of the partition. By refining this partition and taking Benjamini-Schramm limits we can build an invariant measure whose sofic entropy is bounded below by the exponential rate of growth of the maximum cardinality of an $\eps$-separated subset of the topological microstate space. Sending $\eps \searrow 0$ finishes the proof.
 \end{proof}
 
 \subsection{Measures of maximal entropy: existence}       
 
 The variational principle naturally leads to two problems: under what conditions does there exist a measure of maximal entropy and if one exists is it unique? If sofic entropy is upper semi-continuous as a function on $\Prob_\G(X)$ with respect to the weak* topology then compactness of $\Prob_\G(X)$ implies existence. Upper semi-continuity is discussed in \S \ref{sec:continuity2} and in greater depth in \cite{chung-zhang-expansive}. For example, it holds whenever $X$ is a closed $\G$-invariant subpace of $\cA^\G$ for some finite alphabet $\cA$.

\subsection{Measures of maximal entropy: uniqueness}\label{sec:uniqueness}       

It appears that there are no general results concerning uniqueness of measures of maximal entropy outside of $\G=\Z^d$. However, there is a counterexample to a natural conjecture in the case of $\G=\F_2$ and subshifts of finite type which we will go over next.

\begin{defn}[Subshifts of finite type]
Let $\G$ be a countable group, $K$ a finite set and $\G \cc K^\G$ the shift action: $(gx)(f)=x(g^{-1}f)$ for $g,f \in \G, x\in K^\G$. Let $F \Subset \G$ be a finite set and $\Omega \subset K^F$ be a collection of maps from $F$ to $K$. Let $X$ be the set of all $x\in K^\G$ such that for every $g\in \G$,  $gx$ restricted to $F$ is in $\Omega$. Then $X$ is the {\bf subshift of finite type} determined by $\Omega$.  Because subshifts of finite type are expansive, they admit measures of maximal entropy.
\end{defn}

It is a well-known fact that if $X$ is a subshift of finite type (over $\Z$) which is topologically transitive, then $X$ admits a {\em unique} measure of maximal entropy. Indeed more is true - for any continuous potential $\phi:X \to \R$, there is a unique equilibrium measure on $X$ \cite{bowen-equilibrium-states}. By contrast, Burton et al proved in \cite{MR1279469} that $\Z^2$ admits strongly irreducible subshifts of finite type with more than one measure of maximal entropy. See also \cite{MR1324466} for further constructions and general criteria for uniqueness in the $\Z^d$-setting.

\begin{thm}
There exists a topologically transitive subshift of finite type over the free group $\F_2 =\langle a,b\rangle$ that admits more than one measure of maximal $f$-invariant.
\end{thm}

\begin{proof}
Let $n\ge 4$ and define $X \subset [n]^{\F_2}$ by: an element $x\in X$ if and only if for every $g\in \F_2$ and $s \in \{a,b\}$ there is an $\epsilon \in \{0,1\}$ such that $x(gs)=x(g) + \epsilon \mod n$. This is a topologically transitive subshift of finite type.

Let $T=(T^g)_{g\in \F_2}$ denote the shift action $T^gx(f)=x(g^{-1}f)$ on $X$. 
Because the $f$-invariant is an infimum of continuous functions (namely, $f_\mu(T) = \inf_{W\Subset {\F_2}} F_\mu(T,\cP^W)$ where $\cP$ is the canonical partition on $[n]^{\F_2}$) it is upper semi-continuous in $\mu$ with respect to the weak* topology. Therefore, a measure of maximal $f$-invariant exists. Let $\mu$ be such a measure.

Let $\nu$ be the Markov approximation to $\mu$. To be precise, $\nu$ is the law of a Markov process $\bfX=(X_g)_{g\in {\F_2}}$ with values in $[n]$ satisfying the following: if $\bfY=(Y_g)_{g\in {\F_2}}$ is a stationary process with law $\mu$ then the law of the pair $(X_e,X_s)$ equals the law of $(Y_e,Y_s)$ for every $s\in \{a,b,a^{-1},b^{-1}\}$. By Theorem \ref{thm:markov},
$$f_\mu(T) \le F_\mu(T,\cP) = F_\nu(T,\cP) = f_\nu(T).$$
Moreover equality holds if and only if $\mu=\nu$. So $\mu$ is a Markov measure.

Let $A:[n]^{\F_2} \to [n]^{\F_2}$ be the +1 map: $A(x)_g\equiv x_g + 1 \mod n$. This map commutes with ${\F_2}$-action and preserves $X$. Therefore $A_*\mu$ is also a measure of maximal $f$-invariant. Let us assume to obtain a contradiction that $\mu$ is the unique measure of maximal $f$-invariant. Then $A_*\mu=\mu$ which implies the existence of a parameter $0\le \alpha_s \le 1/n$ (for $s\in \{a,b\}$) such that
$$\mu(\{x\in X:~x_e= i\}) = 1/n$$
$$\mu(\{x\in X:~x_e =i, x_s = i\}) = \alpha_s$$
$$\mu(\{x\in X:~x_e =i, x_s = i+1\}) = 1/n-\alpha_s.$$
Recall
\begin{eqnarray*}
f_\mu(T) = -3H_\mu(\cP) + \sum_{s\in S} H_\mu(\cP \vee T^s\cP).
\end{eqnarray*}
In our case, $H_\mu(\cP)=\log(n)$ and $H_\mu(\cP\vee T^s\cP)$ is uniquely maximized by $\alpha_s=\frac{1}{2n}$ in which case $H_\mu(\cP\vee T^s\cP)=\log(2n)$. 
Therefore,
\begin{eqnarray*}
f_\mu(T) = -3\log(n) + 2\log(2n) =\log(4) - \log(n).
\end{eqnarray*}
Because $n\ge 4$, this is non-positive. However, $X$ admits fixed points; for example the element $x_g=0 ~\forall g \in {\F_2}$. The Dirac measure concentrated on a fixed point is also an invariant probability measure and its $f$-invariant is zero. This contradiction proves the theorem.
\end{proof}

\begin{problem}
The example above is not completely satisfying because it uses the $f$-invariant instead of sofic entropy. It would be interesting to find a sofic example or prove that one does not exist. 
\end{problem}

\begin{remark}
The example above exploits the fact that $F_\mu(\cP)$ is not concave in $\mu$ and therefore $f_\mu(\cP)$ is also not concave. By contrast, Shannon entropy $H_\mu(\cP)$ is concave in $\mu$ (this is important in proving uniqueness of measures of maximal entropy for topologically transitive subshifts of finite type over the integers).
\end{remark}

\begin{question}
Christopher Hoffman constructed a subshift of finite type over $\Z^2$ that has a measure of maximal entropy that has completely positive entropy but is non-Bernoulli \cite{MR2773195}. Does an analogous example exist for the free group?
\end{question}







\section{Gibbs measures, pressure and equilibrium states}\label{sec:GIBBS}

\subsection{Motivation: finite graphs}\label{sec:Gibbs}

To motivate the notion of sofic pressure, we begin by recalling Gibbs measures in the setting of finite graphs. Let $\cX$ be a finite set of ``spins'' and $G=(V,E)$ a finite graph. A {\bf spin configuration} is a function $\phi: V \to \cX$. A {\bf vertex potential} is a function $\psi_{vert}: \cX \to \R$ and an {\bf edge potential} is a symmetric function $\psi_{edge}:\cX \times \cX \to \R$. The {\bf energy} of a spin configuration $\phi$ (with respect to $\psi_{vert}$ and $\psi_{edge}$)  is
$$E(\phi):= \sum_{v \in V} \psi_{vert}(\phi(v)) + \sum_{e=\{v,w\}\in E} \psi_{edge}(\phi(v),\phi(w)).$$
Consider the following problem: given $E_0 \in \R$, find a probability measure $\mu$ on $\cX^V$ such that the $\mu$-average energy is $E_0$:
$$\sum_\phi E(\phi)\mu(\{\phi\}) = E_0$$
and so that $\mu$ maximizes entropy over all measures satisfying the above. By way of Lagrange multipliers, it can be shown that the unique measure solving this problem has the form
$$\mu(\{\phi\}) = Z^{-1} \exp( -\beta E(\phi))$$
for some constants $Z, \beta$. Moreover, 
$$Z=Z(G) = \sum_\phi \exp( -\beta E(\phi))$$
is called the {\bf partition function}. It is central importance to understand the exponential growth rate of $Z(G)$ as the graph $G=(V,E)$ Benjamini-Schramm converges to a fixed graph of interest (for example, the Cayley graph of $\Z^d$ or $\F_r$). See \cite{MR2643563, MR2807681, MR1618769}. 


In the next two sections we define the sofic pressure of a given topological action $\G \cc X$ together with a potential function $\psi:X \to \R$ and sofic approximation $\Si$ as the exponential growth rate of an analogous partition function.

\subsection{Pressure}

Let $(X,\rho)$ be a compact metric space,  $\G \cc^T X$ an action by homeomorphisms, $\Psi:X \to \R$ a continuous function which we will call a {\bf potential function} and $\Si=\{\s_n: \G \to \sym(V_n)\}$ a sofic approximation to $\G$. Given $\s: \G \to \sym(V)$ and a microstate $\phi:V \to X$ define its energy by
$$E(\phi) = \sum_{v\in V} \Psi(\phi(v)).$$
Now suppose $\cZ \subset X^V$ is a collection of microstates. Then define the associated partition function by
$$Z(\cZ) =  \sum_{\phi \in \cZ} \exp(E(\phi)).$$
Given $F \Subset \G, \delta,\eps>0$, let
$$Z_\eps(\Psi, \G \cc X, \rho, F,\d, \s) = \sup Z(\cZ)$$
where the sup is over all $(\rho_\infty, \eps)$-separated subsets of $\Map(T,\rho,F,\delta,\s)$. The {\bf $\Si$-pressure} of $(\G \cc X,\Psi)$ is 
$$P_\Si(\G \cc X, \Psi):= \sup_{\eps>0} \inf_{F \Subset \G} \inf_{\delta>0} \limsup_{n\to\infty} |V_n|^{-1} \log Z_\eps(\Psi, \G \cc X, \rho, F,\d, \s_n).$$
This definition was introduced in \cite{chung-pressure} where it is also proven to be independent of the choice of metric $\rho$ (one can even use a generating pseudometric and $\rho_\infty$ can be replaced by $\rho_2$). For example, when $\Psi=0$ the pressure is the same as the sofic entropy. 

If $\G$ is amenable then there is another definition of pressure given in terms of a F\o lner sequence. In \cite{chung-pressure}, Chung proves that this definition agrees with the above. There is also a variational principle for pressure generalizing the one for entropy:

\begin{thm}[The variational principle]\cite{chung-pressure}
Let $\G \cc^T X$ be a continuous action on a compact metrizable space and $\Psi: X \to \R$ be continuous. Let $\Sigma$ be a sofic approximation of $\G$. Then
$$P_\Sigma(T,\Psi) = \sup_\mu h_{\Sigma,\mu}(T) + \int \Psi~d\mu$$
where the supremum is over all $T(\G)$-invariant Borel probability measures on $X$.
\end{thm}

\subsection{Pressure in symbolic systems}\label{sec:pressure-shift}

Here we specialize to the following set-up: let $\cX$ be a finite set and $\G \cc \cX^\G$ be the action by shifting: $(gx)(f)=x(g^{-1}f)$. Let $\Psi:\cX^\G \to \R$ be a continuous potential function. In this case, the sofic pressure admits a more intuitive formulation. Given $\s: \G \to \sym(V)$, $v\in V$ and $\phi:V \to \cX$, let $\Pi^\s_v(\phi) \in \cX^\G$ be the {\bf pullback name} of $\phi$:
$$\Pi^\s_v(\phi)_g = \phi( \s(g)^{-1}v).$$
The {\bf energy} of $\phi$ is
$$E(\phi) = \sum_{v\in V} \Psi(\Pi^{\s}_v(\phi))$$
and the partition function associated to $\s,\Psi$ is
$$Z(\s,\Psi) = \sum_{\phi \in \cX^V} \exp(E(\phi)).$$
Finally, the {\bf $\Si$-pressure} of $(\G \cc \cX^\G,\Psi)$ is 
$$P_\Si(\G \cc \cX^\G, \Psi):= \limsup_{n\to\infty} |V_n|^{-1} \log Z(\s_n,\Psi).$$
It is straightforward to show that this definition agrees with the previous definition.

\subsection{Equilibrium states}

A measure $\mu$ is called an {\bf equilibrium state} for $(\G\cc X,\Psi)$ if it realizes the supremum in the variational principle:
$$P_\Sigma(\G \cc X,\Psi) =  h_{\Sigma,\mu}(\G \cc X) + \int \Psi~d\mu.$$
For example, if the action $\G \cc X$ is expansive then entropy is upper semi-continuous in the measure $\mu$ and therefore there exists an equilibrium state. In the special case of $\G=\Z$, if $X$ is a subshift of finite type (over $\Z$) which is topologically transitive, then there is a unique equilibrium measure on $X$ \cite{bowen-equilibrium-states}. 

\begin{example}
Suppose $\cX$ is a finite set and $\Psi:\cX^\G \to \R$ depends only on the time 0 coordinate:
$$\Psi(x)=\Psi_0(x_{e})$$
for some function $\Psi_0:\cX \to \R$. Because $\G \cc \cX^\G$ is expansive, there exists an equilibrium state $\mu$  for the pair $(\G \cc \cX^\G,\Psi)$. Let $\kappa$ be the 1-dimensional marginal on $\cX$:
$$\kappa(\{k_0\}) := \mu(\{ x\in \cX^\G:~x_e=k_0\})\quad \forall k_0 \in \cX.$$
We claim that $\kappa^\G =\mu$.  Because $\Psi$ depends only on the time 0 coordinate, $\kappa^\G(\Psi) = \mu(\Psi)$. So it suffices to show that $\kappa^\G$ uniquely maximizes entropy over all invariant measures with the given $1$-dimensional marginal. This follows from Seward's Theorem \ref{prop:naive-Rokhlin1}. 

So we have shown that every equilibrium measure is a product measure. In fact, there is a unique equilibrium measure given by
$$\kappa(\{k_0\}) = \frac{\exp( \Psi_0(k_0))}{\sum_{k \in \cX} \exp( \Psi_0(k)) }.$$
This is proven in \cite{chung-pressure}. Alternatively, it follows from Lagrange multipliers. We have now answered Question 5.4 from \cite{chung-pressure} by demonstrating the uniqueness of the equilibrium measure.
\end{example}

\subsection{The Ising model}

The Ising model is a well-studied model of magnetism in statistical mechanics. In the notation of \S \ref{sec:Gibbs}, it amounts to choosing constants $B,\beta \in \R$ and setting $\cX=\{-1,+1\}$, $\Psi_{vert}(k)=Bk$, $\Psi_{edge}(k,l) = \beta kl$ so that for any spin configuration $\phi:V \to \{-1,+1\}$, 
$$E(\phi):= B \sum_{v \in V} \phi(v) + \beta \sum_{e=\{v,w\}\in E} \phi(v)\phi(w)$$
and
$$Z=Z(G) = \sum_{\phi:V \to \{-1,+1\} } \exp( -E(\phi))$$
(we have employed a small change of variables). If $G_n=(V_n,E_n)$ is a sequence of finite graphs then the {\bf asymptotic free energy density} (or free entropy or pressure) is the limit
$$\lim_{n\to\infty} |V_n|^{-1} \log Z(G_n)$$
if it exists. In the special case in which $\{G_n\}$ Benjamini-Schramm converges to a locally finite tree, an exact formula for the above limit is computed in \cite{MR2650042}. Here we will see that asymptotic free energy density can be re-interpreted as sofic pressure whenever the sequence $\{G_n\}$ arises from a sofic approximation.

So let $\G$ be a group, $\Si=\{\s_n: \G \to \sym(V_n)\}$ a sofic approximation to $\G$ and  $S \subset \G$ a finite symmetric generating set. Let $G_n=(V_n,E_n)$ be the graph with edges $(v, \sigma_n(s)v)$ for $s\in S, v\in V$. Because $\Si$ is a sofic approximation, $\{G_n\}$ Benjamini-Schramm converges to the Cayley graph of $(\G,S)$. 

Consider the potential function $\Psi:\{-1,1\}^\G \to \R$ 
$$\Psi(x) = B x(e) + \beta \sum_{s\in S} x(e)x(s).$$
By \S \ref{sec:pressure-shift},
\begin{eqnarray}\label{eqn:ising}
P_\Si(\G \cc \{-1,+1\}^\G, \Psi) = \limsup_{n\to\infty} n^{-1} \log Z(G_n).
\end{eqnarray}

In the special case in which $\G=\F_r$ is a rank $r>1$ free group and $S$ is a free generating set, it follows from the analysis in \cite{MR2650042} that the sofic pressure of $(\G \cc \{-1,+1\}^\G, \Psi)$ does not depend on the choice of sofic approximation and an explicit formula is known.

\subsection{Gibbs measures}

Let $\cX$ be a finite set. A potential function $\Psi:\cX^\G \to \R$ has {\bf finite range} if there is a finite subset $J \subset \G$ such that $\Psi(x)$ depends only on the restriction of $x$ to $J$. In this setting, a {\bf Gibbs measure} is any Borel probability measure $\mu$ on $\cX^\G$ satisfying the following. Let $Y=(Y_g)_{g\in \G}$ be a $\G$-indexed process with law $\mu$. Then $\mu$ is a Gibbs measure if for any fixed $x\in \cX^\G$ and finite $\La \Subset \G$, the law of $(Y_{g})_{g \in \La}$ conditioned on $Y_f = x(f)$ for all $f \in \G \setminus \La$ is given by: for any $z\in \cX^{\G}$ with $z(g)=x(g) ~\forall g\in \G \setminus \La$,
$$\mu( Y=z| Y_g = x(g) ~\forall g\in \G \setminus \La) = Z^{-1} \exp \left(\sum_{g \in \La} \Psi(g^{-1} z)\right)$$
where
$$Z = \sum_{z\in \cX^\G: z(g)=x(g)~\forall g\notin \La} \exp \left(\sum_{g \in \La} \Psi(g^{-1} z)\right).$$

It is well-known, in the case of $\G=\Z^d$ that every equilibrium measure for $(\G \cc \cX^\G,\Psi)$ is a Gibbs measure \cite{MR2807681}. It appears that this question has not been explored in the context of sofic entropy. However, Alpeev proved in \cite{MR3498184} that when $\G$ is sofic, Gibbs measures exist. Moreover, if there is a unique Gibbs measure for $(\G \cc \cX^\G,\beta\Psi)$ and all $\beta \in [0,1]$ then the modified sofic entropies for these measures do not depend on the choice of sofic approximation. Modified sofic entropy refers to the sofic entropy defined in \cite{MR3542515} via measures on model spaces.

\section{Relative entropy}\label{sec:relative}

Suppose $\Ga$ is an amenable group, $\Ga \cc (X,\mu)$ a pmp action,  $\cP$ a partition of $X$ and $\cF$ a $\Ga$-invariant sigma-sub-algebra of the sigma-algebra of Borel sets. Then the {\bf entropy rate of the process $\Ga\cc (X,\mu,\cP)$ relative to $\cF$} is 
$$h_\mu(\Ga \cc X, \cP | \cF)  = \lim_{n\to\infty} |F_n|^{-1} H_\mu(\cP^{F_n}|\cF)$$
where $\{F_n\}$ is a F\o lner sequence in $\Ga$. The relative entropy of the action with respect to $\cF$ is
$$h_\mu(\Ga \cc X| \cF) = \sup_\cP h_\mu(\Ga \cc X, \cP | \cF)$$
where the sup is over all finite partitions $\cP$ of $X$. This is a measure-conjugacy invariant in the sense that if $\Ga \cc (Y,\nu)$ is another pmp action and $\phi:X \to Y$ is a measure-conjugacy, then 
$$h_\mu(\Ga \cc X| \cF)  = h_\nu(\Ga \cc Y| \phi(\cF)).$$

\subsection{The Abramov-Rokhlin formula}\label{sec:abramov}

\begin{thm}[Abramov-Rokhlin formula]
If $\G$ is amenable and $\G \cc (Y,\nu)$ is a factor of $\G \cc (X,\mu)$ then
$$h_\mu(\Ga \cc X) = h_\nu(\G \cc Y) + h_\mu(\G \cc X| \cB_Y)$$
where $\cB_Y \subset \cB_X$ is  the pullback sigma-sub-algebra.
\end{thm}

\begin{proof}[Remarks on the proof]
In the case $\G=\Z$, this result was obtained in \cite{MR0140660}. The general amenable case, due to Ward-Zhang, makes heavy use of the Ornstein-Weiss quasitiling machinery \cite{MR1203977}. Another proof, due to Danilenko, uses orbit-equivalence theory \cite{danilenko-2001}. A new short proof appears in \cite[Section 9.7]{MR3616077}
\end{proof}

The Ornstein-Weiss example shows this formula does not extend to Rokhlin entropy. However, it does extend to the $f$-invariant. To describe this let $\F_r=\langle s_1,\ldots, s_r\rangle$ denote the rank $r$ free group. Suppose $\F_r \cc^T (X,\mu)$ is a pmp action, $\cF \subset \cB_X$ an $\F_r$-invariant sub-sigma-algebra and $\cP$ finite-entropy partitions of $X$. Define
$$F_\mu(T,\cP|\cF) := H_\mu(\cP|\cF) + \sum_{s\in S} \big(H_\mu(\cP \vee T^s\cP|\cF)  - 2H_\mu(\cP|\cF)\big),$$
$$f_\mu(T,\cP|\cF) := \inf_{W \Subset \Ga} F(T,\cP^W|\cF).$$
In \cite{bowen-entropy-2010a} it is shown if $\cP,\cP'$ are generating partitions then $f_\mu(T,\cP|\cF)=f_\mu(T,\cP'|\cF)$. So it makes sense to define 
$$f_\mu(T |\cF) =  f_\mu(T,\cP|\cF)$$
for any finite-entropy generating partition $\cP$.

\begin{thm}\cite{bowen-entropy-2010a}
With notation as above, if $\cQ$ is a finite-entropy partition of $X$ contained in the sigma-algebra $\bigvee_{g\in \F_r} T^g \cP=\s\textrm{-alg}_{\F_r}(\cP)$, then 
$$f_\mu(T,\cP) = f_\mu(T,\cQ) + f_\mu(T,\cP| \s\textrm{-alg}_{\F_r}(\cQ)).$$
So if $\cP$ is generating and $\F_r \cc^S (Y,\nu)$ is a Mackey realization of the sigma-algebra corresponding to $\s\textrm{-alg}_{\F_r}(\cQ)$ then
$$f_\mu(T) = f_\nu(S) + f_\mu(T| \s\textrm{-alg}_{\F_r}(\cQ)).$$
\end{thm}

\begin{proof}[Remarks on the proof]
The proof is obtained from the alternative formulation of the $f$-invariant in \S \ref{sec:other} and the classical Abramov-Rokhlin formula for $\Z$-actions.
\end{proof}

\begin{remark}
The Abramov-Rokhlin formula is used to prove the entropy formula for finite-to-1 factor maps (Theorem \ref{thm:f-finite})  and can be used to prove the ergodic decomposition formula (Theorem \ref{thm:f-ergodic-decomp}). It is also used in the proof of (special cases of) Yuzvinskii's formula \S \ref{sec:yuz}. The Ornstein-Weiss map gives an example where the relative $f$-invariant is negative.
\end{remark}

\begin{remark}
Ben Hayes has recently defined a notion of relative sofic entropy  \cite{hayes-relative-entropy}.
\end{remark}

\section{Outer/extension entropy}\label{sec:outer}


In classical entropy theory, one considers the entropy of an action with respect to a factor. In the new non-amenable theory, we also have to consider the entropy of a factor relative to the source! This gives a non-trivial concept that has been called {\bf extension entropy}, {\bf outer entropy} and {\bf entropy in the presence}. It is the exponential rate of growth of  the number of microstates of the target action that lift to microstates of the extension. The ideas originated in David Kerr's partition definition of sofic entropy (\S \ref{sec:partition}) and were developed in \cite{li-liang-mean-length, MR3635672, hayes-relative-entropy, seward-kreiger-2, seward-weak-containment}.

\subsection{Outer sofic entropy}

Let $\Ga \cc^T X$ and $\Ga \cc^S Y$ be continuous actions on compact metrizable spaces  and suppose $\Phi:X \to Y$ is a continuous factor map. The {\bf outer $\Si$-entropy} of $\Phi$ is defined by
$$h_\Si(\Phi) = \sup_{\epsilon>0} \inf_{F \Subset \G} \inf_{\delta>0} \limsup_{i\to\infty}  |V_i|^{-1}  \log \left(N_\epsilon( \Phi^{V_i}(\Map(T,\rho_X,F,\delta,\sigma_i), \rho_{Y,\infty})\right)$$
where $\rho_X,\rho_Y$ are generating continuous pseudo-metrics on $X$, $Y$, $\Map(\cdots)$ is as defined in \S \ref{sec:pseudo-top}, $N_\epsilon(\cdot, \rho_{Y,\infty}))$ denotes the maximum cardinality of an $\eps$-separated subset and $\Phi^{V_i}: X^{V_i} \to Y^{V_i}$ is the map 
$$\Phi^{V_i}(\bfx)_v = \Phi(x_v)$$
for $\bfx=(x_v)_{v\in V_i} \in X^{V_i}$. It can be shown that this definition does not depend on the choice of generating pseudo-metrics $\rho_X,\rho_Y$. 

If $\mu$ is a $\G$-invariant measure on $X$, $\nu$ is a $\G$-invariant measure on $Y$ and $\Phi_*\mu=\nu$ then define the {\bf outer $\Si$-measure-entropy} of $\Phi$ by
$$h_{\Si,\mu}(\Phi) = \sup_{\epsilon>0}\inf_\cO \inf_{F \Subset \G} \inf_{\delta>0} \limsup_{i\to\infty}  |V_i|^{-1}  \log \left(N_\epsilon\left( \Phi^{V_i}(\Map(T,\rho_X,\cO,F,\delta,\sigma_i), \rho_{Y,\infty})\right)\right)$$
where $\cO$ varies over all weak* open neighborhoods of $\mu $ in $\Prob(X)$. There is an equivalent definition based on partitions: 
\begin{eqnarray*}\label{eqn:outer}
h_{\Si,\mu}(\Phi) = \sup_\cQ \inf_\cP\inf_{1_\G \in F \Subset \Ga} \inf_{\delta>0}  \limsup_{i\to\infty} \frac{1}{|V_i|} \log |\Hom_\mu(\cP,F,\delta,\sigma_i)|_\cQ
\end{eqnarray*}
where the supremum is over all finite partitions $\cQ$ that are measurable with respect to $\Phi^{-1}(\cB_Y)$ and the infimum is over all finite partitions $\cP\ge \cQ$. 

This equivalence shows that outer sofic entropy is a measure-conjugacy invariant in the following sense. Suppose that $\G \cc^{T'} (X',\mu')$ and $\G \cc^{S'} (Y',\nu')$ are pmp actions, $\Phi': (X',\mu') \to (Y',\nu')$ is a factor map and there are measure-conjugacies $\tau_1:X \to X', \tau_2:Y\to Y'$ that make the diagram commute:
\begin{displaymath}\label{diagram2}
\xymatrix{  X  \ar[d]^{\Phi} \ar[r]^{\tau_1} & X' \ar[d]^{\Phi'} \\
  Y \ar[r]^{\tau_2} & Y' }
\end{displaymath}
then $h_{\Si,\mu}(\Phi) = h_{\Si,\mu'}(\Phi')$. For details justifying these claims see \cite[Theorem 1.20]{MR3635672}. There is also a relative notion of outer sofic entropy developed in \cite{hayes-relative-entropy}.

The definition implies the outer sofic entropy is bounded from above by the sofic entropies of both the target and source. If $\G$ is amenable then every microstate for the factor action $\G \cc Y$ lifts to a microstate for the source $\G \cc X$. For this reason, outer sofic entropy agrees with the classical entropy of the target whenever $\G$ is amenable \cite[Appendix A]{MR3635672}. If $\G$ is non-amenable then this no longer holds. Consider the Ornstein-Weiss map
$$\F_2 \cc (\Z/2,u_2)^{\F_2} \to^\Phi \F_2 \cc (\Z/2 \times \Z/2,u_2 \times u_2)^{\F_2}.$$
Because it is finite-to-1, it can be shown that its outer sofic entropy agrees with the sofic entropy of the {\em source}, which is $\log(2)$ instead of the entropy of the target which is $\log(4)$.

\begin{question}
Is there a variational principle connecting outer topological sofic entropy with outer measure sofic entropy?
\end{question}

As the Ornstein-Weiss example shows, sofic entropy is not necessarily monotone under factor maps. However, outer sofic entropy is: if $X \to^\Phi Y \to^\Psi Z$ are factor maps then 
$$h_\Si(\Phi) \ge h_\Si(\Psi \circ \Phi).$$
This is because any microstate for $\G \cc Z$ which lifts to a microstate for $\G \cc X$ via $\Psi\circ \Phi$ necessarily lifts to a microstate for $\G \cc Y$ via $\Psi$.

\subsubsection{Outer sofic Pinsker algebras}

 Let $\G \cc (X,\mu)$ be a pmp action and $\Si$ a sofic approximation to $\G$. The $\Si$-Pinsker algebra is the sigma-sub-algebra of $\cB_X$ generated by all factors with zero $\Si$-entropy. Because of sub-additivity, it has zero $\Si$-entropy itself. However, it does not have good monotonicity properties because there exist actions with zero $\Si$-entropy that factor onto actions with positive $\Si$-entropy (Theorem \ref{thm:variant}).   A better alternative is the outer $\Si$-Pinsker algebra.
 
  The {\bf outer $\Si$-Pinsker algebra} of the action $\G \cc (X,\mu)$ is the largest sigma-sub-algebra $\Pi^\Si(\mu) \subset \cB_X$ such that the corresponding factor has outer $\Si$-entropy zero. Because outer $\Si$-entropy is monotone, any invariant sub-sigma-algebra of $\Pi^\Si(\mu)$ also has zero outer $\Si$-entropy.
  
It is an important classical result that the Pinsker algebra of a direct product of transformations is the direct product of the Pinsker algebras. This follows from the Rokhlin-Sinai Theorem that the Pinsker algebra is the one-sided tail sigma-algebra of any generating partition. The case of general amenable groups is handled in \cite{MR1786718} using Sinai's Factor Theorem and joinings arguments. The case of sofic groups is new:

\begin{thm}\cite{hayes-relative-entropy}\label{thm:pinsker-sofic}
Suppose that $\G \cc (X,\mu), \G \cc (Y,\nu)$ are pmp actions, $\Si$ is a sofic approximation to $\G$ and there exist model measures $\{\mu_n\}, \{\nu_n\}$ that lde-converge to $\mu,\nu$ respectively in the sense of \S \ref{sec:variants}. Then $\Pi^\Si(\mu\times \nu) = \Pi^\Si(\mu) \vee \Pi^\Si(\nu).$
\end{thm}


\subsection{Outer Rokhlin entropy}\label{sec:outer-rokhlin}

Let $\G \cc^T (X,\mu), \G \cc^S (Y,\nu)$ be pmp actions and suppose there is a factor map $\Phi:X \to Y$.  The {\bf outer Rokhlin entropy} of the $\Phi$ is
$$h^{\rm{Rok}}(\Phi) = \inf H_\mu(\cP)$$
where the infimum is over all measurable partitions $\cP$ of $X$ such that $\Phi^{-1}(\cB_Y) \subset \s\textrm{-alg}_\G(\cP)$ and $\cB_Y$ is the Borel sigma-algebra of $Y$.

This outer Rokhlin entropy is bounded above by the Rokhlin entropy of the factor and by the Rokhlin entropy of the source. In the special case in which $\G$ is amenable, the outer Rokhlin entropy equals the Rokhlin entropy of the target because of monotonicity of entropy under factor maps. 

As in the case of outer sofic entropy, the Ornstein-Weiss factor map has outer Rokhlin entropy $\log(2)$.  The Ornstein-Weiss example shows Rokhlin entropy is not necessarily monotone under factor maps. However, outer Rokhlin entropy is: if $X \to^\Phi Y \to^\Psi Z$ are factor maps then 
$$h^{\rm{Rok}}(\Phi) \ge h^{\rm{Rok}}(\Psi \circ \Phi).$$
This is because any partition $\cP$ of $X$ satisfying $\Phi^{-1}(\cB_Y) \subset \s\textrm{-alg}_\G(\cP)$, also satisfies $ \Phi^{-1}\Psi^{-1}(\cB_Z)\subset \s\textrm{-alg}_\G(\cP)$.

 \subsubsection{Outer Rokhlin Pinsker algebras}\label{sec:outer-rokhlin-pinsker}
 
 
  The {\bf outer Rokhlin Pinsker algebra} of the action $\G \cc (X,\mu)$ is the largest sigma-sub-algebra $\Pi^{\rm{Rok}}(\mu) \subset \cB_X$ such that the corresponding factor has outer Rokhlin entropy zero. Because outer Rokhlin entropy is monotone, any invariant sub-sigma-algebra of $\Pi^{\rm{Rok}}(\mu)$ also has zero outer Rokhlin entropy. The action is said to have {\bf completely positive outer Rokhlin entropy} if $\Pi^{\rm{Rok}}(\mu)$ is trivial.
  

\begin{thm}\cite{seward-weak-containment}
If $\G \cc (X,\mu)$ and $\G \cc (Y,\nu)$ are essentially free pmp actions that are weakly contained in all essentially free pmp actions of $\G$ then 
$$\Pi^{\rm{Rok}}(\mu\times \nu) = \Pi^{\rm{Rok}}(\mu) \vee \Pi^{\rm{Rok}}(\nu).$$
\end{thm}

\begin{remark}
The notions of outer Rokhlin entropy and the Theorem above admit relative versions. See \cite{seward-weak-containment} for details.
\end{remark}

\subsection{Completely positive outer entropy}

An action has {\bf completely positive outer $\Si$-entropy}, denoted CPE$^\Si$, if every nontrivial factor has positive outer $\Si$-entropy. Because outer $\Si$-entropy is bounded from above by $\Si$-entropy, this implies that every nontrivial factor has positive $\Si$-entropy. 

It follows from the Rokhlin-Sinai Theorem that Bernoulli shifts over the integers have completely positive entropy. The case of amenable groups follows from the fact that Bernoulli shifts are uniformly mixing and uniform mixing implies CPE. This is the easy half part of \cite{rudolph-weiss-2000} which proves that CPE is equivalent to uniformly mixing. The main result of \cite{kerr-cpe}  is that if $\G$ is any sofic group, then every nontrivial factor of a Bernoulli shift over $\G$ has positive $\Si$-entropy (for every $\Si$). This is obtained via a positive lower bound on the local entropy that is uniform over all good enough sofic appoximations. The proof shows more: that Bernoulli shifts over $\G$ are CPE$^\Si$ for every $\Si$. Similarly, every nontrivial factor of a Bernoulli shift over a free group has positive $f$-invariant.

In \cite{hayes-relative-entropy}, Theorem \ref{thm:pinsker-sofic} is used to show that a large class of algebraic actions are CPE$^\Si$ (see \S \ref{sec:algebraic} for details).

\subsection{Uniformly mixing}       

\begin{defn}
A sequence $\{F_i\}_{i\in \N}$ of finite subsets of $\G$ is said to {\bf spread out} if for every $g\in \G \setminus \{1_\G\}$, $g\notin F_iF_i^{-1}$ for all but finitely many $i$. 
\end{defn}

\begin{defn} A pmp action $\G \cc (X,\mu)$ is {\bf uniformly mixing} if for every finite partition $\cP$ of $X$ and any sequence $F_i\subset \G$ of finite subsets that spread out,
$$\lim_{i\to\infty} \frac{1}{|F_i|} H_\mu\left(\bigvee_{g\in F_i} g^{-1}\cP \right) = H_\mu(\cP).$$
\end{defn}

The main result of \cite{rudolph-weiss-2000} is that if $\G$ is amenable group then CPE implies uniform mixing. The converse was proven earlier (see \cite{golodets-2000} or \cite[Theorem 4.2]{DGRS-2008}).  In \cite{burton-austin1} a new concept, called {\bf uniform model mixing}, that is adapted to a sofic approximation, is shown to imply completely positive le-$\Si$-entropy (le-$\Si$-entropy is defined in \S \ref{sec:variants}). This is used to prove that if $\G$ is sofic and contains an infinite cyclic subgroup, then there exist uncountably many completely positive le-$\Si$-entropy actions that are not factors of Bernoulli shifts. 

The goal of this section is to show that uniform mixing does not imply completely positive sofic entropy (Corollary \ref{cor:ising-non-cpe}). In fact, all mixing Markov chains over a free group are uniformly mixing (Theorem \ref{thm:uniform1}). However, with a spectral criterion due to Ben Hayes (Theorem \ref{thm:hayes-cpe}) we show that the Ising model with small transition probability does not have completely positive sofic entropy with respect to any sofic approximation.

\subsubsection{Markov chains}

\begin{thm}\label{thm:uniform1}
If $\G=\F_r$ is the free group  then all mixing Markov chains with finite state space over $\G$ are uniformly mixing.
\end{thm}

We need a lemma first. Recall that a {\bf tree} is a simply connected graph and a {\bf leaf} of a tree is a vertex with degree 1.

\begin{lem}
Let  $T$ denote a finite tree with at least 2 vertices. For each leaf $v$ of $T$, let $T_v \subset T$ denote the smallest subtree that contains every leaf of $T$ except for $v$. Suppose that for some $n\ge 0$ every pair of distinct leaves of $T$ are at a distance $\ge n$ apart (in the path metric). Then there exists a leaf $v$ such that the distance between $v$ and $T_v$ is at least $\lceil n/2\rceil$.
\end{lem}

\begin{proof}\footnote{This proof, which is shorter than my original, was gracefully provided by an anonymous reviewer.} Let $u,v$ be leaves of $T$ that are a maximum distance apart. Notice that $T_v$ is the union of the paths from $u$ to $w$, where $w$ ranges over all leaves of $T$ except for $v$. Fix a leaf $w \ne v$. Let $z$ be the point on the path from $w$ to $u$ that is closest to $v$. Then $z$ must also lie on the path from $v$ to $u$. Because $u$ and $v$ are at maximum distance apart, $d(v,u) \ge d(w,u)$. Since $d(v,u) = d(v,z)+d(z,u)$ and $d(w,u) = d(w,z) + d(z,u)$, $d(v,z) \ge d(w,z)$. Therefore, 
$$n \le d(v,w) =d(v,z) + d(z,w) \le 2d(v,z).$$
This holds for all leaves $w \ne v$ and therefore, $n \le 2d(v,T_v)$. 
  \end{proof}

\begin{proof}[Proof of Theorem \ref{thm:uniform1}]
Let $d(\cdot,\cdot)$ denote the word metric in $\G$ with respect to a free generating set. Let $\G \cc (K^\G,\mu)$ be a mixing Markov chain with state space $K$. Let $\cP$ be the time 0 partition: $\cP=\{P_k:~k \in K\}$ where 
$$P_k=\{x\in K^\G:~ x_e=k\}.$$

 Because $\cP$ is a Markov partition of a mixing Markov chain for every $\epsilon>0$ there exists $n=n(\epsilon) \in \N$ such that if $g,h\in \G$ satisfy $d(g,h)\ge n$ then $H_\mu(g^{-1}\cP|h^{-1}\cP) \ge H_\mu(\cP)-\epsilon$. 

Suppose $F \subset \G$ is a finite subset such that the word distance between any two distinct elements $g,h \in F$ is at least $2n$. Recall that $\cP^F=\bigvee_{f\in F} f^{-1}\cP$. By the previous lemma, there exists a $w\in F$ such that if $T$ denotes the smallest subtree containing $F-\{w\}$ then $d(T,w)\ge n$. Let $g\in T$ minimize the distance to $w$. The Markov property implies
\begin{eqnarray*}
H_\mu(\cP^{F}) &=& H_\mu(w^{-1} \cP| \cP^{F-\{w\}}) + H_\mu(\cP^{F-\{w\}}) \ge H_\mu(w^{-1} \cP| \cP^{T}) + H_\mu(\cP^{F-\{w\}})\\
  &=& H_\mu(w^{-1} \cP| g^{-1}\cP) + H_\mu(\cP^{F-\{w\}}) \ge H_\mu(\cP) + H_\mu(\cP^{F-\{w\}}) - \epsilon.
  \end{eqnarray*}
So by induction on $|F|$ we obtain, $H_\mu(\cP^F) \ge |F|( H_\mu(\cP)-\epsilon)$. This implies uniform mixing with respect to $\cP$.

Now let $\cQ$ be a partition that is contained in $\cP$.  Note that if $g,h \in \G$ satisfy $d(g,h)\ge n$ then $H_\mu(g^{-1}\cQ|h^{-1}\cP) \ge H_\mu(\cQ)-\epsilon$. Therefore, if $F,T,w,g$ are as above then
\begin{eqnarray*}
H_\mu(\cQ^{F}) &=& H_\mu(w^{-1} \cQ| \cQ^{F-\{w\}}) + H_\mu(\cQ^{F-\{w\}}) \ge H_\mu(w^{-1} \cQ| \cQ^{T}) + H_\mu(\cQ^{F-\{w\}})\\
  &=& H_\mu(w^{-1} \cQ| g^{-1}\cQ) + H_\mu(\cQ^{F-\{w\}}) \ge H_\mu(\cQ) + H_\mu(\cQ^{F-\{w\}}) - \epsilon.
  \end{eqnarray*}
  Uniform mixing with respect to $\cQ$ now follows from induction on $|F|$.

Observe that for any finite $K \subset \G$ with the property that the induced subgraph of the Cayley graph is connected, the partition $\cP^K$ is Markov. So the previous paragraphs imply: for any $\cQ\le \cP^K$, the action $\G\cc (X,\mu)$ is uniformly mixing with respect to $\cQ$. 

Now let $\cR$ be an arbitrary partition. Let $\delta>0$. Because the partition $\cP$ is generating, there exists a partition $\cQ\le \cP^K$ (for some finite $K \subset \G$) such that
$$d(\cQ,\cR) := H_\mu(\cQ|\cR) + H_\mu(\cR|\cQ) <\delta$$
where $d(\cdot,\cdot)$ denotes the Rokhlin distance on partitions (as defined in \S \ref{sec:continuity1}).

It follows that $d(\cQ^F, \cR^F) < \delta |F|$ for any finite $F$. Therefore if $\{F_i\}$ spreads out then because the action is uniformly mixing with respect to $\cQ$,
$$\liminf_i \frac{1}{|F_i|}H_\mu(\cR^{F_i}) \ge \liminf_i \frac{1}{|F_i|}H_\mu(\cQ^{F_i}) -\delta= H_\mu(\cQ)-\delta \ge H_\mu(\cR)-2\delta .$$
Because $\delta>0$ is arbitrary, this implies the theorem.
\end{proof}

\subsubsection{A spectral criterion for CPE}\label{sec:CPE-spectral}

\begin{thm}\cite[Corollary 1.4]{MR3773269}\label{thm:hayes-cpe}
Let $\G \cc (X,\mu)$ be ergodic and let $\kappa^0:\G \to U(L^2_0(X,\mu))$ be the Koopman representation on the orthogonal complement of the constants. If  $\G\cc (X,\mu)$ is CPE with respect to some sofic approximation then $\kappa^0$ embeds into the countable sum of the left regular representation of $\G$.  
\end{thm}

\begin{remark}
The Rokhlin entropy version of the above theorem is Corollary \ref{cor:rokhlin-cpe-spectral}.
\end{remark}

\begin{cor}\label{cor:laplace}
Let $\F_r=\langle s_1,\ldots, s_r\rangle$ denote the rank $r\ge 2$ free group.  Let $\Delta \in \C\G$ denote the ``discrete Lapacian''
$$\Delta = \frac{1}{2r} \sum_{i=1}^r  s_i + s_i^{-1}.$$
There is some $\eps_r>0$ such that if $\G \cc (X,\mu)$ is CPE with respect to some sofic approximation then
$$\|\kappa_0(\Delta)\| \le 1-\eps_r.$$
\end{cor}

\begin{proof}
The operator norm of $\Delta$, considered as an operator on $\ell^2(\G)$ is $1-\eps_r$ for some $\eps_r>0$ by Kesten's criterion (Theorem G.4.4 \cite{bekka2008kazhdan}). Therefore, the operator norm of $\Delta$, considered as an operator on $\ell^2(\G)^{\oplus \infty}$ is also $1-\eps_r$. So this corollary follows from Theorem \ref{thm:hayes-cpe}.
\end{proof}

\begin{cor}\label{cor:ising-non-cpe}
Let $\F_2=\langle s_1,\ldots, s_r\rangle$ denote the rank $r$ free group. If $\eps>0$ is sufficiently small then the Ising model (see Example \ref{ex:ising} \S \ref{sec:markov-chains}) is not CPE with respect to any sofic approximation.
\end{cor}

\begin{proof}
Let $\mu_\eps \in \Prob(\{0,1\}^{\F_r})$ denote the law of the Markov chain $\bfX=(X_g)_{g\in \F_r}$ with state space $\{-1,1\}$ satisfying
$$P(X_e = -1)=P(X_e=1)=1/2$$
$$P(X_e = k | X_s = k) = 1-\epsilon, \quad P(X_e \ne k | X_s = k) = \epsilon$$
for all $s\in S:=\{s_1,\ldots, s_r\} \cup \{s_1^{-1},\ldots, s_r^{-1}\}$. 
Then $\E[X_e]=0, \E[X_e^2] = 1$ and $\E[\Delta(X_e)X_e] = 1-2\eps$. This last computation requires some explanation: we consider $X_e \in L^2_0(\mu_\eps)$. Then 
$$\Delta(X_e)=|S|^{-1}\sum_{s\in S} X_e \circ \kappa_0(s) = |S|^{-1}\sum_{s\in S} X_s.$$
 Because $\langle X_e, X_s\rangle = (1-\eps) -\eps=1-2\eps$ for all $s\in S$, it follows that $\langle \Delta(X_e),X_e\rangle = \E[\Delta(X_e)X_e]=1-2\eps$. Thus $\|\kappa_0(\Delta)\| \ge 1-2\eps$. The previous corollary now implies this one.

\end{proof}

\bibliography{biblio}

\def\cprime{$'$} \def\cprime{$'$} \def\cprime{$'$}
  \def\cfudot#1{\ifmmode\setbox7\hbox{$\accent"5E#1$}\else
  \setbox7\hbox{\accent"5E#1}\penalty 10000\relax\fi\raise 1\ht7
  \hbox{\raise.1ex\hbox to 1\wd7{\hss.\hss}}\penalty 10000 \hskip-1\wd7\penalty
  10000\box7} \def\cprime{$'$} \def\cprime{$'$} \def\cprime{$'$}
  \def\cprime{$'$} \def\cprime{$'$} \def\cprime{$'$} \def\cprime{$'$}
\begin{thebibliography}{BdlHV08}

\bibitem[AB16]{burton-austin1}
Tim Austin and Peter Burton.
\newblock Uniform mixing and completely positive sofic entropy.
\newblock {\em submitted, arXiv:1603.09026}, 2016.

\bibitem[AKM65]{adler-konheim-mcandrew}
R.~L. Adler, A.~G. Konheim, and M.~H. McAndrew.
\newblock Topological entropy.
\newblock {\em Trans. Amer. Math. Soc.}, 114:309--319, 1965.

\bibitem[Alp15]{MR3498184}
A.~Alpeev.
\newblock The entropy of {G}ibbs measures on sofic groups.
\newblock {\em Zap. Nauchn. Sem. S.-Peterburg. Otdel. Mat. Inst. Steklov.
  (POMI)}, 436(Teoriya Predstavlenii, Dinamicheskie Sistemy, Kombinatornye
  Metody. XXV):34--48, 2015.

\bibitem[AR62]{MR0140660}
L.~M. Abramov and V.~A. Rohlin.
\newblock Entropy of a skew product of mappings with invariant measure.
\newblock {\em Vestnik Leningrad. Univ.}, 17(7):5--13, 1962.

\bibitem[AS16]{alpeev-seward}
Andrei Alpeev and Brandon Seward.
\newblock Krieger's finite generator theorem for ergodic actions of countable
  groups {III}.
\newblock {\em preprint}, 2016.

\bibitem[Aus16a]{MR3542515}
Tim Austin.
\newblock Additivity properties of sofic entropy and measures on model spaces.
\newblock {\em Forum Math. Sigma}, 4:e25, 79, 2016.

\bibitem[Aus16b]{MR3579704}
Tim Austin.
\newblock Behaviour of entropy under bounded and integrable orbit equivalence.
\newblock {\em Geom. Funct. Anal.}, 26(6):1483--1525, 2016.

\bibitem[Aus16c]{MR3543677}
Tim Austin.
\newblock The {G}eometry of {M}odel {S}paces for {P}robability-{P}reserving
  {A}ctions of {S}ofic {G}roups.
\newblock {\em Anal. Geom. Metr. Spaces}, 4:Art. 6, 2016.

\bibitem[AW13]{abert-weiss-2013}
Mikl{\'o}s Ab{\'e}rt and Benjamin Weiss.
\newblock Bernoulli actions are weakly contained in any free action.
\newblock {\em Ergodic Theory Dynam. Systems}, 33(2):323--333, 2013.

\bibitem[Ax68]{MR0229613}
James Ax.
\newblock The elementary theory of finite fields.
\newblock {\em Ann. of Math. (2)}, 88:239--271, 1968.

\bibitem[Bal05]{ball-factors1}
Karen Ball.
\newblock Factors of independent and identically distributed processes with
  non-amenable group actions.
\newblock {\em Ergodic Theory Dynam. Systems}, 25(3):711--730, 2005.

\bibitem[BdlHV08]{bekka2008kazhdan}
Bachir Bekka, Pierre de~la Harpe, and Alain Valette.
\newblock {\em Kazhdan's property (T)}.
\newblock Cambridge University Press, 2008.

\bibitem[BG14]{bowen-gutman-2014}
Lewis Bowen and Yonatan Gutman.
\newblock A {J}uzvinskii addition theorem for finitely generated free group
  actions.
\newblock {\em Ergodic Theory Dynam. Systems}, 34(1):95--109, 2014.

\bibitem[BK17]{bartholdi-kielak}
L.~Bartholdi and D.~Kielak.
\newblock Amenability of groups is characterized by {M}yhill?s theorem.
\newblock {\em J. Eur. Math. Soc. to appear. arXiv:1605.09133v2.}, 2017.

\bibitem[BL12]{MR2956925}
Lewis Bowen and Hanfeng Li.
\newblock Harmonic models and spanning forests of residually finite groups.
\newblock {\em J. Funct. Anal.}, 263(7):1769--1808, 2012.

\bibitem[BM97]{MR1446574}
Marc Burger and Shahar Mozes.
\newblock Finitely presented simple groups and products of trees.
\newblock {\em C. R. Acad. Sci. Paris S\'er. I Math.}, 324(7):747--752, 1997.

\bibitem[BM00]{MR1839489}
Marc Burger and Shahar Mozes.
\newblock Lattices in product of trees.
\newblock {\em Inst. Hautes \'Etudes Sci. Publ. Math.}, (92):151--194 (2001),
  2000.

\bibitem[Bow08]{bowen-equilibrium-states}
Rufus Bowen.
\newblock {\em Equilibrium states and the ergodic theory of {A}nosov
  diffeomorphisms}, volume 470 of {\em Lecture Notes in Mathematics}.
\newblock Springer-Verlag, Berlin, revised edition, 2008.
\newblock With a preface by David Ruelle, Edited by Jean-Ren{\'e} Chazottes.

\bibitem[Bow09]{MR2546620}
Lewis Bowen.
\newblock Free groups in lattices.
\newblock {\em Geom. Topol.}, 13(5):3021--3054, 2009.

\bibitem[Bow10a]{bowen-entropy-2010b}
Lewis Bowen.
\newblock The ergodic theory of free group actions: entropy and the
  {$f$}-invariant.
\newblock {\em Groups Geom. Dyn.}, 4(3):419--432, 2010.

\bibitem[Bow10b]{bowen-jams-2010}
Lewis Bowen.
\newblock Measure conjugacy invariants for actions of countable sofic groups.
\newblock {\em J. Amer. Math. Soc.}, 23(1):217--245, 2010.

\bibitem[Bow10c]{bowen-entropy-2010a}
Lewis Bowen.
\newblock Non-abelian free group actions: {M}arkov processes, the
  {A}bramov-{R}ohlin formula and {Y}uzvinskii's formula.
\newblock {\em Ergodic Theory Dynam. Systems}, 30(6):1629--1663, 2010.

\bibitem[Bow10d]{bowen-annals-2010}
Lewis~Phylip Bowen.
\newblock A measure-conjugacy invariant for free group actions.
\newblock {\em Ann. of Math. (2)}, 171(2):1387--1400, 2010.

\bibitem[Bow11a]{MR2794944}
Lewis Bowen.
\newblock Entropy for expansive algebraic actions of residually finite groups.
\newblock {\em Ergodic Theory Dynam. Systems}, 31(3):703--718, 2011.

\bibitem[Bow11b]{bowen-ornstein-2011}
Lewis Bowen.
\newblock Weak isomorphisms between {B}ernoulli shifts.
\newblock {\em Israel J. Math.}, 183:93--102, 2011.

\bibitem[Bow12a]{bowen-ornstein-2012}
Lewis Bowen.
\newblock Every countably infinite group is almost {O}rnstein.
\newblock In {\em Dynamical systems and group actions}, volume 567 of {\em
  Contemp. Math.}, pages 67--78. Amer. Math. Soc., Providence, RI, 2012.

\bibitem[Bow12b]{bowen-entropy-2012}
Lewis Bowen.
\newblock Sofic entropy and amenable groups.
\newblock {\em Ergodic Theory Dynam. Systems}, 32(2):427--466, 2012.

\bibitem[Bow14]{MR3286052}
Lewis Bowen.
\newblock Entropy theory for sofic groupoids {I}: {T}he foundations.
\newblock {\em J. Anal. Math.}, 124:149--233, 2014.

\bibitem[Bow16]{MR3530042}
Lewis Bowen.
\newblock Zero entropy is generic.
\newblock {\em Entropy}, 18(6):Paper No. 220, 20, 2016.

\bibitem[Bow17]{bowen-ri}
Lewis Bowen.
\newblock Finitary random interlacements and the {G}aboriau-{L}yons problem.
\newblock {\em submitted}, 2017.

\bibitem[BS94]{MR1279469}
Robert Burton and Jeffrey~E. Steif.
\newblock Non-uniqueness of measures of maximal entropy for subshifts of finite
  type.
\newblock {\em Ergodic Theory Dynam. Systems}, 14(2):213--235, 1994.

\bibitem[BS95]{MR1324466}
Robert Burton and Jeffrey~E. Steif.
\newblock New results on measures of maximal entropy.
\newblock {\em Israel J. Math.}, 89(1-3):275--300, 1995.

\bibitem[BS01]{benjamini-schramm-2001}
Itai Benjamini and Oded Schramm.
\newblock Recurrence of distributional limits of finite planar graphs.
\newblock {\em Electron. J. Probab.}, 6:no. 23, 13 pp. (electronic), 2001.

\bibitem[Bur17]{MR3649602}
Peter Burton.
\newblock Naive entropy of dynamical systems.
\newblock {\em Israel J. Math.}, 219(2):637--659, 2017.

\bibitem[CHR14]{CHR14}
Laura Ciobanu, Derek~F. Holt, and Sarah Rees.
\newblock Sofic groups: graph products and graphs of groups.
\newblock {\em Pacific J. Math.}, 271(1):53--64, 2014.

\bibitem[Chu13]{chung-pressure}
Nhan-Phu Chung.
\newblock Topological pressure and the variational principle for actions of
  sofic groups.
\newblock {\em Ergodic Theory Dynam. Systems}, 33(5):1363--1390, 2013.

\bibitem[CL15a]{capraro-lupini}
Valerio Capraro and Martino Lupini.
\newblock {\em Introduction to {S}ofic and hyperlinear groups and {C}onnes'
  embedding conjecture}, volume 2136 of {\em Lecture Notes in Mathematics}.
\newblock Springer, Cham, 2015.
\newblock With an appendix by Vladimir Pestov.

\bibitem[CL15b]{MR3314515}
Nhan-Phu Chung and Hanfeng Li.
\newblock Homoclinic groups, {IE} groups, and expansive algebraic actions.
\newblock {\em Invent. Math.}, 199(3):805--858, 2015.

\bibitem[Cor11]{cornulier-sofic-2011}
Yves Cornulier.
\newblock A sofic group away from amenable groups.
\newblock {\em Math. Ann.}, 350(2):269--275, 2011.

\bibitem[CZ15]{chung-zhang-expansive}
Nhan-Phu Chung and Guohua Zhang.
\newblock Weak expansiveness for actions of sofic groups.
\newblock {\em J. Funct. Anal.}, 268(11):3534--3565, 2015.

\bibitem[Dan01]{danilenko-2001}
Alexandre~I. Danilenko.
\newblock Entropy theory from the orbital point of view.
\newblock {\em Monatsh. Math.}, 134(2):121--141, 2001.

\bibitem[Den06]{deninger-2006}
Christopher Deninger.
\newblock Fuglede-{K}adison determinants and entropy for actions of discrete
  amenable groups.
\newblock {\em J. Amer. Math. Soc.}, 19(3):737--758 (electronic), 2006.

\bibitem[DG02]{dooley-golodets-2002}
A.~H. Dooley and V.~Ya. Golodets.
\newblock The spectrum of completely positive entropy actions of countable
  amenable groups.
\newblock {\em J. Funct. Anal.}, 196(1):1--18, 2002.

\bibitem[DGRS08]{DGRS-2008}
A.~H. Dooley, V.~Ya. Golodets, D.~J. Rudolph, and S.~D. Sinel{\cprime}shchikov.
\newblock Non-{B}ernoulli systems with completely positive entropy.
\newblock {\em Ergodic Theory Dynam. Systems}, 28(1):87--124, 2008.

\bibitem[Din70]{dinaburg-1970}
E.~I. Dinaburg.
\newblock A correlation between topological entropy and metric entropy.
\newblock {\em Dokl. Akad. Nauk SSSR}, 190:19--22, 1970.

\bibitem[DKP14]{dykema-2014}
Ken Dykema, David Kerr, and Mika{\"e}l Pichot.
\newblock Sofic dimension for discrete measured groupoids.
\newblock {\em Trans. Amer. Math. Soc.}, 366(2):707--748, 2014.

\bibitem[DM10a]{MR2643563}
Amir Dembo and Andrea Montanari.
\newblock Gibbs measures and phase transitions on sparse random graphs.
\newblock {\em Braz. J. Probab. Stat.}, 24(2):137--211, 2010.

\bibitem[DM10b]{MR2650042}
Amir Dembo and Andrea Montanari.
\newblock Ising models on locally tree-like graphs.
\newblock {\em Ann. Appl. Probab.}, 20(2):565--592, 2010.

\bibitem[Dow11]{down-book}
Tomasz Downarowicz.
\newblock {\em Entropy in dynamical systems}, volume~18 of {\em New
  Mathematical Monographs}.
\newblock Cambridge University Press, Cambridge, 2011.

\bibitem[DP02]{danilenko-2002}
Alexandre~I. Danilenko and Kyewon~K. Park.
\newblock Generators and {B}ernoullian factors for amenable actions and
  cocycles on their orbits.
\newblock {\em Ergodic Theory Dynam. Systems}, 22(6):1715--1745, 2002.

\bibitem[DS07]{deninger-schmidt}
Christopher Deninger and Klaus Schmidt.
\newblock Expansive algebraic actions of discrete residually finite amenable
  groups and their entropy.
\newblock {\em Ergodic Theory Dynam. Systems}, 27(3):769--786, 2007.

\bibitem[Dye59]{MR0131516}
H.~A. Dye.
\newblock On groups of measure preserving transformation. {I}.
\newblock {\em Amer. J. Math.}, 81:119--159, 1959.

\bibitem[Dye63]{MR0158048}
H.~A. Dye.
\newblock On groups of measure preserving transformations. {II}.
\newblock {\em Amer. J. Math.}, 85:551--576, 1963.

\bibitem[Eps08]{epstein-oe}
Inessa Epstein.
\newblock Orbit inequivalent actions of non-amenable groups.
\newblock {\em arXiv preprint arXiv:0707.4215}, 2008.

\bibitem[ES04]{ElekSzabo2004}
G{\'a}bor Elek and Endre Szab{\'o}.
\newblock Sofic groups and direct finiteness.
\newblock {\em J. Algebra}, 280(2):426--434, 2004.

\bibitem[ES05]{elek-szabo-2005}
G{\'a}bor Elek and Endre Szab{\'o}.
\newblock Hyperlinearity, essentially free actions and {$L^2$}-invariants.
  {T}he sofic property.
\newblock {\em Math. Ann.}, 332(2):421--441, 2005.

\bibitem[ES06]{elek-szabo-2006}
G{\'a}bor Elek and Endre Szab{\'o}.
\newblock On sofic groups.
\newblock {\em J. Group Theory}, 9(2):161--171, 2006.

\bibitem[ES11]{elek-szabo-2011}
G{\'a}bor Elek and Endre Szab{\'o}.
\newblock Sofic representations of amenable groups.
\newblock {\em Proc. Amer. Math. Soc.}, 139(12):4285--4291, 2011.

\bibitem[ES12]{MR2964622}
G{\'a}bor Elek and Bal{\'a}zs Szegedy.
\newblock A measure-theoretic approach to the theory of dense hypergraphs.
\newblock {\em Adv. Math.}, 231(3-4):1731--1772, 2012.

\bibitem[FO70]{friedman-ornstein}
N.~A. Friedman and D.~S. Ornstein.
\newblock On isomorphism of weak {B}ernoulli transformations.
\newblock {\em Advances in Math.}, 5:365--394 (1970), 1970.

\bibitem[FW04]{foreman-weiss}
Matthew Foreman and Benjamin Weiss.
\newblock An anti-classification theorem for ergodic measure preserving
  transformations.
\newblock {\em J. Eur. Math. Soc. (JEMS)}, 6(3):277--292, 2004.

\bibitem[Gab17]{MR3666024}
Damien Gaboriau.
\newblock Entropie sofique.
\newblock {\em Ast\'{e}risque}, (390):Exp. No. 1108, 101--138, 2017.
\newblock S\'{e}minaire Bourbaki. Vol. 2015/2016. Expos\'{e}s 1104--1119.

\bibitem[Geo11]{MR2807681}
Hans-Otto Georgii.
\newblock {\em Gibbs measures and phase transitions}, volume~9 of {\em de
  Gruyter Studies in Mathematics}.
\newblock Walter de Gruyter \& Co., Berlin, second edition, 2011.

\bibitem[GL09]{gaboriau-lyons}
Damien Gaboriau and Russell Lyons.
\newblock A measurable-group-theoretic solution to von {N}eumann's problem.
\newblock {\em Invent. Math.}, 177(3):533--540, 2009.

\bibitem[Gla03]{glasner-joinings-book}
Eli Glasner.
\newblock {\em Ergodic theory via joinings}, volume 101 of {\em Mathematical
  Surveys and Monographs}.
\newblock American Mathematical Society, Providence, RI, 2003.

\bibitem[Goo69]{goodwyn-1969}
L.~Wayne Goodwyn.
\newblock Topological entropy bounds measure-theoretic entropy.
\newblock {\em Proc. Amer. Math. Soc.}, 23:679--688, 1969.

\bibitem[Goo71]{goodman-1971}
T.~N.~T. Goodman.
\newblock Relating topological entropy and measure entropy.
\newblock {\em Bull. London Math. Soc.}, 3:176--180, 1971.

\bibitem[Goo72]{goodwyn-1972}
L.~Wayne Goodwyn.
\newblock Comparing topological entropy with measure-theoretic entropy.
\newblock {\em Amer. J. Math.}, 94:366--388, 1972.

\bibitem[Gro99]{gromov-1999}
M.~Gromov.
\newblock Endomorphisms of symbolic algebraic varieties.
\newblock {\em J. Eur. Math. Soc. (JEMS)}, 1(2):109--197, 1999.

\bibitem[GS00a]{golodets-2000}
Valentin~Ya. Golodets and Sergey~D. Sinel{\cprime}shchikov.
\newblock Complete positivity of entropy and non-{B}ernoullicity for
  transformation groups.
\newblock {\em Colloq. Math.}, 84/85(part 2):421--429, 2000.
\newblock Dedicated to the memory of Anzelm Iwanik.

\bibitem[GS00b]{MR1784210}
Gernot Greschonig and Klaus Schmidt.
\newblock Ergodic decomposition of quasi-invariant probability measures.
\newblock {\em Colloq. Math.}, 84/85(part 2):495--514, 2000.
\newblock Dedicated to the memory of Anzelm Iwanik.

\bibitem[GS15]{gaboriau-seward-2015}
Damien Gaboriau and Brandon Seward.
\newblock Cost, $\ell^2$-betti numbers, and the sofic entropy of some algebraic
  actions.
\newblock {\em To appear in Journal d'Analyse Math\'ematique}, 2015.

\bibitem[GTW00]{MR1786718}
E.~Glasner, J.-P. Thouvenot, and B.~Weiss.
\newblock Entropy theory without a past.
\newblock {\em Ergodic Theory Dynam. Systems}, 20(5):1355--1370, 2000.

\bibitem[Hay]{MR3635672}
Ben Hayes.
\newblock Mixing and spectral gap relative to {P}insker factors for sofic
  groups.
\newblock In {\em Proceedings of the 2014 {M}aui and 2015 {Q}inhuangdao
  conferences in honour of {V}aughan {F}. {R}. {J}ones' 60th birthday}.

\bibitem[Hay16a]{hayes-doubly-quenched}
Ben Hayes.
\newblock Doubly quenched convergence and the entropy of algebraic actions of
  sofic groups.
\newblock {\em submitted}, 2016.

\bibitem[Hay16b]{hayes-fk-determinants}
Ben Hayes.
\newblock Fuglede-{K}adison determinants and sofic entropy.
\newblock {\em Geom. Funct. Anal.}, 26(2):520--606, 2016.

\bibitem[Hay16c]{hayes-relative-entropy}
Ben Hayes.
\newblock Relative entropy and the {P}insker product formula for sofic groups.
\newblock {\em submitted}, 2016.

\bibitem[Hay17a]{MR3641841}
Ben Hayes.
\newblock Independence tuples and {D}eninger's problem.
\newblock {\em Groups Geom. Dyn.}, 11(1):245--289, 2017.

\bibitem[Hay17b]{MR3693125}
Ben Hayes.
\newblock Sofic entropy of {G}aussian actions.
\newblock {\em Ergodic Theory Dynam. Systems}, 37(7):2187--2222, 2017.

\bibitem[Hay18]{MR3773269}
Ben Hayes.
\newblock Polish models and sofic entropy.
\newblock {\em J. Inst. Math. Jussieu}, 17(2):241--275, 2018.

\bibitem[HLS14]{lovasz-local-global}
Hamed Hatami, L{\'a}szl{\'o} Lov{\'a}sz, and Bal{\'a}zs Szegedy.
\newblock Limits of locally-globally convergent graph sequences.
\newblock {\em Geom. Funct. Anal.}, 24(1):269--296, 2014.

\bibitem[Hof11]{MR2773195}
Christopher Hoffman.
\newblock Subshifts of finite type which have completely positive entropy.
\newblock {\em Discrete Contin. Dyn. Syst.}, 29(4):1497--1516, 2011.

\bibitem[HS16]{hayes-sale1}
Ben Hayes and Andrew Sale.
\newblock The wreath product of two sofic groups is sofic.
\newblock {\em submitted}, 2016.

\bibitem[HS18]{MR3795485}
Ben Hayes and Andrew~W. Sale.
\newblock Metric approximations of wreath products.
\newblock {\em Ann. Inst. Fourier (Grenoble)}, 68(1):423--455, 2018.

\bibitem[IKT09]{IKT}
Adrian Ioana, Alexander~S. Kechris, and Todor Tsankov.
\newblock Subequivalence relations and positive-definite functions.
\newblock {\em Groups Geom. Dyn.}, 3(4):579--625, 2009.

\bibitem[Juz65a]{MR0201560}
S.~A. Juzvinski{\u\i}.
\newblock Metric properties of automorphisms of locally compact commutative
  groups.
\newblock {\em Sibirsk. Mat. \v Z.}, 6:244--247, 1965.

\bibitem[Juz65b]{MR0194588}
S.~A. Juzvinski{\u\i}.
\newblock Metric properties of the endomorphisms of compact groups.
\newblock {\em Izv. Akad. Nauk SSSR Ser. Mat.}, 29:1295--1328, 1965.

\bibitem[Kec10]{Kechris-global-aspects}
Alexander~S. Kechris.
\newblock {\em Global aspects of ergodic group actions}, volume 160 of {\em
  Mathematical Surveys and Monographs}.
\newblock American Mathematical Society, Providence, RI, 2010.

\bibitem[Kel98]{MR1618769}
Gerhard Keller.
\newblock {\em Equilibrium states in ergodic theory}, volume~42 of {\em London
  Mathematical Society Student Texts}.
\newblock Cambridge University Press, Cambridge, 1998.

\bibitem[Ker13]{kerr-partitions}
David Kerr.
\newblock Sofic measure entropy via finite partitions.
\newblock {\em Groups Geom. Dyn.}, 7(3):617--632, 2013.

\bibitem[Ker14]{kerr-cpe}
David Kerr.
\newblock Bernoulli actions of sofic groups have completely positive entropy.
\newblock {\em Israel J. Math.}, 202(1):461--474, 2014.

\bibitem[Key70]{MR0274710}
Harvey~B. Keynes.
\newblock Lifting of topological entropy.
\newblock {\em Proc. Amer. Math. Soc.}, 24:440--445, 1970.

\bibitem[KH95]{katok-hasselblatt}
Anatole Katok and Boris Hasselblatt.
\newblock {\em Introduction to the modern theory of dynamical systems},
  volume~54 of {\em Encyclopedia of Mathematics and its Applications}.
\newblock Cambridge University Press, Cambridge, 1995.
\newblock With a supplementary chapter by Katok and Leonardo Mendoza.

\bibitem[Kid08]{MR2369194}
Yoshikata Kida.
\newblock Orbit equivalence rigidity for ergodic actions of the mapping class
  group.
\newblock {\em Geom. Dedicata}, 131:99--109, 2008.

\bibitem[Kie75]{kieffer-1975a}
J.~C. Kieffer.
\newblock A generalized {S}hannon-{M}c{M}illan theorem for the action of an
  amenable group on a probability space.
\newblock {\em Ann. Probability}, 3(6):1031--1037, 1975.

\bibitem[KL11a]{MR2813530}
David Kerr and Hanfeng Li.
\newblock Bernoulli actions and infinite entropy.
\newblock {\em Groups Geom. Dyn.}, 5(3):663--672, 2011.

\bibitem[KL11b]{kerr-li-variational}
David Kerr and Hanfeng Li.
\newblock Entropy and the variational principle for actions of sofic groups.
\newblock {\em Invent. Math.}, 186(3):501--558, 2011.

\bibitem[KL13a]{kerr-li-comb-ind}
David Kerr and Hanfeng Li.
\newblock Combinatorial independence and sofic entropy.
\newblock {\em Commun. Math. Stat.}, 1(2):213--257, 2013.

\bibitem[KL13b]{kerr-li-sofic-amenable}
David Kerr and Hanfeng Li.
\newblock Soficity, amenability, and dynamical entropy.
\newblock {\em Amer. J. Math.}, 135(3):721--761, 2013.

\bibitem[KL16]{MR3616077}
David Kerr and Hanfeng Li.
\newblock {\em Ergodic theory}.
\newblock Springer Monographs in Mathematics. Springer, Cham, 2016.
\newblock Independence and dichotomies.

\bibitem[Kol58]{kolmogorov-1958}
A.~N. Kolmogorov.
\newblock A new metric invariant of transient dynamical systems and
  automorphisms in {L}ebesgue spaces.
\newblock {\em Dokl. Akad. Nauk SSSR (N.S.)}, 119:861--864, 1958.

\bibitem[Kol59]{kolmogorov-1959}
A.~N. Kolmogorov.
\newblock Entropy per unit time as a metric invariant of automorphisms.
\newblock {\em Dokl. Akad. Nauk SSSR}, 124:754--755, 1959.

\bibitem[Kri70]{krieger-generator}
Wolfgang Krieger.
\newblock On entropy and generators of measure-preserving transformations.
\newblock {\em Trans. Amer. Math. Soc.}, 149:453--464, 1970.

\bibitem[Kun16]{kun-2016}
G{\'a}bor Kun.
\newblock On sofic approximations of property ({T}) groups.
\newblock {\em arXiv preprint arXiv:1606.04471}, 2016.

\bibitem[Li12]{Li-annals-2012}
Hanfeng Li.
\newblock Compact group automorphisms, addition formulas and
  {F}uglede-{K}adison determinants.
\newblock {\em Ann. of Math. (2)}, 176(1):303--347, 2012.

\bibitem[LL16]{li-liang-mean-length}
Hanfeng Li and Bingbing Liang.
\newblock Sofic mean length.
\newblock {\em preprint}, 2016.

\bibitem[LSW90]{LSW-1990}
Douglas Lind, Klaus Schmidt, and Tom Ward.
\newblock Mahler measure and entropy for commuting automorphisms of compact
  groups.
\newblock {\em Invent. Math.}, 101(3):593--629, 1990.

\bibitem[LT14]{MR3110799}
Hanfeng Li and Andreas Thom.
\newblock Entropy, determinants, and {$L^2$}-torsion.
\newblock {\em J. Amer. Math. Soc.}, 27(1):239--292, 2014.

\bibitem[Mal40]{malcev-1940}
A.~Malcev.
\newblock On isomorphic matrix representations of infinite groups.
\newblock {\em Rec. Math. [Mat. Sbornik] N.S.}, 8 (50):405--422, 1940.

\bibitem[MB09]{MR2533986}
Richard Miles and Michael Bj{\"o}rklund.
\newblock Entropy range problems and actions of locally normal groups.
\newblock {\em Discrete Contin. Dyn. Syst.}, 25(3):981--989, 2009.

\bibitem[Mey16]{MR3550271}
Tom Meyerovitch.
\newblock Positive sofic entropy implies finite stabilizer.
\newblock {\em Entropy}, 18(7):Paper No. 263, 14, 2016.

\bibitem[Mil08]{MR2457481}
Richard Miles.
\newblock The entropy of algebraic actions of countable torsion-free abelian
  groups.
\newblock {\em Fund. Math.}, 201(3):261--282, 2008.

\bibitem[Mis76]{MR0444904}
Michal Misiurewicz.
\newblock A short proof of the variational principle for a {${\bf Z}_{+}^{N}$}
  action on a compact space.
\newblock pages 147--157. Ast\'erisque, No. 40, 1976.

\bibitem[MO85]{ollagnier-book}
Jean Moulin~Ollagnier.
\newblock {\em Ergodic theory and statistical mechanics}, volume 1115 of {\em
  Lecture Notes in Mathematics}.
\newblock Springer-Verlag, Berlin, 1985.

\bibitem[Ol{\cprime}91]{olshankii-book}
A.~Yu. Ol{\cprime}shanski{\u\i}.
\newblock {\em Geometry of defining relations in groups}, volume~70 of {\em
  Mathematics and its Applications (Soviet Series)}.
\newblock Kluwer Academic Publishers Group, Dordrecht, 1991.
\newblock Translated from the 1989 Russian original by Yu. A. Bakhturin.

\bibitem[OP10]{MR2680430}
Narutaka Ozawa and Sorin Popa.
\newblock On a class of {${\rm II}_1$} factors with at most one {C}artan
  subalgebra.
\newblock {\em Ann. of Math. (2)}, 172(1):713--749, 2010.

\bibitem[Orn70a]{ornstein-1970a}
Donald Ornstein.
\newblock Bernoulli shifts with the same entropy are isomorphic.
\newblock {\em Advances in Math.}, 4:337--352 (1970), 1970.

\bibitem[Orn70b]{ornstein-1970c}
Donald Ornstein.
\newblock Factors of {B}ernoulli shifts are {B}ernoulli shifts.
\newblock {\em Advances in Math.}, 5:349--364 (1970), 1970.

\bibitem[Orn70c]{ornstein-1970b}
Donald Ornstein.
\newblock Two {B}ernoulli shifts with infinite entropy are isomorphic.
\newblock {\em Advances in Math.}, 5:339--348 (1970), 1970.

\bibitem[OW80]{OW80}
Donald~S. Ornstein and Benjamin Weiss.
\newblock Ergodic theory of amenable group actions. {I}. {T}he {R}ohlin lemma.
\newblock {\em Bull. Amer. Math. Soc. (N.S.)}, 2(1):161--164, 1980.

\bibitem[OW87]{OW87}
Donald~S. Ornstein and Benjamin Weiss.
\newblock Entropy and isomorphism theorems for actions of amenable groups.
\newblock {\em J. Analyse Math.}, 48:1--141, 1987.

\bibitem[OW07]{ornstein-weiss-finitely-observable}
Donald Ornstein and Benjamin Weiss.
\newblock Entropy is the only finitely observable invariant.
\newblock {\em J. Mod. Dyn.}, 1(1):93--105, 2007.

\bibitem[P{\u{a}}u11]{paunescu-2011}
Liviu P{\u{a}}unescu.
\newblock On sofic actions and equivalence relations.
\newblock {\em J. Funct. Anal.}, 261(9):2461--2485, 2011.

\bibitem[Pes08]{pestov-sofic-survey}
Vladimir~G. Pestov.
\newblock Hyperlinear and sofic groups: a brief guide.
\newblock {\em Bull. Symbolic Logic}, 14(4):449--480, 2008.

\bibitem[Pet89]{petersen-book}
Karl Petersen.
\newblock {\em Ergodic theory}, volume~2 of {\em Cambridge Studies in Advanced
  Mathematics}.
\newblock Cambridge University Press, Cambridge, 1989.
\newblock Corrected reprint of the 1983 original.

\bibitem[PK12]{pestov-kwiatkowska}
Vladimir~G Pestov and Alexsandra Kwiatkowska.
\newblock An introduction to hyperlinear and sofic groups.
\newblock {\em arXiv preprint arXiv:0911.4266}, 2012.

\bibitem[Pop06a]{popa-cohomology1}
Sorin Popa.
\newblock Some computations of 1-cohomology groups and construction of
  non-orbit-equivalent actions.
\newblock {\em J. Inst. Math. Jussieu}, 5(2):309--332, 2006.

\bibitem[Pop06b]{MR2231962}
Sorin Popa.
\newblock Strong rigidity of {$\rm II_1$} factors arising from malleable
  actions of {$w$}-rigid groups. {II}.
\newblock {\em Invent. Math.}, 165(2):409--451, 2006.

\bibitem[Pop07]{popa-cocycle-superrigidity-2007}
Sorin Popa.
\newblock Cocycle and orbit equivalence superrigidity for malleable actions of
  {$w$}-rigid groups.
\newblock {\em Invent. Math.}, 170(2):243--295, 2007.

\bibitem[Pop08]{popa-malleable}
Sorin Popa.
\newblock On the superrigidity of malleable actions with spectral gap.
\newblock {\em J. Amer. Math. Soc.}, 21(4):981--1000, 2008.

\bibitem[PS07]{popa-sasyk}
Sorin Popa and Roman Sasyk.
\newblock On the cohomology of {B}ernoulli actions.
\newblock {\em Ergodic Theory Dynam. Systems}, 27(1):241--251, 2007.

\bibitem[Roh67]{rohlin-lectures}
V.~A. Rohlin.
\newblock Lectures on the entropy theory of transformations with invariant
  measure.
\newblock {\em Uspehi Mat. Nauk}, 22(5 (137)):3--56, 1967.

\bibitem[Rud78]{MR508264}
Daniel~J. Rudolph.
\newblock If a finite extension of a {B}ernoulli shift has no finite rotation
  factors, it is {B}ernoulli.
\newblock {\em Israel J. Math.}, 30(3):193--206, 1978.

\bibitem[Rud90]{rudolph-book}
Daniel~J. Rudolph.
\newblock {\em Fundamentals of measurable dynamics}.
\newblock Oxford Science Publications. The Clarendon Press, Oxford University
  Press, New York, 1990.
\newblock Ergodic theory on Lebesgue spaces.

\bibitem[RW00]{rudolph-weiss-2000}
Daniel~J. Rudolph and Benjamin Weiss.
\newblock Entropy and mixing for amenable group actions.
\newblock {\em Ann. of Math. (2)}, 151(3):1119--1150, 2000.

\bibitem[Sch95]{schmidt-book}
Klaus Schmidt.
\newblock {\em Dynamical systems of algebraic origin}.
\newblock Modern Birkh\"auser Classics. Birkh\"auser/Springer Basel AG, Basel,
  1995.
\newblock [2011 reprint of the 1995 original] [MR1345152].

\bibitem[Sew14a]{seward-small-action}
Brandon Seward.
\newblock Every action of a nonamenable group is the factor of a small action.
\newblock {\em J. Mod. Dyn.}, 8(2):251--270, 2014.

\bibitem[Sew14b]{seward-kreiger-1}
Brandon Seward.
\newblock Krieger's finite generator theorem for ergodic actions of countable
  groups {I}.
\newblock {\em arXiv:1405.3604}, 2014.

\bibitem[Sew14c]{seward-subgroup}
Brandon Seward.
\newblock A subgroup formula for f-invariant entropy.
\newblock {\em Ergodic Theory Dynam. Systems}, 34(1):263--298, 2014.

\bibitem[Sew15a]{seward-krieger-0}
Brandon Seward.
\newblock Ergodic actions of countable groups and finite generating partitions.
\newblock {\em Groups Geom. Dyn.}, 9(3):793--810, 2015.

\bibitem[Sew15b]{seward-kreiger-2}
Brandon Seward.
\newblock Krieger's finite generator theorem for ergodic actions of countable
  groups {II}.
\newblock {\em arXiv:1501.03367v2}, 2015.

\bibitem[Sew16a]{MR3540601}
Brandon Seward.
\newblock Finite entropy actions of free groups, rigidity of stabilizers, and a
  {H}owe-{M}oore type phenomenon.
\newblock {\em J. Anal. Math.}, 129:309--340, 2016.

\bibitem[Sew16b]{seward-weak-containment}
Brandon Seward.
\newblock Weak containment and {R}okhlin entropy.
\newblock {\em arXiv:1602.06680}, 2016.

\bibitem[Sew18]{seward-sinai}
Brandon Seward.
\newblock Positive entropy actions of countable groups factor onto {B}ernoulli
  shifts.
\newblock {\em 1804.05269}, 2018.

\bibitem[Sin64]{sinai-weak}
Ja.~G. Sina{\u\i}.
\newblock On a weak isomorphism of transformations with invariant measure.
\newblock {\em Mat. Sb. (N.S.)}, 63 (105):23--42, 1964.

\bibitem[ST75]{shields-thouvenot-1975}
Paul Shields and J.-P. Thouvenot.
\newblock Entropy zero {$\times $} {B}ernoulli processes are closed in the
  {$\bar d$}-metric.
\newblock {\em Ann. Probability}, 3(4):732--736, 1975.

\bibitem[STD16]{seward-tucker-drob}
Brandon Seward and Robin~D. Tucker-Drob.
\newblock Borel structurability on the 2-shift of a countable group.
\newblock {\em Ann. Pure Appl. Logic}, 167(1):1--21, 2016.

\bibitem[Ste75]{stepin-1975}
A.~M. Stepin.
\newblock Bernoulli shifts on groups.
\newblock {\em Dokl. Akad. Nauk SSSR}, 223(2):300--302, 1975.

\bibitem[Tho71]{MR0293064}
R.~K. Thomas.
\newblock The addition theorem for the entropy of transformations of
  {$G$}-spaces.
\newblock {\em Trans. Amer. Math. Soc.}, 160:119--130, 1971.

\bibitem[Tho75]{thouvenot-rel-isom-1975}
Jean-Paul Thouvenot.
\newblock Quelques propri\'et\'es des syst\`emes dynamiques qui se
  d\'ecomposent en un produit de deux syst\`emes dont l'un est un sch\'ema de
  {B}ernoulli.
\newblock {\em Israel J. Math.}, 21(2-3):177--207, 1975.
\newblock Conference on Ergodic Theory and Topological Dynamics (Kibbutz, Lavi,
  1974).

\bibitem[Tho10]{thom-examples}
Andreas Thom.
\newblock Examples of hyperlinear groups without factorization property.
\newblock {\em Groups Geom. Dyn.}, 4(1):195--208, 2010.

\bibitem[Wei00]{weiss-2000}
Benjamin Weiss.
\newblock Sofic groups and dynamical systems.
\newblock {\em Sankhy\=a Ser. A}, 62(3):350--359, 2000.
\newblock Ergodic theory and harmonic analysis (Mumbai, 1999).

\bibitem[Wei03]{weiss-actions-of-amenable-groups}
Benjamin Weiss.
\newblock Actions of amenable groups.
\newblock In {\em Topics in dynamics and ergodic theory}, volume 310 of {\em
  London Math. Soc. Lecture Note Ser.}, pages 226--262. Cambridge Univ. Press,
  Cambridge, 2003.

\bibitem[Wei15]{MR3411529}
Benjamin Weiss.
\newblock Entropy and actions of sofic groups.
\newblock {\em Discrete Contin. Dyn. Syst. Ser. B}, 20(10):3375--3383, 2015.

\bibitem[WZ92]{MR1203977}
Thomas Ward and Qing Zhang.
\newblock The {A}bramov-{R}okhlin entropy addition formula for amenable group
  actions.
\newblock {\em Monatsh. Math.}, 114(3-4):317--329, 1992.

\bibitem[Zha12]{zhang-local-variational}
Guohua Zhang.
\newblock Local variational principle concerning entropy of sofic group action.
\newblock {\em J. Funct. Anal.}, 262(4):1954--1985, 2012.

\bibitem[Zim84]{zimmer-1984}
Robert~J. Zimmer.
\newblock {\em Ergodic theory and semisimple groups}, volume~81 of {\em
  Monographs in Mathematics}.
\newblock Birkh\"auser Verlag, Basel, 1984.

\end{thebibliography}
\bibliographystyle{alpha}

\end{document}